\documentclass[11pt,a4paper]{article}

\usepackage[utf8]{inputenc}
\usepackage[T1]{fontenc}
\usepackage[english]{babel}
\usepackage{amsmath}
\usepackage{amsfonts}
\usepackage{amsthm}
\usepackage{amssymb}
\usepackage{makeidx}
\usepackage{graphicx}
\usepackage{lmodern}
\usepackage[left=2cm,right=2cm,top=2cm,bottom=2cm]{geometry}

\usepackage{color}
\usepackage{comment}
\usepackage[dvipsnames]{xcolor}
\usepackage{graphicx}
\usepackage{framed}
\usepackage{tikz}
\usepackage{calrsfs}
\DeclareMathAlphabet{\pazocal}{OMS}{zplm}{m}{n}
\usepackage{bm}

\usepackage{fourier-orns}
\usepackage{mathtools}
\usepackage{relsize}
\usepackage{dsfont}
\usepackage{comment}
\usepackage{hyperref}

\usepackage{hyperref}
\hypersetup{
	linktocpage = true,
	colorlinks = true,
	linkcolor = {RoyalBlue},
	urlcolor = {RoyalBlue},
	citecolor = {Green}}
	
\newcommand{\B}{\mathbb{B}}
\newcommand{\R}{\mathbb{R}}

\newcommand{\Lbb}{\mathbb{L}}

\newcommand{\V}{\pazocal{V}}
\newcommand{\Apazo}{\pazocal{A}}
\newcommand{\Mpazo}{\pazocal{M}}

\newcommand{\Cpazo}{\pazocal{C}}

\newcommand{\Hpazo}{\pazocal{H}}
\newcommand{\Npazo}{\pazocal{N}}
\newcommand{\Lpazo}{\pazocal{L}}
\newcommand{\Ppazo}{\pazocal{P}}

\newcommand{\Dpazo}{\pazocal{D}}

\newcommand{\Spazo}{\pazocal{S}}

\newcommand{\Xpazo}{\pazocal{X}}

\newcommand{\Dcal}{\mathcal{D}}

\newcommand{\Lcal}{\mathcal{L}}

\newcommand{\Ccal}{\mathcal{C}}
\newcommand{\Scal}{\mathcal{S}}

\newcommand{\LcalI}{\Lcal^1_{\mathlarger{\llcorner} I}}
\newcommand{\LcalII}{\Lcal^2_{\mathlarger{\llcorner} I \times I}}

\newcommand{\Id}{\textnormal{Id}}

\newcommand{\supp}{\textnormal{supp}}

\newcommand{\Lip}{\textnormal{Lip}}

\newcommand{\loc}{\textnormal{loc}}
\newcommand{\Div}{\textnormal{div}}

\newcommand{\rg}{\textnormal{rg}}
\newcommand{\argmax}{\textnormal{argmax}}
\newcommand{\disc}{\textnormal{disc}}
\newcommand{\ess}{\textnormal{ess}}
\newcommand{\sym}{\textnormal{sym}}

\newcommand{\textbn}[1]{\textnormal{\textbf{#1}}}
\newcommand{\co}{\overline{\textnormal{co}} \hspace{0.05cm}}

\newcommand{\xb}{\boldsymbol{x}}
\newcommand{\yb}{\boldsymbol{y}}

\newcommand{\vb}{\boldsymbol{v}}

\newcommand{\Ab}{\boldsymbol{A}}

\newcommand{\Lb}{\boldsymbol{L}}

\newcommand{\vt}{\texttt{t}}

\renewcommand{\epsilon}{\varepsilon}

\newcommand{\INTDom}[3]{\int_{#2} #1 \textnormal{d} #3}
\newcommand{\INTSeg}[4]{\int_{#3}^{#4} #1 \textnormal{d} #2}
\newcommand{\NormL}[3]{\parallel \hspace{-0.1cm} #1 \hspace{-0.1cm} \parallel _ {L^{#2}(#3)}}

\newcommand{\Norm}[1]{\parallel \hspace{-0.1cm} #1 \hspace{-0.1cm} \parallel}

\newcommand{\tderv}[2]{\tfrac{\textnormal{d} #1}{ \textnormal{d} #2}}

\newtheorem{Def}{Definition}[section]
\newtheorem{thm}[Def]{Theorem}
\newtheorem{prop}[Def]{Proposition}
\newtheorem{rmk}[Def]{Remark}
\newtheorem{lem}[Def]{Lemma}

\newtheorem*{open}{Open problem}

\numberwithin{figure}{section}

\newtheorem{taggedhypx}{Hypotheses}
\newenvironment{taggedhyp}[1]
    {\renewcommand\thetaggedhypx{#1}\taggedhypx}
    {\endtaggedhypx}

\title{Consensus Formation in First-Order Graphon Models with Time-Varying Topologies}

\date{December 8, 2022}

\author{Benoît Bonnet\footnote{Inria Paris $\mathsmaller{\&}$ Laboratoire Jacques-Louis Lions, Sorbonne Université, Université Paris-Diderot SPC, CNRS, Inria, 75005 Paris, France \textit{Email:} \texttt{benoit.a.bonnet@inria.fr}} , Nastassia Pouradier Duteil\footnote{Inria Paris $\mathsmaller{\&}$ Laboratoire Jacques-Louis Lions, Sorbonne Université, Université Paris-Diderot SPC, CNRS, Inria, 75005 Paris, France \textit{Email:} \texttt{nastassia.pouradierduteil@inria.fr}}\, and Mario Sigalotti\footnote{Inria Paris $\mathsmaller{\&}$ Laboratoire Jacques-Louis Lions, Sorbonne Université, Université Paris-Diderot SPC, CNRS, Inria, 75005 Paris, France \textit{Email:} \texttt{mario.sigalotti@inria.fr}}}

\begin{document}

\maketitle

\begin{abstract}
In this article, we investigate the asymptotic formation of consensus for several classes of time-dependent cooperative graphon dynamics. After motivating the use of this type of macroscopic models to describe multi-agent systems, we adapt the classical notion of scrambling coefficient to this setting, leverage it to establish sufficient conditions ensuring the exponential convergence to consensus with respect to the $L^{\infty}$-norm topology. We then shift our attention to consensus formation expressed in terms of the $L^2$-norm, and prove three different consensus result for symmetric, balanced and strongly connected topologies, which involve a suitable generalisation of the notion of algebraic connectivity to this infinite-dimensional framework. We then show that, just as in the finite-dimensional setting, the notion of algebraic connectivity that we propose encodes information about the connectivity properties of the underlying interaction topology. We finally use the corresponding results to shed some light on the relation between $L^2$- and $L^{\infty}$-consensus formation, and illustrate our contributions by a series of numerical simulations.
\end{abstract}

{\footnotesize

\vspace{0.25cm}

\textbf{MSC2020 Subject Classification :} 05C63, 05C90, 37L15, 93A16
}

\tableofcontents

%%%%%%%%%%%%%%%%%%%%%%%%%%%%%%%%%%%%%%%%%%%%%%%%%%%%%%%%%%%%%%%%%%%%%%%%%%%
%								NEW SECTION AHEAD                         %
%%%%%%%%%%%%%%%%%%%%%%%%%%%%%%%%%%%%%%%%%%%%%%%%%%%%%%%%%%%%%%%%%%%%%%%%%%%

\section{Introduction}
\setcounter{equation}{0} \renewcommand{\theequation}{\thesection.\arabic{equation}}

Since its appearance at the end of the nineteen nineties, the analysis of self-organisation and pattern formation in large groups of interacting individuals has become a prominent topic in applied mathematics. This subject was initially brought forth by the pioneering articles \cite{Flierl1999,Vicsek1995} on statistical physics oriented models for biological systems, and subsequently cemented by a wealth of contributions in the fields of automation theory and engineering, see e.g. \cite{BeardRen,Blondel2009,Blondel2005,Jadbabaie2003,Martin2014,Egerstedt2010,Moreau2005,Tanner2007} and references therein. In the midst of this broad academical trend, a research current led by the works of Hegselman and Krause \cite{Hegselmann2002,Krause2000} on bounded confidence models, and the groundbreaking papers of Cucker and Smale \cite{CS1,CS2} on emergent behaviours, started to focus more specifically on the problem of \textit{consensus formation}. In the meantime, many of the ideas that came into existence around these questions found their way to a large score of application areas, such as herds and flocks analysis \cite{Ballerini,Bellomo2013}, aggregation models in biology \cite{Chuang2007}, decentralised control of fleets of autonomous vehicles \cite{BeardRen,Bullo2009} or pedestrian dynamics \cite{CPT,Piccoli2018}. Besides consensus problems, which by far attracted the most attention in the literature, other families of relevant and perhaps more intricate stable patterns -- such as mills \cite{Birnir2007,Carrillo2021} or rings \cite{Bertozzi2015} -- are being investigated by communities centered around multi-agent dynamics.

Motivated by the apparent ubiquity of clustering behaviours in applied sciences, a growing literature has been devoted to the precise mathematical understanding of the mechanisms subtending this type of pattern formation in the class of \textit{multi-agent systems}. These latter colloquially refer to systems of $N \geq 1$ agents represented by points in a given configuration space -- in our context the euclidean space $\R^d$ --, which evolve according to a coupled dynamics of the form
\begin{equation}
\label{eq:IntroODE}
\dot x_i(t) = \frac{1}{N} \sum_{j=1}^N \psi_{ij}( t,x_i(t)-x_j(t)).
\end{equation}
Here, the vector $(x_1(t),\dots,x_N(t)) \in (\R^d)^N$ represents the collection of all the states of the agents at some time $t \geq 0$, while the maps $\psi_{ij} : \R_+ \times \R^d \rightarrow \R^d$ defined for each pair of indices $i,j \in \{1,\dots,N\}$, encode nonrepulsive pairwise interactions between agents, which may depend on time, on their relative distance and orientation, as well as on their respective labels. In this context, one is usually interested in understanding under which circumstances solutions of \eqref{eq:IntroODE} \textit{converge to consensus}, namely, when there does exist an element $x^{\infty} \in \R^d$ such that 
\begin{equation*}
\lim_{t \rightarrow +\infty} |x_i(t) - x^{\infty}| = 0, 
\end{equation*}
for each agent label $i \in \{1,\dots,N\}$. This question has been the object of a tremendously large amount of work -- see e.g. \cite{Blondel2005,Caponigro2013,Caponigro2015,Dalmao2011,Jabin2014,Jadbabaie2003,Martin2014,Moreau2005,Motsch2011,Olfati2004,Shen2008,TahbazSalehi2008} and references therein --, whose number can be explained by the necessity to tackle a wide variety of specific interaction rules described by the maps $(\psi_{ij})_{1 \leq i,j \leq N}$. Some of the farthest-reaching results on this topic can be attributed to Moreau \cite{Moreau2005} as well as to Olfati-Saber and Murray \cite{Olfati2004}, while the main approaches and proof strategies to study asymptotic clustering are outlined in the reference survey \cite{Motsch2014} by Motsch and Tadmor.

Even though the literature devoted to consensus formation is very broad, quantitative convergence results -- typically obtained by means of Lyapunov arguments that do not rely on LaSalle's invariance principle --, are almost exclusively based on \textit{diameter} or \textit{energy} decay estimates. In both cases, the convergence towards consensus is dictated by a scalar quantity, which carries information about the connectivity properties of the system. In the first scenario, diameter contraction estimates usually depend on the so-called \textit{scrambling coefficient} of the topology. The latter is positive whenever each pair of agents in the system is either directly interacting, or following a common third party agent, and thus encodes information about ``local'' communications in the system. It first appeared in the theory of stochastic matrices as a way of quantifying ergodicity properties of Markov chains, and is discussed in the reference monograph of Seneta \cite{Seneta1979}. In the second case, energy dissipation estimates are often expressed in terms of \textit{algebraic connectivity}. This quantity -- which is widespread in algebraic graph theory -- corresponds to the second smallest eigenvalue of the so-called \textit{graph-Laplacian} matrix associated with the interaction topology. It is known to provide information about the number of connected components in the interaction graph, and allows to quantity the minimum communication strength between agents over all connected components. As such, the algebraic connectivity carries a fairly ``global'' information about the structure of the topology, and a detailed account on its many properties in the context of undirected graph can be found in the reference article of Mohar \cite{Mohar1991}.

\bigskip

In addition to identifying conditions ensuring the formation of consensus for systems of the form \eqref{eq:IntroODE}, it is also very natural to try to understand the behaviour of such patterns in large interacting systems, i.e. when the number $N \geq 1$ of agents goes to infinity. This problem is all the more relevant since in the common application areas of cooperative systems, the number of agents can range between tens of thousands and millions, making frontal resolution potentially arduous from a numerical standpoint. In the situation where the interaction maps $(\psi_{ij})_{1 \leq i,j \leq N}$ do not depend on the agent labels $i,j \in \{1,\dots,N\}$ and are all equal to a common function $\psi : \R_+ \times \R^d \rightarrow \R^d$, it was shown in \cite{Carrillo2010,HaLiu,Ha2008} that meaningful kinetic approximations formulated in terms of \textit{mean-field limits} could be proposed for \eqref{eq:IntroODE}, see also \cite{ControlKCS}. These latter are formulated as transport equations of the form 
\begin{equation}
\label{eq:IntroPDE}
\partial_t \mu(t) + \Div_x \big( \Psi(t,\mu(t)) \mu(t) \big) = 0, 
\end{equation}
where the nonlocal interaction velocities driving the system are given by 
\begin{equation*}
\Psi(t,\mu,x) := \INTDom{\psi(t,y-x)}{\R^d}{\mu(y)}.
\end{equation*}
In this context, a rigorous embedding can be established between the microscopic system \eqref{eq:IntroODE} and its macroscopic approximation \eqref{eq:IntroPDE}, by considering the \textit{empirical measures} defined by
\begin{equation*}
\mu_N(t) := \frac{1}{N} \sum_{i=1}^N \delta_{x_i(t)}
\end{equation*}
which are supported on the discrete positions $(x_1(t),\dots,x_N(t)) \in (\R^d)^N$ of the agents at time $t \geq 0$. These infinite-dimensional models are usually studied by means of a variety of tools coming from optimal transport theory -- see for instance the survey \cite{Choi2014} --, and we point the reader to \cite{ContInc,Pedestrian} for standard Cauchy-Lipschitz well-posedness results for nonlocal dynamics of the form \eqref{eq:IntroPDE}.

However when the interaction kernels $(\psi_{ij})_{1 \leq i,j \leq N}$ depend explicitly on the agent labels, the standard mean-field approach is not applicable any more. Indeed, the latter can only account for systems whose dynamics are invariant under permutations, which greatly limits the scope of admissible interaction models. Motivated by this observation, several recent contributions have been aiming at studying macroscopic approximations of \eqref{eq:IntroODE} which allow for more general and possibly asymmetric interaction functions. The corresponding approaches are based on the theory of \textit{graph limits} introduced by Lov\'asz and Segedy in \cite{Lovasz2006} -- see also the corresponding monograph \cite{Lovasz2012} --, and popularised by Medvedev in \cite{Medvedev2014,Medvedev2014Bis}. In this context, the discrete set of labels $\{1,\dots,N\}$ allowing to discriminate between agents becomes a continuous unit interval $I := [0,1]$, and the state of the system at time $t \geq 0$ is represented by a Lebesgue measurable function $x(t) \in L^2(I,\R^d)$ which maps to each generalised index $i \in I$ a position $x(t,i) \in \R^d$. In this context, the \textit{graphon dynamics} corresponding to \eqref{eq:IntroODE} writes as
\begin{equation}
\label{eq:IntroGraphonGen}
\partial_t x(t,i) = \INTDom{\psi(t,i,j,x(t,i) - x(t,j))}{I}{j}, 
\end{equation}
where $\psi : \R_+ \times I \times I \times \R^d \rightarrow \R^d$ is a general function encoding interactions between agents, depending on time, on their relative position and on their respective labels. At present, graphon models of the form \eqref{eq:IntroGraphonGen} are starting to emerge in the literatures related to multi-agent systems \cite{Ayi2021,Boudin2022}, control theory \cite{Biccari2019,Gao2019}, mean-field games \cite{Caines2019,Carmona2021}, and data sciences \cite{Ruiz2020,Vizuete2021}. We also point the reader to the papers \cite{Boudin2022,Lee2018}, which to the best of our knowledge are currently the only ones investigating consensus formation problems for graphon dynamics. We also mention the much older work \cite{Blondel2009}, in which the authors intuited a consensus model with a continuum of agent labels without relying on, or referring to, the terminology of dense graphs. 

\bigskip

In this article, our goal is to investigate several families of sufficient conditions ensuring the formation of consensus for solutions of generalised Cucker-Smale graphon models of the form 
\begin{equation}
\label{eq:IntroGraphon}
\partial_t x(t,i) = \INTDom{a(t,i,j) \phi(|x(t,i) - x(t,j)|)(x(t,j) - x(t,i))}{I}{j}.
\end{equation}
Here, the map $\phi : \R_+ \rightarrow \R_+$ is a smooth and bounded nonlinear interaction kernel accounting for the \textit{distance-based} magnitude of communications between agents, while the time-dependent kernel $a : \R_+ \times I \times I \rightarrow [0,1]$ encodes the \textit{interaction topology} of the system. We stress that the latter is not symmetric a priori, so that the fine structure of the interaction topology will be the main factor determining what kind of convergence result one may hope to recover. Taking inspiration from the classical dichotomy between $L^{\infty}$- and $L^2$-dissipation estimates in the consensus literature, we start in Section \ref{section:DiamConsensus} by studying sufficient conditions ensuring the contraction of the diameter of the system
\begin{equation*}
\Dpazo(t) := \underset{i,j \in I}{\textnormal{ess sup}} |x(t,i) - x(t,j)|, 
\end{equation*}
defined for all times $t \geq 0$ along solutions of \eqref{eq:IntroGraphon}. To this end, we introduce in Section \ref{subsection:Scrambling} the following macroscopic generalisation for graphon models of the classical scrambling coefficient 
\begin{equation*}
\eta(\Apazo(t)) := \underset{{i,j \in I}}{\textnormal{ess inf}} \INTDom{\min \Big\{ a(t,i,k),a(t,j,k) \Big\}}{I}{k}.
\end{equation*}
Here, $\Apazo(t) \in \Lpazo(L^2(I,\R^d))$ denotes the kernel-type \textit{adjacency operator} generated by $a(t) \in L^{\infty}(I \times I,[0,1])$, and can be seen as the counterpart of the notion of adjacency matrix for standard graphs. By following an original methodology relying on the theorems of Danskin and Scorza-Dragoni, combined with a key geometric lemma on the maximisers of relaxed diameters, we show in Theorem \ref{thm:DiamContraction} that 
\begin{equation}
\label{eq:IntroDiam}
\Dpazo(t) \leq \Dpazo(0) \exp \bigg( - \gamma_R \INTSeg{\eta(\Apazo(s))}{s}{0}{t} \bigg), 
\end{equation}
for all times $t \geq 0$, where $\gamma_R > 0$ is a constant related to the minimum value taken by the nonlinear kernel $\phi(\cdot)$ when the $L^{\infty}$-norm of the initial datum of \eqref{eq:IntroGraphon} is bounded by some constant $R > 0$. In addition to its novelty in the context of graphon models, this contraction estimate also generalises those already available for discrete multi-agent systems, as it does not require that the adjacency matrix be stochastic, see e.g. \cite{Motsch2011,Motsch2014}. Building on \eqref{eq:IntroDiam}, we proceed by showing in Theorem \ref{thm:ConsensusDiam} that solutions of \eqref{eq:IntroGraphon} exponentially converge to consensus in the $L^{\infty}$-norm topology as soon as the scrambling coefficients $\eta(\Apazo(\cdot))$ are \textit{persistent}, in the sense that
\begin{equation}
\label{eq:IntroPersistenceEta}
\frac{1}{\tau} \INTSeg{\eta(\Apazo(s))}{s}{t}{t+\tau} \geq \mu, 
\end{equation}
for all times $t \geq 0$ and some fixed parameters $(\tau,\mu) \in \R_+^* \times (0,1]$. This type of integral condition is fairly standard in stability theory for time-varying systems, as evidenced by the monographs \cite{MazencMalisoff,Narendra1989}. 

In Section \ref{section:ConsensusL2}, we shift our focus to energy dissipation estimates  involving the standard deviations
\begin{equation*}
X(t) := \bigg( \frac{1}{2} \INTDom{\INTDom{|x(t,i) - x(t,j)|^2}{I}{j}}{I}{i} \bigg)^{1/2}, 
\end{equation*}
defined for all times $t \geq 0$, and which allow for a convenient investigation of $L^2$-consensus formation. To this end, we start in Section \ref{subsection:Algebraic} by reformulating the dynamics \eqref{eq:IntroGraphon} as the differential equation 
\begin{equation}
\label{eq:IntroLapDyn}
\dot x(t) = -\Lbb(t,x(t)) x(t), 
\end{equation}
posed in the Hilbert space $L^2(I,\R^d)$. Therein, we introduced the so-called \textit{graph-Laplacian operators} $\Lbb(t,x) \in \Lpazo(L^2(I,\R^d))$ associated with the nonlinear interaction topology of the system. Taking inspiration from the work of Wu \cite{Wu2005,Wu2005ter,Wu2005bis} on the adaptation of the notion of \textit{algebraic connectivity} to directed graphs, we propose several generalisations of this object to the setting of graphon dynamics. For \textit{symmetric} and \textit{balanced} interaction topologies -- see Definition \ref{def:Balanced} below --, we define the algebraic connectivity of an interaction topology whose graph-Laplacian operator is $\Lbb \in \Lpazo(L^2(I,\R^d))$ as
\begin{equation}
\label{eq:IntroAlgebraic}
\lambda_2(\Lbb) := \inf_{x \in \Ccal^{\perp}} \frac{\langle \Lbb x , x \rangle_{L^2(I,\R^d)}}{\NormL{x}{2}{I,\R^d}^2}, 
\end{equation}
where $\Ccal \subset L^2(I,\R^d)$ is the so-called \textit{consensus manifold}. In the subtler and more technical case in which the interaction topology is a disjoint union of strongly connected components -- understood in the sense of Definition \ref{def:StrongCon} below --, it is again possible to define the algebraic connectivity $\lambda_2(\Lbb)$ in a fashion that is similar to \eqref{eq:IntroAlgebraic} by renormalising $\Lbb \in \Lpazo(L^2(I,\R^d))$ using a suitable function belonging to the kernel of its adjoint. Let it be noted that the existence of this canonical weight for graphon models is a consequence of the recent work \cite{Boudin2022}. 

We then proceed by deriving three general convergence results towards consensus depending on the nature of the interaction topology, and presented in ascending order of generality and difficulty. We thus start in Section \ref{subsubsection:Symmetric} by the simplest scenario in which the adjacency operators $\Apazo(t) \in \Lpazo(L^2(I,\R^d))$, whose kernels are $a(t) \in L^{\infty}(I \times I,[0,1])$, define symmetric interaction topologies for almost every time $t \geq 0$. By adapting a methodology developed in \cite{CSComFail} by the first author, we are then able to prove in Theorem \ref{thm:ConsensusSym} that solutions of \eqref{eq:IntroLapDyn} exponentially converge to consensus in the $L^2$-norm topology, whenever the interaction topology satisfies the persistence condition 
\begin{equation}
\label{eq:IntroPersistenceLambda21}
\lambda_2 \bigg( \frac{1}{\tau} \INTSeg{\Lbb^a(s)}{s}{t}{t+\tau} \bigg) \geq \mu 
\end{equation}
for all times $t \geq 0$, in which $\Lbb^a(t)$ stands for the linear graph-Laplacian 
\begin{equation*}
\Lbb^a(t) : x \mapsto \INTDom{a(t,i,j) (x(i) - x(j))}{I}{j},
\end{equation*}
generated by the adjacency operator $\Apazo(t)$ at some time $t \geq 0$. It is worth recalling that, for finite-dimensional multi-agent systems, persistence conditions of the form \eqref{eq:IntroPersistenceLambda21} bearing on the smallest eigenvalue of the averaged matrix are in fact \textit{necessary and sufficient} for the exponential convergence towards consensus, as shown in \cite{Manfredi2016}. We subsequently move on in Section \ref{subsubsection:Balanced} to the more general situation in which the interaction graph is directed, but balanced, in the sense that the in- and out-degrees coincide at each node. Contrary to the symmetric case, we need to restrict our attention to linear dynamics in this context. In Theorem \ref{thm:ConsensusBalanced}, we show that solutions of \eqref{eq:IntroLapDyn} for which $\phi \equiv 1$ and $\Apazo(t)$ defines a balanced topology for almost all $t \geq 0$ exponentially converge to consensus in the $L^2$-norm topology, whenever
\begin{equation}
\label{eq:IntroPersistenceLambda22}
\frac{1}{\tau} \INTSeg{\lambda_2(\Lbb(s))}{s}{t}{t+\tau} \geq \mu 
\end{equation}
for all times $t \geq 0$. At this point, a crucial observation has to be made regarding the persistence conditions \eqref{eq:IntroPersistenceLambda21} and \eqref{eq:IntroPersistenceLambda22}. Indeed, the latter -- which heuristically imposes that the algebraic connectivity of the system be positive sufficiently often -- is much more restrictive than the former, which merely requires that the average graph on every time window of length $\tau > 0$ be strongly connected. This requirement can be met even in the extreme case where $\lambda_2(\Lbb^a(t)) = 0$ for almost every times $t \geq 0$, which corresponds to the scenario in which the interaction graph is disconnected at all times. This intuition is made rigorous in Proposition \ref{prop:Concavity}, where we show that $\lambda_2(\cdot)$ is in fact a concave mapping over the space of linear operators. Finally in Section \ref{subsubsection:StronglyConnected}, we treat the more general situation in which the adjacency operators $\Apazo(t)$ define topologies which are disjoint unions of strongly connected components -- see Definition \ref{def:StrongCon} below -- for almost every time $t \geq 0$. These are much more delicate to handle than balanced topologies, as the system admits dissipative Lyapunov functions which are different for each possible interaction topology. For this reason, even though the corresponding exponential convergence result is supported by a suitable generalisation of the persistence condition \eqref{eq:IntroPersistenceLambda22}, we need to impose additional assumptions on the structure of the interactions to be able to compare the different energy functionals. The corresponding result, which builds on the methodology introduced for balanced graphons combined with stability techniques for switching systems, is presented in Theorem \ref{thm:ConsensusStrong}.

We conclude the paper in Section \ref{section:Discussion} by establishing several auxiliary results which shed light on the interplay between consensus formation and graphon connectivity. In Section \ref{subsection:Connectivity}, we generalise to symmetric graphon models one of the instrumental results about algebraic connectivity, namely that the underlying interaction topology is strongly connected if and only if $\lambda_2(\Lbb) > 0$. We then prove in Section \ref{subsection:L2Linfty} a somewhat surprising result, stating that $L^2$- and $L^{\infty}$-consensus formation are in fact equivalent whenever the in-degree function associated with the graphon model is persistent in a suitable sense. We end the article by displaying numerical illustrations of our consensus results in Section \ref{subsection:Numerics}, along with a discussion on some of the limit cases of the theory that we develop here. 

\bigskip

The structure of the article is the following. In Section \ref{section:Preliminaries}, we present standard notions of measure, graphs and spectral theory, and provide a short primer on discrete multi-agent systems and graphon models. In particular, we prove in this occasion that the convex hull of a graphon dynamics is nonincreasing, which is a new and interesting result in its own right. In Section \ref{section:DiamConsensus}, we define scrambling coefficients for graphons, and study diameter contraction estimates and $L^{\infty}$-consensus formation. In Section \ref{section:ConsensusL2}, we adopt another point of view and study the formation of $L^2$-consensus for symmetric, balanced, and strongly connected interaction topologies. These latter are all based on adequate generalisations of the notion of algebraic connectivity for graph-Laplacian operators. We finally conclude the paper by proving auxiliary results and presenting numerical illustrations in Section \ref{section:Discussion}. 

%%%%%%%%%%%%%%%%%%%%%%%%%%%%%%%%%%%%%%%%%%%%%%%%%%%%%%%%%%%%%%%%%%%%%%%%%%%%
%								NEW SECTION AHEAD						   %
%%%%%%%%%%%%%%%%%%%%%%%%%%%%%%%%%%%%%%%%%%%%%%%%%%%%%%%%%%%%%%%%%%%%%%%%%%%%

\section{Preliminaries}
\label{section:Preliminaries}
\setcounter{equation}{0} \renewcommand{\theequation}{\thesection.\arabic{equation}}

In this section, we recollect a series of notions pertaining to functional analysis and multi-agent dynamics at large, all of which will be extensively used throughout the article. 

%%%%%%%%%%%%%%%%%%%%%%%%%%%%%%%%%%%%%%%%%%%%%%%%%%%%%%%%%%%%%%%%%%%%%%%%%%%%

\subsection{Measure theory and integration}

We start here by introducing some notations and results of measure theory, for which we refer the reader to the very complete monographs  \cite{AmbrosioFuscoPallara,EvansGariepy}.  

In what follows, given an integer $n \geq 1$, we denote by $\Lcal^n$ the standard $n$-dimensional Lebesgue measure defined over $\R^n$. Letting $p \in [1,+\infty]$, taking a Borel set $\Omega \subset \R^n$ and denoting by $(X,\Norm{\cdot}_X)$ a separable Banach space, we shall use the notation $L^p(\Omega,X)$ for the space of $X$-valued integrable maps with respect to $\Lcal^n$, defined in the sense of Bochner (see e.g. \cite[Chapter II]{DiestelUhl}). We will also write $C^0(\Omega,\Scal)$ and $\Lip_{\loc}(\Omega,\Scal)$ to refer to continuous and locally Lipschitz mappings with values in a complete separable metric space $(\Scal,d_{\Scal})$ respectively. Throughout this article, we will use the notations ``$\inf$'', ``$\sup$'' and ``$\rg$'' to refer to the essential infimum, supremum and range of a measurable function, all of which are understood with respect to the standard Lebesgue measure.

We state below simplified versions of the \textit{Scorza-Dragoni} and \textit{Danskin} theorems (see \cite{Berliocchi1973} and \cite{Danskin1967} respectively), which will be used in the sequel. While these latter are both known results in modern analysis, they are somewhat nonstandard, and we include their statements for the sake of completeness. 

\begin{thm}[Scorza-Dragoni]
\label{thm:ScorzaDragoni}
Let $\Omega \subset \R^n$ be a Borel set and $f : \R_+ \times \Omega \rightarrow \Scal$ be such that $x \in \Omega \mapsto f(t,x) \in \Scal$ is $\Lcal^n$-measurable for each $t \geq 0$,  and $t \in \R_+ \mapsto f(t,x) \in \Scal$ is continuous for $\Lcal^n$-almost every $x \in \Omega$. Then for every $\epsilon > 0$, there exists a compact set $\Omega_{\epsilon} \subset \Omega$ satisfying $\Lcal^n(\Omega \setminus \Omega_{\epsilon}) < \epsilon$ and such that the restricted map $f : \R_+ \times \Omega_{\epsilon} \rightarrow \Scal$ is continuous. 
\end{thm}

\begin{thm}[Danskin]
\label{thm:Danskin}
Let $\Omega \subset \R^n$ be a compact set and $f : \R_+ \times \Omega \rightarrow \R$ be a continuous function such that $t \in \R_+ \mapsto f(t,x) \in \R$ is differentiable for all $x \in \Omega$. Then, the application $g : t \in \R_+ \mapsto \max_{x \in \Omega} f(t,x) \in \R$ is differentiable $\Lcal^1$-almost everywhere, with
\begin{equation*}
\tderv{}{t}{} g(t) = \max_{x \in \overline{\Omega}(t)} \partial_t f(t,x)
\end{equation*}
for $\Lcal^1$-almost every $t \geq 0$, where we introduced the notation $\overline{\Omega}(t) := \underset{x \in \Omega}{\argmax} \, f(t,x)$.  
\end{thm}

In the remainder of this article, we will often use both theorems in conjunction with the following convergence result for restricted essential supremums. Therein, we denote by $\Ppazo(\Omega)$ the power set generated by some arbitrary $\Omega \subset \R^n$. 

\begin{lem}[Quantitative interior estimates for essential supremums]
\label{lem:SupConv}
Let $\Omega \subset \R^n$ be a compact set and $f \in L^{\infty}(\Omega,\R^d)$. Then for every $\delta > 0$, there exists $\epsilon > 0$ such that 
\begin{equation}
\label{eq:EsssupIneq}
\NormL{f}{\infty}{\Omega ,\R^d} - \delta \leq \, \NormL{f}{\infty}{\Omega_{\epsilon},\R^d} \, \leq \, \NormL{f}{\infty}{\Omega,\R^d}, 
\end{equation}
whenever $\Omega_{\epsilon} \subset \Omega$ is a measurable set satisfying $\Lcal^n(\Omega \setminus \Omega_{\epsilon}) < \epsilon$. In particular, it holds that 
\begin{equation*}
\NormL{f}{\infty}{\Omega_{\epsilon},\R^d} ~\underset{\epsilon \rightarrow 0^+}{\longrightarrow}~ \NormL{f}{\infty}{\Omega,\R^d},
\end{equation*}
for each family of sets $(\Omega_{\epsilon})_{\epsilon > 0} \subset \Ppazo(\Omega)$ satisfying these properties. 
\end{lem}

\begin{proof}
The second inequality in \eqref{eq:EsssupIneq} is trivially satisfied for every set $\Omega_{\epsilon} \subset \Omega$. Concerning the first inequality, observe that by definition of the $L^{\infty}$-norm, there exists for every $\delta > 0$ a set $\Omega_{\delta} \subset \Omega$ of positive measure such that 
\begin{equation*}
|f(x)| \geq \, \NormL{f}{\infty}{\Omega,\R^d} - \, \delta
\end{equation*}
for $\Lcal^n$-almost every $x \in \Omega_{\delta}$. Hence, by choosing $\epsilon > 0$ satisfying $\epsilon < \Lcal^n(\Omega_{\delta})$ and considering any closed set $\Omega_{\epsilon} \subset \Omega$ with $\Lcal^n(\Omega \setminus \Omega_{\epsilon}) < \epsilon$, it necessarily holds that $\Lcal^n(\Omega_{\delta} \cap \Omega_{\epsilon}) > 0$. Indeed assuming the converse by contradiction, one would have
\begin{equation*}
\Lcal^n(\Omega_{\delta} \cup \Omega_{\epsilon}) = \Lcal^n(\Omega_{\delta}) + \Lcal^n(\Omega_{\epsilon}) > \Lcal^n(\Omega), 
\end{equation*}
which is absurd since both sets are contained in $\Omega$. Thus by construction, one has that 
\begin{equation*}
|f(x)| \geq \, \NormL{f}{\infty}{\Omega,\R^n} - \, \delta, 
\end{equation*}
for $\Lcal^n$-almost every $x \in \Omega_{\delta} \cap \Omega_{\epsilon}$, which concludes the proof of our claim.
\end{proof}

%%%%%%%%%%%%%%%%%%%%%%%%%%%%%%%%%%%%%%%%%%%%%%%%%%%%%%%%%%%%%%%%%%%%%%%%%%%%

\subsection{Graph and spectral theory} 
\label{subsection:Spectral}

In this second preliminary section, we recollect more specifically elementary notions pertaining to graph theory, together with elements of spectral analysis for bounded linear operators defined over Hilbert spaces. We point to the reference monographs \cite{West2001} and \cite{ReedI1981} respectively for a detailed and accessible treatment of these topics.

\paragraph*{Introductory notions of graph theory.}

Given an integer $N \geq 1$, we say that $\Ab_N := (a_{ij})_{1 \leq i,j \leq N} \in \R^{N \times N}$ is the \textit{adjacency matrix} of a directed graph -- or digraph -- with vertex labels $\{1,\dots,N\}$ provided that $a_{ij} \in [0,1]$ and $a_{ii} = 1$ for each $i,j \in \{1,\dots,N\}$. In the context of cooperative dynamics, the set of labels will represent the agents of the system, while each coefficient $a_{ij} \in [0,1]$ stands for the magnitude of communication flowing from agent $j$ to agent $i$. In particular, $a_{ij} > 0$ if and only if agent $j$ influences the dynamics of agent $i$. We recall below standard notions of connectivity for general digraphs, which are illustrated in Figure \ref{fig:Connected}

\begin{Def}[Strong and simple digraph connectivity]
\label{def:Connected}
Let $\Ab_N \in [0,1]^{N \times N}$ be an adjacency matrix. Then, the underlying digraph is \textnormal{strongly connected} if for every pair of indices $i,j \in \{1,\dots,N\}$, there exist an integer $m \geq 1$ and a finite sequence $(l_k)_{1 \leq k \leq m} \subset \{1,\dots,N\}$ such that $i = l_1$, $j = l_m$ and $a_{l_k l_{k+1}} > 0$ for each $k \in \{1,\dots,m-1\}$.

In addition, a digraph is said to be \textnormal{simply connected} if its symmetric part -- which is defined as the undirected graph supported by the adjacency matrix $\tfrac{1}{2}(\Ab_N + \Ab_N^{\top})$ --, is strongly connected. In particular, an undirected graph is connected if and only if it is strongly connected. 
\end{Def} 

\begin{figure}[!ht]
\centering
\begin{tikzpicture}
% First figure
% Drawing the nodes
\draw (1,-1) node[shape=circle, draw] {$1$}; 
\draw (2,0) node[shape=circle, draw] {$2$};
\draw (2,-2) node[shape=circle, draw] {$3$};
\draw (0,0) node[shape=circle, draw] {$5$};
\draw (0,-2) node[shape=circle, draw] {$4$};
% Drawing the arrows
\draw[->, line width = 0.75, blue] (0.75,-0.75)--(0.25,-0.25);
\draw[->, line width = 0.75, blue] (0,-0.35)--(0,-1.65); 
\draw[->, line width = 0.75, blue] (0.25,-1.75)--(0.75,-1.25);
\draw[<->, line width = 0.75, blue] (1.25,-0.75)--(1.75,-0.25);  
\draw[<->, line width = 0.75, blue] (1.75,-1.75)--(1.25,-1.25); 
% Second figure
% Drawing the nodes
\begin{scope}[xshift = 6cm]
\draw (1,-1) node[shape=circle, draw] {$1$}; 
\draw (2,0) node[shape=circle, draw] {$2$};
\draw (2,-2) node[shape=circle, draw] {$3$};
\draw (0,0) node[shape=circle, draw] {$5$};
\draw (0,-2) node[shape=circle, draw] {$4$};
% Drawing the arrows
\draw[->, line width = 0.75, blue] (0.75,-0.75)--(0.25,-0.25);
\draw[->, line width = 0.75, blue] (0,-0.35)--(0,-1.65); 
\draw[->, line width = 0.75, blue] (0.25,-1.75)--(0.75,-1.25);
\draw[<-, line width = 0.75, blue] (1.25,-0.75)--(1.75,-0.25);  
\draw[<-, line width = 0.75, blue] (1.75,-1.75)--(1.25,-1.25);
\end{scope}
% Third figure
% Drawing the nodes
\begin{scope}[xshift = 12cm]
\draw (1,-1) node[shape=circle, draw] {$1$}; 
\draw (2,0) node[shape=circle, draw] {$2$};
\draw (2,-2) node[shape=circle, draw] {$3$};
\draw (0,0) node[shape=circle, draw] {$5$};
\draw (0,-2) node[shape=circle, draw] {$4$};
% Drawing the arrows
\draw[<->, line width = 0.75, blue] (0,-0.35)--(0,-1.65); 
\draw[<-, line width = 0.75, blue] (2,-0.35)--(2,-1.65); 
\draw[<-, line width = 0.75, blue] (1.25,-0.75)--(1.75,-0.25);  
\draw[<-, line width = 0.75, blue] (1.75,-1.75)--(1.25,-1.25);
\end{scope}
\end{tikzpicture}
\caption{{\small \textit{Illustration of the notions of strong and simple connectivity on graphs with $N = 5$ vertices. The digraph on the left is strongly connected, while that in the center is simply connected. The digraph on the right is the disjoint union of two strongly connected components formed respectively by $\{1,2,3\}$ and $\{4,5\}$.}}}
\label{fig:Connected}
\end{figure}
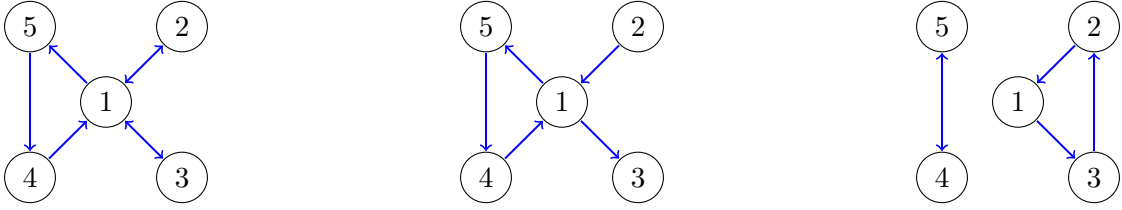

\paragraph*{Elements of spectral theory for bounded operators.} 

Let $(\Hpazo,\langle \cdot , \cdot \rangle_{\Hpazo})$ be an arbitrary finite- or infinite-dimensional Hilbert space, and $\Lpazo(\Hpazo)$ denote the space of all bounded linear operators from $\Hpazo$ into itself. We will say that an element $T \in \Lpazo(\Hpazo)$ is \textit{symmetric} if $T = T^*$, where  $T^* \in \Lpazo(\Hpazo)$ is the adjoint of $T$ with respect to $\langle \cdot,\cdot \rangle_{\Hpazo}$. We also use the notation $T^{\sym} := \tfrac{1}{2}(T + T^*)$ for its symmetric part. Notice that a symmetric operator which is bounded -- namely whose domain is the whole space -- is automatically self-adjoint in the usual sense. Writing $\Id$ for the identity operator over $\Hpazo$, we recall below several concepts pertaining to the \textit{spectrum} of an element of $\Lpazo(\Hpazo)$. 

\begin{Def}[The spectrum of bounded linear operators and its decomposition]
\label{def:Spectrum}
Given an element $T \in \Lpazo(\Hpazo)$, its \textnormal{spectrum} $\sigma(T)$ is defined as the set of all the real numbers $\lambda \in \R$ such that 
\begin{equation*}
(T - \lambda \, \Id) \in \Lpazo(\Hpazo) \qquad \text{is not bijective}. 
\end{equation*}
In addition, the spectrum of a bounded linear operator $T$ admits the following decomposition 
\begin{equation*}
\sigma(T) = \sigma_{\disc}(T) \cup \sigma_{\ess}(T), 
\end{equation*}
where the union is disjoint. Therein, $\sigma_{\disc}(T)$ denotes the \textnormal{discrete spectrum} of $T$, defined by 
\begin{equation*}
\begin{aligned}
\sigma_{\disc}(T) := \bigg\{ \lambda \in \sigma(T) ~\text{s.t.}~ & \text{there exists an open set $\Npazo_{\lambda} \subset \R$ satisfying $\Npazo_{\lambda} \cap \sigma(T) = \{ \lambda\}$} \\
& \text{and the space of all $f \in \Hpazo$ such that $T f = \lambda f$ is finite-dimensional } \bigg\}, 
\end{aligned}
\end{equation*}
while $\sigma_{\ess}(T) := \sigma(T) \setminus \sigma_{\disc}(T)$ is the \textnormal{essential spectrum} of $T$. 
\end{Def}

We will use the notation $I := [0,1]$ for the unit interval representing the continuum of indices, and denote by $\LcalI$ the restriction of the standard $1$-dimensional Lebesgue measure thereon. We shall say that $\Apazo \in \Lpazo(L^2(I,\R^d))$ is an \textit{adjacency operator} if there exists a map $a \in L^{\infty}(I\times I,[0,1])$ such that 
\begin{equation}
\label{eq:AdjacencyOp}
(\Apazo \, x)(i) := \INTDom{a(i,j)x(j)}{I}{j} \qquad \text{for $\LcalI$-almost every $i \in I$},
\end{equation}
for each $x \in L^2(I,\R^d)$. Analogously given a map $d \in L^{\infty}(I,[0,1])$, we define the corresponding \textit{multiplication operator} $\Mpazo_d \in \Lpazo(L^2(I,\R^d))$ by
\begin{equation}
\label{eq:MultiplicationOp}
(\Mpazo_d \, x)(i) := d(i) x(i) \qquad \text{for $\LcalI$-almost every $i \in I$},
\end{equation}
for each $x \in L^2(I,\R^d)$. We end this preliminary section by recalling the following standard structure result on the spectra of adjacency and multiplication operators, for which we point to \cite[Theorem VI.15]{ReedI1981} and \cite[Proposition 3.2]{Hardt1996} respectively. 

\begin{prop}[Spectra of adjacency and multiplication operators]
\label{prop:Spectrum}
Let $\Apazo,\Mpazo_d \in \Lpazo(L^2(I,\R^d))$ be an adjacency and a multiplication operator defined as in \eqref{eq:AdjacencyOp} and \eqref{eq:MultiplicationOp} respectively. Then, one has
\begin{equation*}
\sigma(\Apazo) = \sigma_{\disc}(\Apazo) = \{ \lambda_n \}_{n=1}^{+\infty} \qquad \text{and} \qquad \sigma(\Mpazo_d) = \sigma_{\ess}(\Mpazo_d) = \rg(d),
\end{equation*}
where $(\lambda_n) \subset [0,1]$ is an at most countable family of real numbers which can only accumulate at zero.
\end{prop}

%%%%%%%%%%%%%%%%%%%%%%%%%%%%%%%%%%%%%%%%%%%%%%%%%%%%%%%%%%%%%%%%%%%%%%%%%%%%

\subsection{Multi-agent systems and graphon models}

In this section, we recall some of the properties of the multi-agent systems under consideration throughout this article, both at the discrete and macroscopic level. We point the reader to \cite{Choi2014,Golse} for a detailed account on mean-field limits for particle systems, and to \cite{Lovasz2012,Medvedev2014} for a thorough analysis of graph limits. 

%%%%%%%%%%%%%%%%%%%%%%%%%%%%%%%%%%%%%%%%%%%%%%%%%%%%%%%%%%%%%%%%%%%%%%%%%%%%

\paragraph*{Finite-dimensional cooperative systems.} We start by considering a system of $N \geq 1$ agents, represented by a collection $ \xb(\cdot) := (x_1(\cdot),\dots,x_N(\cdot)) \in C^0(\R_+,(\R^d)^N)$ of curves in $\R^d$. Given a set $(x_1^0,\dots,x_N^0) \in (\R^d)^N$ of initial data, their evolution in time is prescribed by the family of ODEs
\begin{equation}
\label{eq:DiscreteDynamics}
\left\{
\begin{aligned}
& \dot x_i(t) = \frac{1}{N} \sum_{j=1}^N a_{ij}(t) \phi(|x_i(t) - x_j(t)|)(x_j(t) - x_i(t)), \\
& x_i(0) = x_i^0,
\end{aligned}
\right.
\end{equation}
written for every $i \in \{1,\dots,N\}$. Here, the positive nonlinear interaction kernel $\phi(\cdot)$ accounts for the magnitude of the pairwise interactions between agents -- depending on their relative distance --, while the maps $(a_{ij}(\cdot))_{1 \leq i,j \leq N }$ encode the time-varying interaction topology of the system. 

In the sequel, we make the following assumptions on the data of \eqref{eq:DiscreteDynamics}. 

\begin{taggedhyp}{\textbn{(MA)}}
\label{hyp:MA}
Assume that the following holds. 
\begin{enumerate}
\item[$(i)$] The weight functions $a_{ij} : \R_+ \rightarrow [0,1]$ are $\Lcal^1$-measurable for all $i,j \in \{1,\dots,N\}$. Moreover, one has that $a_{ii}(t) = 1$ for $\Lcal^1$-almost every $t \geq 0$ and each $i \in \{1,\dots,N\}$.
\item[$(ii)$] The nonlinear kernel $\phi : \R_+ \rightarrow \R_+^*$ is locally Lipschitz continuous, positive and bounded. 
\end{enumerate}
\end{taggedhyp}
Let it be noted that the weight functions $a_{ij}(\cdot)$ could more generally be allowed to take any positive value as long as they are bounded, but that the corresponding range can always be normalised to $[0,1]$ e.g. by suitably rescaling the nonlinear kernel $\phi(\cdot)$. Under this series of assumptions -- which are very standard in the literature devoted to consensus problems, see e.g. \cite{Caponigro2015,Motsch2014} --, we have the following  well-posedness result for \eqref{eq:DiscreteDynamics}. 

\begin{prop}[Well-posedness of the discrete dynamics]
Let $\xb^0 := (x_1^0,\dots,x_N^0)$ and suppose that hypotheses \ref{hyp:MA} hold. Then, there exists a unique curve $\xb \in \Lip_{\loc}(\R_+,(\R^d)^N)$  solution of \eqref{eq:DiscreteDynamics}. 
\end{prop}

\begin{proof}
It can be easily checked that under hypotheses \ref{hyp:MA}, the mapping 
\begin{equation*}
(t,\xb) \in \R_+ \times (\R^d)^N \mapsto \bigg( \frac{1}{N} \sum\limits_{j=1}^N a_{ij}(t) \phi(|x_i-x_j|)(x_j-x_i) \bigg)_{1 \leq i \leq N} \in (\R^d)^N, 
\end{equation*}
is $\Lcal^1$-measurable with respect to $t \geq 0$ as well as locally Lipschitz continuous and sublinear with respect to $\xb \in (\R^d)^N$. Whence, by classical well-posedness results for ordinary differential equations (see e.g. \cite[Chapter 1]{Filippov2013}), there exists a unique solution $\xb(\cdot) \in \Lip_{\loc}(\R_+,(\R^d)^N)$ to \eqref{eq:DiscreteDynamics}.
\end{proof}

%%%%%%%%%%%%%%%%%%%%%%%%%%%%%%%%%%%%%%%%%%%%%%%%%%%%%%%%%%%%%%%%%%%%%%%%%%%%

\paragraph*{Macroscopic approximations of multi-agent systems.} In the classical mean-field setting, macroscopic versions of the discrete dynamics \eqref{eq:DiscreteDynamics} are expressed by curves of empirical measures. However, as explained in the Introduction, this approach is only operational when the agents are \textit{indistinguishable}, or equivalently provided that the dynamics \eqref{eq:DiscreteDynamics} satisfies some invariance properties with respect to the agent labels $i \in \{ 1,\dots,N\}$. To circumvent this intrinsic limitation, one can instead consider \textit{graph limit} approximations, which are built as follows. Consider the piecewise constant mappings 
\begin{equation}
\label{eq:PiecewiseState}
x^N(t,i) := \sum_{k=1}^N \mathds{1}_{ \big[ \tfrac{k-1}{N}, \tfrac{k}{N} \big)}(i) \, x_k(t)
\end{equation}
and 
\begin{equation}
\label{eq:PiecewiseWeight}
a^N(t,i,j) := \sum_{k=1}^N \sum_{l=1}^N \mathds{1}_{ \big[ \tfrac{k-1}{N}, \tfrac{k}{N} \big)}(i) \mathds{1}_{ \big[ \tfrac{l-1}{N}, \tfrac{l}{N} \big)}(j) \, a_{kl}(t), 
\end{equation}
defined for $\Lcal^1$-almost every $t \geq 0$ and every $i,j \in I$. Then, one can show that the curve $x^N(\cdot) \in C^0(\R_+,L^2(I,\R^d))$ satisfies the following. 

\begin{prop}[Dynamics of the discrete graphon model]
Let $N \geq 1$ and $\xb(\cdot) \in \Lip_{\loc}(\R_+,(\R^d)^N)$ be a solution of \eqref{eq:DiscreteDynamics}. Then, the curve $x^N(\cdot) \in \Lip_{\loc}(\R_+,L^2(I,\R^d))$ defined in \eqref{eq:PiecewiseState} is a solution of the Cauchy problem 
\begin{equation}
\label{eq:DiscGraphonDyn}
\left\{
\begin{aligned}
& \partial_t x^N(t,i) = \INTDom{a^N(t,i,j) \phi(|x^N(t,i) - x^N(t,j)|)(x^N(t,j) - x^N(t,i))}{I}{j}, \\
& x^N(0,i) = x^0(i),
\end{aligned}
\right.
\end{equation}
formulated for $\LcalI$-almost every $i \in I$, where $a^N \in L^{\infty}(\R_+ \times I \times I,[0,1])$ is defined as in \eqref{eq:PiecewiseWeight}.
\end{prop}

\begin{proof}
Observe first that, for $\LcalI$-almost every $i \in I$, the map $t \in \R_+ \mapsto x^N(t,i) \in \R^d$ is differentiable $\Lcal^1$-almost everywhere. Then, by combining \eqref{eq:DiscreteDynamics} together with \eqref{eq:PiecewiseState}, one has that
\begin{equation*}
\begin{aligned}
\partial_t x^N(t,i) & = \sum_{k=1}^N \mathds{1}_{ \big[ \tfrac{k-1}{N}, \tfrac{k}{N} \big)}(i) \, \dot x_k(t) \\
& = \sum_{k=1}^N \mathds{1}_{ \big[ \tfrac{k-1}{N}, \tfrac{k}{N} \big)}(i) \Bigg(\frac{1}{N} \sum_{l=1}^N  \, a_{kl}(t) \phi(|x_k(t) - x_l(t)|)(x_l(t) - x_k(t)) \Bigg) \\
& = \sum_{l=1}^N \INTSeg{\Bigg( \sum_{k=1}^N \mathds{1}_{ \big[ \tfrac{k-1}{N}, \tfrac{k}{N} \big)}(i) \, a_{kl}(t) \phi(|x_k(t) - x_l(t)|)(x_l(t) - x_k(t)) \Bigg)}{j}{(l-1)/N}{l/N} \\
& = \INTDom{a^N(t,i,j) \phi(|x^N(t,i)-x_N (t,j)|)(x^N(t,j) - x^N(t,i))}{I}{j}, 
\end{aligned}
\end{equation*}
for $\Lcal^1$-almost every $t \geq 0$ and $\LcalI$-almost every $i \in I$, which concludes the proof of our claim. 
\end{proof}

\begin{rmk}[Subordination of finite multi-agent systems to graphon models]
It is worth noting that any finite-dimensional multi-agent system of the form \eqref{eq:DiscreteDynamics} can be recast as a graphon dynamics via \eqref{eq:PiecewiseState} and \eqref{eq:PiecewiseWeight}. This implies in particular that all the results that are discussed hereinbelow for general graphon models -- including the well-posedness results and sufficient conditions for various types of consensus formation -- have suitable counterparts in the discrete setting. 
\end{rmk}

%%%%%%%%%%%%%%%%%%%%%%%%%%%%%%%%%%%%%%%%%%%%%%%%%%%%%%%%%%%%%%%%%%%%%%%%%%%

\paragraph*{Well-posedness and elementary properties of graphon models.} While the dynamics in \eqref{eq:DiscGraphonDyn} describes the evolution of a piecewise constant macroscopic approximation of \eqref{eq:DiscreteDynamics}, it can more generally be seen as an infinite-dimensional integro-differential equation. Therefore in the sequel, given an initial datum $x^0 \in L^2(I,\R^d)$, we shall study the well-posedness of \textit{graphon dynamics} of the form
\begin{equation}
\label{eq:GraphonDynamics}
\left\{
\begin{aligned}
& \partial_t x(t,i) = \INTDom{a(t,i,j) \phi(|x(t,i)| - x(t,j)|)(x(t,j) - x(t,i))}{I}{j}, \\
& x(0) = x^0.
\end{aligned}
\right.
\end{equation}
In this context, the time-varying interaction topology is encoded by a measurable mapping $a \in L^{\infty}(\R_+ \times I \times I , [0,1])$, while the distance-based communications between agents remain described by a nonlinear interaction kernel $\phi(\cdot) \in \Lip_{\loc}(\R_+,\R_+^*)$. 

Throughout this manuscript, we will make the following assumptions on the data of the graphon dynamics \eqref{eq:GraphonDynamics}.

\begin{taggedhyp}{\textbn{(GD)}}
\label{hyp:GD}
Assume that the following holds. 
\begin{enumerate}
\item[$(i)$] The weight function $ a : \R_+ \times I \times I \rightarrow [0,1]$ is $\Lcal^1 \times \LcalI \times \LcalI$-measurable.  
\item[$(ii)$] The nonlinear kernel $\phi : \R_+ \rightarrow \R_+^*$ is locally Lipschitz continuous, positive and bounded.
\end{enumerate}
\end{taggedhyp}

We refer the reader to \cite{Ayi2021} for a detailed and pedagogical introduction to graphon models in the context of multi-agent systems.  

\begin{prop}[Well-posedness and estimates for graphon dynamics]
\label{prop:WellPosed}
Let $x^0 \in L^2(I,\R^d)$ and assume that hypotheses \ref{hyp:GD} hold. Then, there exists a unique solution $x(\cdot) \in \Lip_{\loc}(\R_+,L^2(I,\R^d))$ to the Cauchy problem \eqref{eq:GraphonDynamics}. If in addition it holds that $x^0 \in L^{\infty}(I,\R^d)$, one then has 
\begin{equation}
\label{eq:NormInequalityGraphon}
\NormL{x(t)}{\infty}{I,\R^d} \, \leq \, \NormL{x^0}{\infty}{I,\R^d}, 
\end{equation}
for all times $t \geq 0$. 
\end{prop}

\begin{proof}
The existence and uniqueness of a solution $x(\cdot) \in \Lip_{\loc}(\R_+,L^2(I,\R^d))$ to \eqref{eq:GraphonDynamics} can be recovered by repeating the method of \cite[Section 3.1]{Ayi2021}, up to some minor modifications. Assume now that $x^0 \in L^{\infty}(I,\R^d)$, and observe that by integrating \eqref{eq:GraphonDynamics} with respect to $t \geq 0$, one has 
\begin{equation}
\label{eq:RadiusEst1} 
\begin{aligned}
|x(t,i)| & \leq |x^0(i)| + \INTSeg{\INTDom{a(s,i,j) \phi(|x(s,i) - x(s,j)|) | x(s,j) - x(s,i) |}{I}{j}}{s}{0}{t}  \\
& \leq \,  \NormL{x^0}{\infty}{I,\R^d} + \, c_{\phi} \INTSeg{\bigg( |x(s,i)| + \INTDom{|x(s,j)|}{I}{j} \bigg)}{s}{0}{t}, 
\end{aligned}
\end{equation}
for all times $t \geq 0$ and $\LcalI$-almost every $i \in I$, where $c_{\phi} := \sup_{r \in \R_+} \phi(r) < +\infty$. By integrating \eqref{eq:RadiusEst1} with respect to $i \in I$ and applying Fubini's theorem along with Gr\"onwall's lemma, we obtain
\begin{equation}
\label{eq:MomentumEst}
\INTDom{|x(t,i)|}{I}{i} \leq \, \NormL{x^0}{\infty}{I,\R^d} \exp \big(2 c_{\phi} \hspace{0.025cm} t \big), 
\end{equation}
for all times $t \geq 0$. Then, plugging \eqref{eq:MomentumEst} into \eqref{eq:RadiusEst1} while resorting again to Gr\"onwall's lemma yields 
\begin{equation}
\label{eq:RadiusEst2} 
|x(t,i)| \leq \, \NormL{x^0}{\infty}{I,\R^d} \Big( 1 + \exp \big( 2 c_{\phi} \hspace{0.025cm} t \big) \Big) \frac{\exp \big( c_{\phi} \hspace{0.025cm} t \big)}{2}, 
\end{equation}
for $\LcalI$-almost every $i \in I$, which implies that $x(t) \in L^{\infty}(I,\R^d)$ for all times $t \geq 0$. 

Our goal now is to prove the sharper stability estimate displayed \eqref{eq:NormInequalityGraphon}. To this end, fix a real number $\epsilon > 0$, and observe that by Scorza-Dragoni's theorem (see Theorem \ref{thm:ScorzaDragoni} above), there exists a compact subset of indices $I_{\epsilon} \subset I$ satisfying $\LcalI(I \setminus I_{\epsilon}) < \epsilon$ such that the mapping $x : \R_+ \times I_{\epsilon} \rightarrow  \R^d$ is continuous. We then define for all times $t \geq 0$ the family of restricted supremum norms 
\begin{equation*}
L_{\epsilon}(t) := \max_{i \in I_{\epsilon}} |x(t,i)|.
\end{equation*}
The mapping $L_{\epsilon}(\cdot)$ is locally Lipschitz as the pointwise maximum of a family of equi-locally Lipschitz functions, and is therefore differentiable $\Lcal^1$-almost everywhere by Rademacher's theorem (see e.g. \cite[Theorem 3.2]{EvansGariepy}). Moreover by Danskin's theorem (see Theorem \ref{thm:Danskin} above), one has that
\begin{equation}
\label{eq:DerivativeLinfty}
\tfrac{1}{2} \tderv{}{t}{} L_{\epsilon}(t)^2 = \max_{i \in J_{\epsilon}(t)} \langle \partial_t x(t,i) , x(t,i) \rangle, 
\end{equation}
where $J_{\epsilon}(t) := \argmax_{i \in I_{\epsilon}} |x(t,i)|$ for $\Lcal^1$-almost every $t \geq 0$. Taking an arbitrary $i \in J_{\epsilon}(t)$, one observes that 
\begin{equation}
\label{eq:RelaxedRadEst1}
\begin{aligned}
\langle \partial_t x(t,i) , x(t,i) \rangle & = \INTDom{a(t,i,j) \phi(|x(t,i) - x(t,j)|) \big\langle x(t,i) , x(t,j) - x(t,i) \rangle}{I}{j} \\
& \leq \INTDom{a(t,i,j) \phi(|x(t,i) - x(t,j)|) L_{\epsilon}(t) \big( |x(t,j)| - L_{\epsilon}(t) \big) }{I_{\epsilon}}{j} \\
& \hspace{0.45cm} + \INTDom{a(t,i,j) \phi(|x(t,i) - x(t,j)|) L_{\epsilon}(t) \big( |x(t,j)| - L_{\epsilon}(t) \big)}{I \setminus I_{\epsilon}}{j} \\
& \leq c_{\phi} \epsilon \NormL{x(t)}{\infty}{I,\R^d}^2, 
\end{aligned}
\end{equation}
where we used the facts that $|x(t,j)| \leq L_{\epsilon}(t)$ for each $j \in I_{\epsilon}$ and $\LcalI(I \setminus I_{\epsilon}) < \epsilon$ by construction. By merging \eqref{eq:DerivativeLinfty} together with \eqref{eq:RelaxedRadEst1} and applying Gr\"onwall's lemma, we further obtain the estimate 
\begin{equation}
\label{eq:RelaxedRadEst2}
\tfrac{1}{2} L_{\epsilon}(t)^2 \leq \tfrac{1}{2} L_{\epsilon}(0)^2 + c_{\phi} \epsilon \INTSeg{\NormL{x(s)}{\infty}{I,\R^d}^2}{s}{0}{t},
\end{equation}
which holds for all times $t \geq 0$. Observing that by Lemma \ref{lem:SupConv}, we have the pointwise convergence
\begin{equation*}
L_{\epsilon}(t) ~\underset{\epsilon \rightarrow 0^+}{\longrightarrow}~ \NormL{x(t)}{\infty}{\Omega,\R^n},
\end{equation*}
we can pass to the limit as $\epsilon \rightarrow 0^+$ in \eqref{eq:RelaxedRadEst2} and conclude that 
\begin{equation*}
\NormL{x(t)}{\infty}{I,\R^d} \, \leq \, \NormL{x^0}{\infty}{I,\R^d}, 
\end{equation*}
for all times $t \geq 0$, which ends the proof of our claim. 
\end{proof}

Another feature of finite-dimensional multi-agent systems that still holds for graphon models is that the \textit{closed convex hull} of the essential range of the initial datum $x^0 \in L^{\infty}(I,\R^d)$, defined by 
\begin{equation*}
\Cpazo(x^0)= \co(\rg(x^0)) := \overline{\bigcap \bigg\{ K \subset \R^d ~\text{s.t.}~ \text{$K \subset \R^d$ is convex and $\rg(x^0) \subset K$} \bigg\}}, 
\end{equation*}
is invariant under the dynamics \eqref{eq:GraphonDynamics}, as shown in the following proposition.

\begin{prop}[Invariance of the convex hull]
\label{prop:ConvexHull}
Let $x^0 \in L^{\infty}(I,\R^d)$ and $x(\cdot) \in \Lip_{\loc}(\R_+,L^2(I,\R^d))$ be the corresponding solution of the graphon dynamics \eqref{eq:GraphonDynamics}. Then, the closed convex hulls satisfy the inclusion $\Cpazo(x(t)) \subset \Cpazo(x^0)$ for all times $t \geq 0$.
\end{prop} 

\begin{proof}
Similarly to what was done in the proof of Proposition \ref{prop:WellPosed}, we start by fixing  $\epsilon > 0$ and let $I_{\epsilon} \subset I$ be a compact subset given by Scorza-Dragoni's theorem (see Theorem \ref{thm:ScorzaDragoni} above). We also consider the restricted closed convex hulls, defined by  
\begin{equation*}
\Cpazo_{\epsilon}(x(t)) := \bigcap \bigg\{ K \subset \R^d ~\text{s.t.}~ \text{$K \subset \R^d$ is convex and $x(t,I_{\epsilon}) \subset K$} \bigg\}
\end{equation*}
for all times $t \geq 0$, which are compact sets because $x(t,I_{\epsilon})$ is compact. Furthermore, recall that as a consequence of the separation theorem (see e.g. \cite[Theorem 2.4.2]{Aubin1990}), the closed convex hull of the compact sets $x(t,I_{\epsilon}) \subset \R^d$ and $\rg(x(t))$ can be characterised respectively by the identities
\begin{equation}
\label{eq:ConvexHullCharac1}
\Cpazo_{\epsilon}(x(t)) = \bigg\{ z \in \R^d ~\text{s.t.}~ \langle p , z \rangle \leq ~ \max_{i \in I_{\epsilon}} \langle p , x(t,i)  \rangle ~~ \text{for each $p \in \R^d$} \bigg\}, 
\end{equation}
and
\begin{equation}
\label{eq:ConvexHullCharac2}
\Cpazo(x(t)) = \bigg\{ z \in \R^d ~\text{s.t.}~ \langle p , z \rangle \leq ~ \sup_{i \in I} \langle p , x(t,i)  \rangle ~~ \text{for each $p \in \R^d$} \bigg\}, 
\end{equation}
where we used the fact that the essential supremum of a real-valued measurable map coincides with the maximum value of its essential range. 

Given an element $p \in \R^d$, set $J_{\epsilon,p}(t) := \argmax_{i \in I_{\epsilon}} \langle p , x(t,i) \rangle$ and observe that by Danskin's theorem (see Theorem \ref{thm:Danskin} above), it holds that
\begin{equation*}
\tderv{}{t}{} \max_{i \in I_{\epsilon}} \langle p , x(t,i)  \rangle = \max_{i \in J_{\epsilon,p}(t)} \langle p , \partial_t x(t,i)  \rangle, 
\end{equation*}
for $\Lcal^1$-almost every $t \geq 0$. For an arbitrary $i \in J_{\epsilon,p}(t)$, one further has 
\begin{equation}
\label{eq:ConvexHullIneq1}
\begin{aligned}
\langle p , \partial_t x(t,i)  \rangle & = \INTDom{a(t,i,j) \phi(|x(t,i) - x(t,j)|) \big\langle p , x(t,j) - x(t,i) \big\rangle}{I}{j} \\
& = \INTDom{a(t,i,j) \phi(|x(t,i) - x(t,j)|) \big\langle p , x(t,j) - x(t,i) \big\rangle}{I_{\epsilon}}{j} + \epsilon \hspace{0.025cm} C_p(t), 
\end{aligned}
\end{equation}
where $C_p(\cdot) \in L^{\infty}_{\loc}(\R_+,\R)$. Notice now that the under hypotheses \ref{hyp:GD}, the identity 
\begin{equation}
\label{eq:NullWeight}
\INTDom{a(t,i,j) \phi(|x(t,i) - x(t,j)|)}{I_{\epsilon}}{j} = 0, 
\end{equation}
can only hold true provided that $a(t,i,j) = 0$ for $\LcalI$-almost every $j \in I_{\epsilon}$. In that case, \eqref{eq:ConvexHullIneq1} can be simply rewritten as 
\begin{equation}
\label{eq:ConvexHullIneq2}
\langle p , \partial_t x(t,i) \rangle = \epsilon \hspace{0.025cm} C_p(t), 
\end{equation}
for $\Lcal^1$-almost every $t \geq 0$. On the contrary, upon assuming that \eqref{eq:NullWeight} does not hold, one then has 
\begin{equation}
\label{eq:ConvexHullIneq3}
\langle p , \partial_t x(t,i)  \rangle = \bigg( \INTDom{a(t,i,j) \phi(|x(t,i)-x(t,j)|) }{I_{\epsilon}}{j} \bigg) \big \langle p , \bar{X}_{\epsilon}(t,i) - x(t,i) \big \rangle + \epsilon \hspace{0.025cm} C_p(t), 
\end{equation}
where the weighted barycenter point, defined by 
\begin{equation}
\bar{X}_{\epsilon}(t,i) = \INTDom{ \bigg( \frac{a(t,i,j) \phi(|x(t,i)-x(t,j)|) x(t,j)}{\INTDom{a(t,i,k) \phi(|x(t,i)-x(t,k)|)}{I_{\epsilon}}{k}} \bigg)}{I_{\epsilon}}{j}, 
\end{equation}
is an element of $\Cpazo_{\epsilon}(x(t))$ by the standard convexification property of the integral (see e.g. \cite[Chapter 5, Theorem 3]{Aubin1984}). Whence, using the characterisation \eqref{eq:ConvexHullCharac1} of the convex hull together with the definition of the set of indices $J_{\epsilon,p}(t)$ in the identity \eqref{eq:ConvexHullIneq3} finally yields 
\begin{equation}
\label{eq:ConvexHullIneq4}
\langle p , \partial_t x(t,i) \rangle \leq \epsilon \hspace{0.025cm} C_p(t), 
\end{equation}
for all times $t \geq 0$ and any $i \in J_{\epsilon,p}(t)$. Thus by combining \eqref{eq:ConvexHullIneq2} and \eqref{eq:ConvexHullIneq4} and performing an integration with respect to the time variable, we recover the pointwise inequality 
\begin{equation*}
\max_{i \in I_{\epsilon}} \langle p , x(t,i) \rangle \leq \max_{i \in I_{\epsilon}} \langle p , x^0(i) \rangle + \epsilon \INTSeg{C_p(s)}{s}{0}{t}, 
\end{equation*}
for all times $t \geq 0$. By letting $\epsilon \rightarrow 0^+$ and applying the convergence result of Lemma \ref{lem:SupConv}, it thus holds 
\begin{equation*}
\sup_{i \in I} \langle p , x(t,i) \rangle \leq \sup_{i \in I} \langle p , x^0(i) \rangle, 
\end{equation*}
and in particular, by choosing an element $y \in \Cpazo(x(t))$, one has that 
\begin{equation*}
\langle p , y \rangle \leq \sup_{i \in I} \langle p , x(t,i) \rangle \leq \sup_{i \in I} \langle p , x^0(i) \rangle, 
\end{equation*}
for every $p \in \R^d$ as a consequence of \eqref{eq:ConvexHullCharac2}. We have thus shown that $\Cpazo(x(t)) \subset \Cpazo(x^0)$ for all times $t \geq 0$, which concludes the proof of our claim. 
\end{proof}

\begin{rmk}[Performing differential estimates combining Scorza-Dragoni and Danskin theorems]
In the setting of discrete multi-agent models, the essential suprema over $I = [0,1]$ reduce to simple maxima of a finite collection of equi-locally Lipschitz mappings. Hence, these latter are always reached along \textnormal{piecewise constant} curves of indices $t \in \R_+ \mapsto i(t) \in \{1,\dots,N\}$, which greatly simplifies the analysis of differential estimates involving various geometric quantities of the system. In the graph limit framework, however, these suprema are a priori not attained, and while it may be possible to build piecewise constant curves of indices $t \in \R_+ \mapsto i(t) \in I$ whose images lie arbitrarily close to the desired quantity, the underlying construction would rely on careful adaptations of measurable selection theorems (see e.g. \cite[Theorem 8.3.1]{Aubin1990}) to the preimages of $\epsilon$-tubes built around curves of suprema.

In this context, the application of Scorza-Dragoni's theorem allows to consider well-defined maxima of continuous functions defined over compact sets instead of essential suprema. The time-derivatives of these latter can then be estimated in a fashion similar to that of classical finite-dimensional approaches, up to using Danskin's theorem for the differentiation part, and there only remains to properly account for the error terms stemming from the inner approximation while passing to the limit as $\epsilon \rightarrow 0^+$.
\end{rmk}

%%%%%%%%%%%%%%%%%%%%%%%%%%%%%%%%%%%%%%%%%%%%%%%%%%%%%%%%%%%%%%%%%%%%%%%%%%%%
%								NEW SECTION AHEAD						   %
%%%%%%%%%%%%%%%%%%%%%%%%%%%%%%%%%%%%%%%%%%%%%%%%%%%%%%%%%%%%%%%%%%%%%%%%%%%%

\section{Consensus formation via diameter estimates}
\label{section:DiamConsensus}
\setcounter{equation}{0} \renewcommand{\theequation}{\thesection.\arabic{equation}}

In this section, we study sufficient conditions for the exponential decay of the diameter of a graphon model. We start in Section \ref{subsection:Scrambling} by a preliminary discussion on scrambling coefficients, and study the formation of $L^{\infty}$-consensus in Section \ref{subsection:ConsensusDiam}.

%%%%%%%%%%%%%%%%%%%%%%%%%%%%%%%%%%%%%%%%%%%%%%%%%%%%%%%%%%%%%%%%%%%%%%%%%%%%

\subsection{Generalised scrambling coefficient for graphon models}
\label{subsection:Scrambling}

As already discussed in the Introduction, the quantity that usually dictates the asymptotic behaviour of the diameter of a multi-agent system is the so-called \textit{scrambling coefficient}. We point the reference monograph \cite{Seneta1979} for a detailed overview of its properties. 

%%%%%%%%%%%%%%%%%%%%%%%%%%%%%%%%%%%%%%%%%%%%%%%%%%%%%%%%%%%%%%%%%%%%%%%%%%%%

\paragraph*{Scrambling coefficients in finite and infinite dimension.} Given an adjacency matrix $\Ab_N := (a_{ij})_{1 \leq i,j \leq N} \in [0,1]^{N \times N}$ encoding the interactions of a discrete system of $N \geq 1$ agents, the scrambling coefficient is defined by
\begin{equation}
\label{eq:FiniteDimScramb}
\eta(\Ab_N) := \min_{1 \leq i,j \leq N} \frac{1}{N} \sum_{k=1}^N \min\{ a_{ik} , a_{jk} \}.
\end{equation}
Since the self-interactions of the agents are set to be $a_{ii} = 1$ for each $i \in \{1,\dots,N\}$, the latter can be equivalently rewritten as 
\begin{equation}
\label{eq:FiniteDimScraAis}
\eta(\Ab_N) = \min_{1 \leq i,j \leq N} \frac{1}{N} \bigg( \sum_{k=1, k \neq i,j}^N \min\{ a_{ik} , a_{jk} \} + a_{ij} + a_{ji} \bigg).
\end{equation}
What this formula expresses is that the scrambling coefficient of an interaction topology is positive if and only if, for every pair of agents $i,j \in \{1,\dots,N\}$, either $i$ and $j$ are communicating with one another, or there exists a third party agent $k \in \{1,\dots,N\} \setminus \{i,j\}$ which is followed by both $i$ and $j$. We provide in Figure \ref{fig:Scrambling} three simple examples of interaction topologies with positive scrambling.

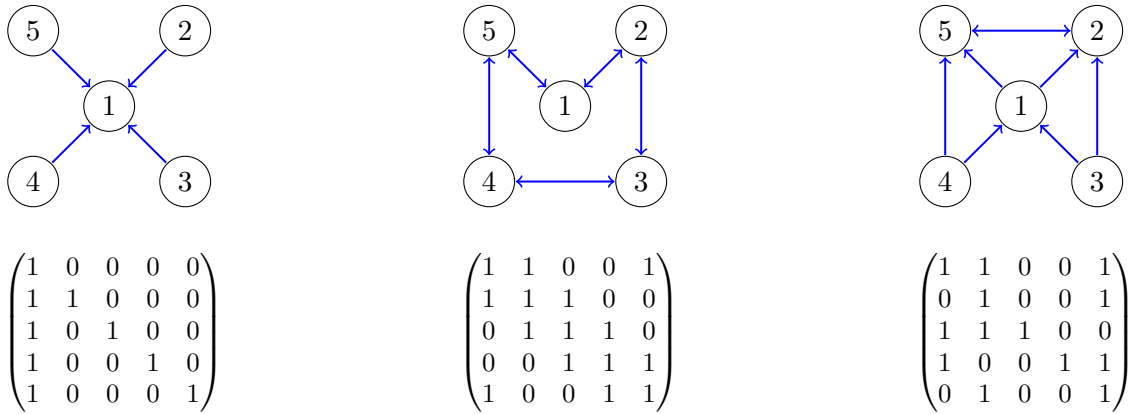
\begin{figure}[!ht]
\centering
\begin{tikzpicture}
% First figure
% Drawing the nodes
\draw (1,-1) node[shape=circle, draw] {$1$}; 
\draw (2,0) node[shape=circle, draw] {$2$};
\draw (2,-2) node[shape=circle, draw] {$3$};
\draw (0,0) node[shape=circle, draw] {$5$};
\draw (0,-2) node[shape=circle, draw] {$4$};
% Drawing the arrows
\draw[<-, line width = 0.75, blue] (0.75,-0.75)--(0.25,-0.25);
\draw[->, line width = 0.75, blue] (0.25,-1.75)--(0.75,-1.25);
\draw[<-, line width = 0.75, blue] (1.25,-0.75)--(1.75,-0.25);  
\draw[->, line width = 0.75, blue] (1.75,-1.75)--(1.25,-1.25);
% Drawing the matrix
\draw (1.05,-4) node {\small $
\begin{pmatrix} 
1 & 0 & 0 & 0 & 0 \\ 
1 & 1 & 0 & 0 & 0 \\
1 & 0 & 1 & 0 & 0 \\
1 & 0 & 0 & 1 & 0 \\
1 & 0 & 0 & 0 & 1 \\
\end{pmatrix}$};
% Second figure
% Drawing the nodes
\begin{scope}[xshift = 6cm]
\draw (1,-1) node[shape=circle, draw] {$1$}; 
\draw (2,0) node[shape=circle, draw] {$2$};
\draw (2,-2) node[shape=circle, draw] {$3$};
\draw (0,0) node[shape=circle, draw] {$5$};
\draw (0,-2) node[shape=circle, draw] {$4$};
% Drawing the arrows
\draw[<->, line width = 0.75, blue] (0.75,-0.75)--(0.25,-0.25);
\draw[<->, line width = 0.75, blue] (0,-0.35)--(0,-1.65); 
\draw[<->, line width = 0.75, blue] (1.25,-0.75)--(1.75,-0.25);  
\draw[<->, line width = 0.75, blue] (2,-0.35)--(2,-1.65); 
\draw[<->, line width = 0.75, blue] (1.65,-2)--(0.35,-2);
\draw (1.05,-4) node {\small $
\begin{pmatrix} 
1 & 1 & 0 & 0 & 1 \\ 
1 & 1 & 1 & 0 & 0 \\
0 & 1 & 1 & 1 & 0 \\
0 & 0 & 1 & 1 & 1 \\
1 & 0 & 0 & 1 & 1 \\
\end{pmatrix}$};
\end{scope}
% Third figure
% Drawing the nodes
\begin{scope}[xshift = 12cm]
\draw (1,-1) node[shape=circle, draw] {$1$}; 
\draw (2,0) node[shape=circle, draw] {$2$};
\draw (2,-2) node[shape=circle, draw] {$3$};
\draw (0,0) node[shape=circle, draw] {$5$};
\draw (0,-2) node[shape=circle, draw] {$4$};
% Drawing the arrows
\draw[->, line width = 0.75, blue] (0.75,-0.75)--(0.25,-0.25);
\draw[<-, line width = 0.75, blue] (0,-0.35)--(0,-1.65); 
\draw[->, line width = 0.75, blue] (0.25,-1.75)--(0.75,-1.25);
\draw[->, line width = 0.75, blue] (1.25,-0.75)--(1.75,-0.25);  
\draw[<-, line width = 0.75, blue] (2,-0.35)--(2,-1.65); 
\draw[->, line width = 0.75, blue] (1.75,-1.75)--(1.25,-1.25);
\draw[<->, line width = 0.75, blue] (1.65,0)--(0.35,0);
\draw (1.05,-4) node {\small $
\begin{pmatrix} 
1 & 1 & 0 & 0 & 1 \\ 
0 & 1 & 0 & 0 & 1 \\
1 & 1 & 1 & 0 & 0 \\
1 & 0 & 0 & 1 & 1 \\
0 & 1 & 0 & 0 & 1 \\
\end{pmatrix}$};
\end{scope}
\end{tikzpicture}
\caption{{\small \textit{Three examples of interaction topologies with positive scrambling for graphs with $N = 5$ vertices. The interaction graphs are represented on top, and the corresponding adjacency matrices right below.}}}
\label{fig:Scrambling}
\end{figure}

Given an adjacency operator $\Apazo \in \Lpazo(L^2(I,\R^d))$ with kernel $a \in L^{\infty}(I \times I,[0,1])$, we introduce the following infinite-dimensional generalisation of the scrambling coefficient
\begin{equation}
\label{eq:InfiniteDimScramb}
\eta(\Apazo) := \inf_{i,j \in I} \INTDom{\min \{ a(i,k) , a(j,k) \}}{I}{k}. 
\end{equation}
As we shall see in Theorem \ref{thm:DiamContraction} below, this macroscopic quantity does indeed quantify the decay of the diameter of a graphon dynamics of the form \eqref{eq:GraphonDynamics}. It is also worth noting that if $\Apazo_N \in \Lpazo(L^2(I,\R^d))$ is the adjacency operator associated with a piecewise constant kernel $a^N \in L^{\infty}(I \times I,[0,1])$ of the form \eqref{eq:PiecewiseWeight} for some $N \geq 1$, then $\eta(\Apazo_N) = \eta(\Ab_N)$ as in \eqref{eq:FiniteDimScramb}, where $\Ab_N \in \R^{N \times N}$ is the adjacency matrix of the underlying digraph. This shows that the definition proposed in \eqref{eq:InfiniteDimScramb} is indeed a proper generalisation of the usual notion of scrambling coefficient.

\paragraph*{The role of stochasticity in the literature.} In most of the existing works studying the asymptotic properties of the diameter in multi-agent dynamics (see e.g. \cite{MPPD,Motsch2011,Motsch2014} and references therein), the proof of the decay estimates explicitly rely on the assumption that the adjacency matrices are \textit{stochastic}, namely that they satisfy
\begin{equation*}
\frac{1}{N} \sum_{j=1}^N a_{ij} = 1, 
\end{equation*}
for each $i \in \{1,\dots,N\}$. Indeed, observing that a multi-agent dynamics of the form 
\begin{equation*}
\dot x_i(t) = \frac{1}{N} \sum_{j=1}^N a_{ij} (x_j(t) - x_i(t)), 
\end{equation*}
is independent of the values of the diagonal coefficients $(a_{11},\dots,a_{NN})$, one can equivalently consider that any of its solutions is driven by the stochastic interaction weights $(\tilde{a}_{ij})_{1 \leq i,j \leq N}$, defined by 
\begin{equation}
\label{eq:StoReparam}
\tilde{a}_{ii} := N \, - \sum_{k =1, k \neq i}^N a_{ik} \qquad \text{and} \qquad \tilde{a}_{ij} := a_{ij}, 
\end{equation}
for every $i,j \in \{1,\dots,N\}$. In the following proposition, we show that the scrambling coefficient of an interaction topology does not change when the communication weights are redefined using \eqref{eq:StoReparam}. While this result is not difficult to establish, it is crucial in order to ensure that the latter transformation does not modify important qualitative properties of the dynamics.

\begin{prop}[Invariance of scrambling coefficients under stochastic reparametrisation]
Let $\Ab_N := (a_{ij})_{1 \leq i,j \leq N} \in [0,1]^{N \times N}$ be the adjacency matrix of an interaction graph, and $\tilde{\Ab}_N := (\tilde{a}_{ij})_{1 \leq i,j \leq N}$ be its stochastic reparametrisation obtained via the transformation \eqref{eq:StoReparam}. Then, $\eta(\tilde{\Ab}_N) = \eta(\Ab_N)$. 
\end{prop}

\begin{proof}
Observe first that since $(a_{ij})_{1 \leq i,j \leq N} \in [0,1]^{N \times N}$, it necessarily holds  
\begin{equation*}
a_{ij} \leq 1 \leq N \, - \sum_{k=1, k \neq i}^N a_{ik} \qquad \text{and} \qquad  a_{ji} \leq 1 \leq N \, - \sum_{k=1, k \neq j}^N a_{jk}, 
\end{equation*}
for every $i,j \in \{1,\dots,N\}$. Thus by definition \eqref{eq:StoReparam} of the coefficients of the modified adjacency matrix $\tilde{\Ab}_N$, one has that
\begin{equation*}
\begin{aligned}
\eta(\tilde{\Ab}_N) & = \min_{1 \leq i,j \leq N} \frac{1}{N} \Bigg( \sum_{k=1, k \neq i,j}^N \min\{ a_{ik} , a_{jk} \} + \min \bigg\{ a_{ij} , N \, - \mathsmaller{\sum\limits_{k=1,k \neq i}^N a_{ik}} \bigg\} \\
& \hspace{6.15cm} + \min \bigg\{N \, - \mathsmaller{\sum\limits_{k=1,k \neq j}^N a_{jk}} ,  a_{ji} \bigg\} \Bigg) \\
& = \min_{1 \leq i,j \leq N} \frac{1}{N} \bigg( \sum_{k=1, k \neq i,j}^N \min\{ a_{ik} , a_{jk} \} + a_{ij} + a_{ji} \bigg) \\
& = \eta(\Ab_N), 
\end{aligned}
\end{equation*}
where we used the alternative expression \eqref{eq:FiniteDimScraAis} of the scrambling coefficient.
\end{proof}

\begin{rmk}[Obstructions to stochastic reparametrisation for graphons]
While the approach described above is very useful in order to simplify computations when considering finite-dimensional multi-agent systems, it is not readily applicable to graphon dynamics. Indeed, one can easily check that the transformation of the adjacency matrix displayed in \eqref{eq:StoReparam} does not admit a well-defined pointwise limit as $N \rightarrow +\infty$. On the other hand, it can be shown that the corresponding piecewise constant interaction maps defined as in \eqref{eq:PiecewiseWeight} would converge in the sense of distributions towards the sum of a kernel $a \in L^{\infty}(I \times I,[0,1])$ and a Dirac measure supported on the diagonal. Therefore, the resulting right-hand side falls outside the scope of classical graphon dynamics, in which the adjacency operator is simply given by the integral of an $\LcalI$-measurable kernel. 
\end{rmk}

%%%%%%%%%%%%%%%%%%%%%%%%%%%%%%%%%%%%%%%%%%%%%%%%%%%%%%%%%%%%%%%%%%%%%%%%%%%%

\subsection{Exponential consensus formation for topologies with persistent scramblings}
\label{subsection:ConsensusDiam}

In Theorem \ref{thm:DiamContraction} below, we derive a general contraction estimate on the diameter of graphon dynamics, expressed in terms of the generalised scrambling coefficient. In what follows for $\Lcal^1$-almost every $t \geq 0$, we denote by $\Apazo(t) \in \Lpazo(L^2(I,\R^d))$ the adjacency operator with kernel $a(t) \in L^{\infty}(I \times I,[0,1])$. 

\begin{thm}[A general contraction estimate for the diameter]
\label{thm:DiamContraction}
Fix a radius $R > 0$ and an initial datum $x^0 \in L^{\infty}(I,\R^d)$ satisfying $\NormL{x^0}{\infty}{I,\R^d} \leq R$. Moreover, assume that hypotheses \ref{hyp:GD} hold, and consider a solution $x(\cdot) \in \Lip_{\loc}(\R_+,L^2(I,\R^d))$ of \eqref{eq:DiscGraphonDyn}.

Then for all times $t \geq 0$, it holds that
\begin{equation*}
\Dpazo(t) \leq \Dpazo(0) \exp \bigg( - \gamma_R \INTSeg{\eta(\Apazo(s))}{s}{0}{t} \bigg), 
\end{equation*}
with $\gamma_R := \min_{r \in [0,2R]} \phi(r)$, and where the scrambling coefficient $\eta(\Apazo(s))$ generated by the communication weights is defined as in \eqref{eq:InfiniteDimScramb} for $\Lcal^1$-almost every $s \geq 0$.  
\end{thm}

Before proving Theorem \ref{thm:DiamContraction}, we establish a useful geometric lemma which allows for the  estimation of the time-derivative of the diameter without requiring the stochasticity of the adjacency operator. 

\begin{lem}[A scalar product inequality]
\label{lem:ScalarIneq}
Let $J \subset I$ be a closed set, $x \in C^0(J,\R^d)$ and $i,j \in J$ be a pair of indices such that $\max_{k,l \in J} |x(k) - x(l)| = |x(i) - x(j)| $. Then, one has that
\begin{equation}
\label{eq:ScalarProdIneq}
\begin{aligned}
\langle x(j) , x(i) - x(j) \rangle & = \min_{k \in J} \langle x(k) , x(i) - x(j) \rangle \\
& \leq \max_{k \in J} \langle x(k) , x(i) - x(j) \rangle = \langle x(i) , x(i) - x(j) \rangle.
\end{aligned}
\end{equation}
\end{lem}

\begin{proof}
It is a classical result in geometric analysis that the diameter of any subset of $\R^d$ coincides with the diameter of its convex hull.  Thus, defining $x_{\lambda}(k) := (1-\lambda)x(i) + \lambda x(k)$ for some $k \in J$ and every $\lambda \in [0,1]$, it necessarily holds
\begin{equation}
\label{eq:DiamInterpIneq}
|x_{\lambda}(k) - x(j) |^2 \leq |x(i) - x(j)|^2,
\end{equation}
by construction of the indices $i,j \in J$. An easy computation then shows that
\begin{equation*}
\begin{aligned}
|x_{\lambda}(k) - x(j) |^2 - |x(i) - x(j)|^2 & = \Big\langle (x_{\lambda}(k) - x(j)) - (x(i) - x(j)) \, , \, (x_{\lambda}(k) - x(j)) + (x(i) - x(j))   \Big\rangle \\
& = \lambda^2 |x(i) - x(k)|^2 + 2 \lambda \big\langle x(k) - x(i) , x(i) - x(j)  \big\rangle, 
\end{aligned}
\end{equation*}
so that \eqref{eq:DiamInterpIneq} can be equivalently reformulated as 
\begin{equation}
\label{eq:DiamInterpIneqBis}
\lambda^2 |x(i) - x(k)|^2 + 2 \lambda \big\langle x(k) - x(i) , x(i) - x(j)  \big\rangle \leq 0, 
\end{equation}
for every $\lambda \in [0,1]$. Dividing by $\lambda \in (0,1]$ and letting $\lambda \to 0^+$ in \eqref{eq:DiamInterpIneqBis}, we therefore obtain 
\begin{equation*}
\langle x(k) , x(i) - x(j) \rangle \leq \langle x(i) , x(i) - x(j) \rangle, 
\end{equation*}
for every $k \in J$, which is equivalent to stating that $\langle x(i) , x(i) - x(j) \rangle = \max_{k \in J} \langle x(k) , x(i) - x(j) \rangle$. The other inequality in \eqref{eq:ScalarProdIneq} can be proven by repeating the same arguments.  
\end{proof}

\begin{proof}[Proof of Theorem \ref{thm:DiamContraction}]
Let $x \in \Lip_{\loc}(\R_+,L^2(I,\R^d))$ be a solution of \eqref{eq:DiscGraphonDyn}, and fix a parameter $\epsilon > 0$. Then by Scorza-Dragoni's theorem (see Theorem \ref{thm:ScorzaDragoni} above), there exists a compact set $I_{\epsilon} \subset I$ with $\LcalI(I \setminus I_{\epsilon}) < \epsilon$ such that the restriction $x : \R_+ \times I_{\epsilon} \rightarrow \R^d$ is a continuous map. We then consider the restricted diameters, defined by 
\begin{equation*}
\Dpazo_{\epsilon}(t) := \max_{i,j \in I_{\epsilon}} |x(t,i) - x(t,j)|
\end{equation*}
for all times $t \geq 0$, and introduce the nonempty subset of pairs of indices $\Pi_{\epsilon}(t) \subset I_{\epsilon} \times I_{\epsilon}$ for which this maximum is reached. 

Observe first that $\Dpazo_{\epsilon}(\cdot)$ is defined as the pointwise maximum of a family of equi-locally Lipschitz mappings, and is therefore a locally Lipschitz function. Hence by Rademacher's theorem (see e.g. \cite[Theorem 3.2]{EvansGariepy}), it is differentiable $\Lcal^1$-almost everywhere, and it holds by Danskin's theorem (see Theorem \ref{thm:Danskin} above) that  
\begin{equation}
\label{eq:DanskinDer}
\tfrac{1}{2} \tderv{}{t} \Dpazo_{\epsilon}(t)^2 = \max_{(i,j) \in \Pi_{\epsilon}(t)} \big\langle \partial_t x(t,i) - \partial_t x(t,j) , x(t,i) - x(t,j) \big\rangle.
\end{equation}
We now fix a pair of indices $(i,j) \in \Pi_{\epsilon}(t)$, and aim at evaluating the right-hand side of \eqref{eq:DanskinDer}. Recalling that $x(\cdot)$ solves \eqref{eq:DiscGraphonDyn}, one has that
\begin{equation}
\label{eq:DiamGraphon1}
\begin{aligned}
& \big\langle \partial_t x(t,i) - \partial_t x(t,j) , x(t,i) - x(t,j) \big\rangle \\
& = \INTDom{a(t,i,k) \phi(|x(t,i) - x(t,k)|) \big\langle x(t,k) - x(t,i) , x(t,i) - x(t,j) \big\rangle}{I}{k}  \\
& \hspace{0.45cm} - \INTDom{a(t,j,k) \phi(|x(t,j) - x(t,k)|) \big\langle x(t,k) - x(t,j) , x(t,i) - x(t,j)  \big\rangle}{I}{k} \\
& = \INTDom{a(t,i,k) \phi(|x(t,i) - x(t,k)|) \big\langle x(t,k) - x(t,i) , x(t,i) - x(t,j) \big\rangle}{I_{\epsilon}}{k}  \\
& \hspace{0.45cm} - \INTDom{a(t,j,k) \phi(|x(t,j) - x(t,k)|) \big\langle x(t,k) - x(t,j) , x(t,i) - x(t,j)  \big\rangle}{I_{\epsilon}}{k} + \epsilon C(t),
\end{aligned}
\end{equation}
where $C(\cdot) \in L^{\infty}_{\loc}(\R_+,\R)$. Upon defining the quantity
\begin{equation*}
\eta_{ij}(t,k) := \min \Big\{ a(t,i,k) \phi(|x(t,i)-x(t,k)|) \, , \, a(t,j,k) \phi(|x(t,j)-x(t,k)|) \Big\}, 
\end{equation*}
for $\Lcal^1$-almost every $t \geq 0$ and $\LcalI$-almost every $k \in I$, one can rewrite \eqref{eq:DiamGraphon1} as 
\begin{equation}
\label{eq:DiamGraphon1Bis}
\begin{aligned}
& \big\langle \partial_t x(t,i) - \partial_t x(t,j) , x(t,i) - x(t,j) \big\rangle \\
& = \INTDom{a(t,i,k) \phi(|x(t,i) - x(t,k)|) \big\langle x(t,k) , x(t,i) - x(t,j) \big\rangle}{I_{\epsilon}}{k} \\
& \hspace{1cm} - \INTDom{a(t,j,k) \phi(|x(t,j) - x(t,k)|) \big\langle x(t,k) , x(t,i) - x(t,j)  \big\rangle}{I_{\epsilon}}{k} \\
& \hspace{1cm} - \Bigg( \INTDom{a(t,i,k) \phi(|x(t,i) - x(t,k)|)}{I_{\epsilon}}{k}  \Bigg) \big\langle x(t,i) , x(t,i) - x(t,j) \big\rangle \\
& \hspace{1cm} + \Bigg( \INTDom{a(t,j,k) \phi(|x(t,j) - x(t,k)|) }{I_{\epsilon}}{k} \Bigg) \big\langle x(t,j) , x(t,i) - x(t,j) \big\rangle + \epsilon C(t) \\
& = \INTDom{ \Big( a(t,i,k) \phi(|x(t,i) - x(t,k)|) - \eta_{ij}(t,k) \Big) \big\langle x(t,k) , x(t,i) - x(t,j) \big\rangle}{I_{\epsilon}}{k} \\
& \hspace{1cm} + \INTDom{ \Big( \eta_{ij}(t,k) - a(t,j,k) \phi(|x(t,j) - x(t,k)|) \Big) \big\langle x(t,k) , x(t,i) - x(t,j) \big\rangle}{I_{\epsilon}}{k} \\
& \hspace{1cm} - \bigg( \INTDom{a(t,i,k) \phi(|x(t,i) - x(t,k)|)}{I_{\epsilon}}{k} \bigg) \langle x(t,i) , x(t,i) - x(t,j) \rangle \\ 
& \hspace{1cm} + \bigg( \INTDom{a(t,j,k) \phi(|x(t,j) - x(t,k)|) }{I_{\epsilon}}{k} \bigg) \langle x(t,j) , x(t,i) - x(t,j) \rangle + \epsilon C(t), 
\end{aligned}
\end{equation}
for $\Lcal^1$-almost every $t \geq 0$. Observe now that by construction of the parameters $\eta_{ij}(t,k)$, the first two terms in the right-hand side of \eqref{eq:DiamGraphon1Bis} can be estimated as follows
\begin{equation}
\label{eq:DiamGraphon2}
\begin{aligned}
& \INTDom{ \Big( a(t,i,k) \phi(|x(t,i) - x(t,k)|) - \eta_{ij}(t,k) \Big) \big\langle x(t,k) , x(t,i) - x(t,j) \big\rangle}{I_{\epsilon}}{k} \\
& \hspace{3.3cm} + \INTDom{ \Big( \eta_{ij}(t,k) - a(t,j,k) \phi(|x(t,j) - x(t,k)|) \Big) \big\langle x(t,k) , x(t,i) - x(t,j) \big\rangle}{I}{k} \\
& \leq  \bigg( \INTDom{ \Big( a(t,i,k) \phi(|x(t,i) - x(t,k)|) - \eta_{ij}(t,k) \Big)}{I_{\epsilon}}{k} \bigg) \max_{k \in I_{\epsilon}} \big\langle x(t,k) , x(t,i) - x(t,j) \big\rangle \\
& \hspace{3.3cm} - \bigg( \INTDom{ \Big( a(t,j,k) \phi(|x(t,j) - x(t,k)|) - \eta_{ij}(t,k) \Big)}{I_{\epsilon}}{k} \bigg) \min_{k \in I_{\epsilon}} \big\langle x(t,k) , x(t,i) - x(t,j) \big\rangle \\
& \leq \bigg( - \INTDom{\eta_{ij}(t,k)}{I_{\epsilon}}{k} \bigg) \max_{k,l \in I_{\epsilon}} \big\langle x(t,k) - x(t,l) , x(t,i) - x(t,j) \big\rangle \\
& \hspace{0.45cm} + \bigg( \INTDom{a(t,i,k) \phi(|x(t,i) - x(t,k)|) }{I_{\epsilon}}{k} \bigg) \max_{k \in I_{\epsilon}} \big\langle x(t,k), x(t,i) - x(t,j) \big\rangle \\
& \hspace{0.45cm} - \bigg( \INTDom{a(t,j,k) \phi(|x(t,j) - x(t,k)|)}{I_{\epsilon}}{k} \bigg) \min_{k \in I_{\epsilon}} \big\langle x(t,k), x(t,i) - x(t,j) \big\rangle \\
& \leq \bigg( - \inf_{i,j \in I} \INTDom{\eta_{ij}(t,k)}{I}{k} + c_{\phi} \epsilon \bigg)  \Dpazo_{\epsilon}(t)^2 \\
& \hspace{0.45cm} + \bigg( \INTDom{a(t,i,k) \phi(|x(t,i) - x(t,k)|)}{I_{\epsilon}}{k} \bigg) \max_{k \in I_{\epsilon}} \big\langle x(t,k), x(t,i) - x(t,j) \big\rangle \\
& \hspace{0.45cm} - \bigg( \INTDom{a(t,j,k) \phi(|x(t,j) - x(t,k)|)}{I_{\epsilon}}{k} \bigg) \min_{k \in I_{\epsilon}} \big\langle x(t,k), x(t,i) - x(t,j) \big\rangle, 
\end{aligned}
\end{equation}
for $\Lcal^1$-almost every $t \geq 0$, where we recall that $c_{\phi} = \sup_{r \in \R_+} \phi(r) < +\infty$ is finite as a consequence of \ref{hyp:GD}-$(i)$. Remark now that by Lemma \ref{lem:ScalarIneq}, it further holds
\begin{equation}
\label{eq:DiamGraphon3}
\begin{aligned}
\langle x(t,j) , x(t,i) - x(t,j) \rangle & = \min_{k \in I_{\epsilon}} \big\langle x(t,k), x(t,i) - x(t,j) \big\rangle \\
& \leq \max_{k \in I_{\epsilon}} \big\langle x(t,k), x(t,i) - x(t,j) \big\rangle = \langle x(t,i) , x(t,i) - x(t,j) \rangle,
\end{aligned}
\end{equation}
for $\Lcal^1$-almost every $t \geq 0$ and $\LcalI$-almost every $k \in I_{\epsilon}$, where we used the fact that $(i,j) \in \Pi_{\epsilon}(t)$. Hence, by plugging  \eqref{eq:DiamGraphon1Bis}, \eqref{eq:DiamGraphon2} and \eqref{eq:DiamGraphon3} into \eqref{eq:DanskinDer} while observing that 
\begin{equation*}
\eta_{ij}(t,k) \geq \gamma_R \min\big\{ a(t,i,k) , a(t,j,k) \big\}, 
\end{equation*}
and also that the former series of estimates hold true for every pair of indices $(i,j) \in \Pi_{\epsilon}(t)$, we recover the differential inequality  
\begin{equation*}
\tfrac{1}{2} \tderv{}{t} \Dpazo_{\epsilon}(t)^2 \leq \Big( -\gamma_R \hspace{0.05cm} \eta(\Apazo(t)) + c_{\phi} \epsilon \Big)  \Dpazo_{\epsilon}(t)^2 + \epsilon C(t), 
\end{equation*}
for $\Lcal^1$-almost every $t \geq 0$. By applying the differential version of Gr\"onwall's lemma, we then deduce 
\begin{equation}
\label{eq:DepsDecay}
\begin{aligned}
\Dpazo_{\epsilon}(t)^2 & \leq \Dpazo_{\epsilon}(0)^2 \exp \bigg( - 2\gamma_R \INTSeg{\eta(\Apazo(s))}{s}{0}{t} + 2 c_{\phi} \epsilon \hspace{0.05cm} t \bigg) \\
& \hspace{0.45cm}+ 2 \epsilon \INTSeg{C(s) \exp \bigg( -2\gamma_R \INTSeg{\eta(\Apazo(\sigma))}{\sigma}{s}{t} - 2 c_{\phi} \epsilon \hspace{0.05cm} (t-s) \bigg)}{s}{0}{t}, 
\end{aligned}
\end{equation}
for all times $t \geq 0$. Recall now that by Lemma \ref{lem:SupConv}, we have the pointwise convergence 
\begin{equation*}
\Dpazo_{\epsilon}(t) ~\underset{\epsilon \rightarrow 0^+}{\longrightarrow}~ \Dpazo(t), 
\end{equation*}
and thus by passing to the limit as $\epsilon \rightarrow 0^+$ in \eqref{eq:DepsDecay}, we recover the desired exponential decay estimate
\begin{equation*}
\Dpazo(t) \leq \Dpazo(0) \exp \bigg( - \gamma_R \INTSeg{\eta(\Apazo(s))}{s}{0}{t} \bigg), 
\end{equation*}
for all times $t \geq 0$, which concludes the proof of Theorem \ref{thm:DiamContraction}.
\end{proof}

\begin{rmk}[Comparing our approach with previous diameter estimates in the literature]
As mentioned above, the existing proofs of diameter contractions for finite-dimensional multi-agent systems involving scrambling coefficient (see e.g. \cite{Motsch2011,Motsch2014}) presuppose that the adjacency matrices $\Ab_N(t)$ of the system are stochastic for all times $t \geq 0$. In practice, this simplifying assumption allows to write meaningful differential estimates on the quantity $t \in \R_+ \mapsto |x_i(t)-x_j(t)| \in \R_+$ for every pair of indices $i,j \in \{1,\dots,N\}$. Hereinabove on the contrary, we only estimate the derivative of this quantity for pairs of indices which realise the (relaxed) diameter, and use the geometric information of Lemma \ref{lem:ScalarIneq} to take care of the terms that were usually simplified using the stochasticity assumption. In this sense, our proof strategy provides an alternative and perhaps more general diameter estimate for multi-agent systems, as it is more natural and valid for both finite- and infinite-dimensional models. 
\end{rmk}

Based on this general diameter estimate, it is possible to establish the formation of $L^{\infty}$-consensus in the case where the scrambling coefficients are persistent in a suitable sense. 

\begin{thm}[Exponential $L^{\infty}$-consensus formation under persistent scrambling]
\label{thm:ConsensusDiam}
Assume that the hypotheses of Theorem \ref{thm:DiamContraction} are satisfied, and suppose that there exists a pair of coefficients $(\tau,\mu) \in \R_+^* \times (0,1]$ such that the  \textnormal{persistence condition}
\begin{equation}
\label{eq:PersistenceScramb}
\frac{1}{\tau} \INTSeg{\eta(\Apazo(s))}{s}{t}{t+\tau} \geq \mu, 
\end{equation}
holds for all times $t \geq 0$. 

Then, there exists an element $x^{\infty} \in \Cpazo(x^0)$ and constants $\alpha,\gamma > 0$ depending only on $(\phi(\cdot),R,\tau,\mu)$ such that
\begin{equation}
\label{eq:ThmConsensusLinfty}
\NormL{x(t)- x^{\infty}}{\infty}{I,\R^d}  \leq \alpha \Dpazo(0) \exp \big( \hspace{-0.1cm} -\gamma \mu \hspace{0.01cm} t \big), 
\end{equation}
for all times $t \geq 0$. In particular, the solutions of \eqref{eq:DiscGraphonDyn} exponentially converge to consensus in the $L^{\infty}(I,\R^d)$-norm topology. 
\end{thm}

\begin{proof}
Start by fixing an element $i \in I$, and observe that for every $0 \leq t_1 \leq t_2 < +\infty$, one has that
\begin{equation}
\label{eq:IneqPersistence1}
\begin{aligned}
|x(t_2,i) - x(t_1,i)| & \leq \INTSeg{\INTDom{a(s,i,j) \phi(|x(s,i)-x(s,j)|)|x(s,j)-x(s,i)|}{I}{j}}{s}{t_1}{t_2} \\
& \leq c_{\phi} \INTSeg{\Dpazo(s)}{s}{t_1}{t_2} \\
& \leq c_{\phi} \INTSeg{\Dpazo(0) \exp \bigg( - \gamma_R \INTSeg{\eta(\Apazo(\sigma))}{\sigma}{0}{s} \bigg)}{s}{t_1}{+\infty}, 
\end{aligned}
\end{equation}
where we recall that $c_{\phi} = \sup_{r \in \R_+} \phi(r) < +\infty$. Then, as a consequence of the persistence condition \eqref{eq:PersistenceScramb}, it further holds that
\begin{equation}
\label{eq:IneqPersistence2}
\begin{aligned}
\INTSeg{\eta(\Apazo(\sigma))}{\sigma}{0}{s} & \geq \INTSeg{\eta(\Apazo(\sigma))}{\sigma}{0}{\lfloor s / \tau \rfloor \tau} \\
& \geq \sum_{k=0}^{\lfloor s / \tau \rfloor -1} \INTSeg{\eta(\Apazo(\sigma))}{\sigma}{k \tau}{(k+1) \tau} ~\geq~ \mu (s-\tau), 
\end{aligned}
\end{equation}
for all times $s \in [t_1,t_2]$, where $\lfloor \bullet \rfloor$ denotes the lower integer part of a real number. Thus, by plugging \eqref{eq:IneqPersistence2} into \eqref{eq:IneqPersistence1}, we obtain
\begin{equation}
\label{eq:IneqPersistence3}
\begin{aligned}
|x(t_2,i) - x(t_1,i)| & \leq c_{\phi} \INTSeg{\Dpazo(0) \exp \Big( - \gamma_R \mu (s-\tau) \Big)}{s}{t_1}{+\infty} ~\underset{t_1,t_2 \rightarrow +\infty}{\longrightarrow}~ 0, 
\end{aligned}
\end{equation}
which yields the existence of a map $x^{\infty} \in L^{\infty}(I,\R^d)$ such that $x(t,i) \rightarrow x^{\infty}(i)$ as $t \rightarrow +\infty$ for $\LcalI$-almost every $i \in I$. Upon noticing that the convergence in \eqref{eq:IneqPersistence3} is uniform with respect to $i \in I$, we further have 
\begin{equation}
\label{eq:ConsensusFormationProof}
\NormL{x(t) - x^{\infty}}{\infty}{I,\R^d} ~\underset{t \rightarrow +\infty}{\longrightarrow}~ 0.
\end{equation}
We now choose any pair of indices $i,j \in I$, and remark that for all times $t \geq 0$, it holds
\begin{equation*}
|x^{\infty}(i) - x^{\infty}(j)| \leq |x^{\infty}(i) - x(t,i)| + |x(t,j) - x^{\infty}(j)| + \Dpazo(t).
\end{equation*}
By combining the diameter contraction of Theorem \ref{thm:DiamContraction} together with the convergence results of  \eqref{eq:IneqPersistence3} and \eqref{eq:ConsensusFormationProof}, we obtain that $x^{\infty}(i) = x^{\infty}(j)$ for $\LcalI$-almost every $i,j \in I$. Therefore, the map $x^{\infty}$ is in fact constant, and by Proposition \ref{prop:ConvexHull} it also belongs to the initial convex hull $\Cpazo(x^0)$ of the system.

Observing that $x^{\infty} \in \Cpazo(x^0)$, and using the fact that the diameter of a set coincides with that of its convex hull together with the series of estimates \eqref{eq:IneqPersistence2}, we finally get 
\begin{equation*}
|x(t,i) - x^{\infty}| \leq \Dpazo(t) \leq \Dpazo(0) \exp \Big( - \gamma_R \mu (t - \tau) \Big),  
\end{equation*}
for all times $t \geq 0$ and $\LcalI$-almost every $i \in I$. Taking the essential supremum over $I$ in the previous expression then yields the exponential contraction estimate \eqref{eq:ThmConsensusLinfty} with $\alpha := \exp(\gamma_R \, \mu \tau)$ and $\gamma := \gamma_R$. 
\end{proof}

\begin{rmk}[On the nature of the asymptotic consensus point]
Observe that, in general, the emerging consensus $x^{\infty} \in \Cpazo(x^0)$ -- whose existence follows from Theorem \ref{thm:ConsensusDiam} -- does not necessarily coincide with the initial barycenter $\bar{x}^0 := \INTDom{x^0(i)}{I}{i}$ of the system. This is due to the fact that the latter may not be an intrinsic constant of motion for asymmetric topologies (see also Section \ref{subsubsection:StronglyConnected} below). 
\end{rmk}

%%%%%%%%%%%%%%%%%%%%%%%%%%%%%%%%%%%%%%%%%%%%%%%%%%%%%%%%%%%%%%%%%%%%%%%%%%%%
%								NEW SECTION AHEAD						   %
%%%%%%%%%%%%%%%%%%%%%%%%%%%%%%%%%%%%%%%%%%%%%%%%%%%%%%%%%%%%%%%%%%%%%%%%%%%%

\section{Consensus formation via energy estimates}
\label{section:ConsensusL2}
\setcounter{equation}{0} \renewcommand{\theequation}{\thesection.\arabic{equation}}

In this section, we investigate another family of sufficient conditions ensuring the asymptotic formation of $L^2$-consensus for time-dependent graphon models. We first start in Section \ref{subsection:Algebraic} by a general discussion on the notion of \textit{algebraic connectivity}, along with a generalisation of this classical object to graphon models. We then study in Section \ref{subsection:L2Consensus} the exponential convergence to consensus for topologies that are symmetric, balanced, and made of a disjoint union of strongly connected components.

%%%%%%%%%%%%%%%%%%%%%%%%%%%%%%%%%%%%%%%%%%%%%%%%%%%%%%%%%%%%%%%%%%%%%%%%%%%%

\subsection{Generalised algebraic connectivity for graphon models}
\label{subsection:Algebraic}

In this section, we start by a discussion on the notion of algebraic connectivity for finite-dimensional collective dynamics, and then propose a natural generalisation for graphon models. 

\paragraph*{Graph-Laplacians in the finite-dimensional setting.} 

It is widely known (see e.g. \cite{Moreau2005,Olfati2004} along with the monographs \cite{BeardRen,Egerstedt2010}) that in order to study the formation of consensus with respect to the $\ell^2$-norm for discrete systems of the form \eqref{eq:DiscreteDynamics}, it is often insightful to introduce the \textit{graph-Laplacian operator} $\Lb_N : \R_+ \times (\R^d)^N \rightarrow \Lpazo((\R^d)^N)$ associated with the interaction topology, defined by 
\begin{equation}
\big( \Lb_N(t,\xb) \yb \big)_i := \frac{1}{N} \sum_{j=1}^N a_{ij}(t) \phi(|x_i - x_j|) (y_i - y_j) \qquad \text{for each $i \in \{1,\dots,N\}$}, 
\end{equation}
for $\Lcal^1$-almost every $t \in \R_+$ and all $\xb,\yb \in (\R^d)^N$. This permits to rewrite \eqref{eq:DiscreteDynamics} as the semilinear evolution equation
\begin{equation}
\label{eq:DiscreteLapDyn}
\dot \xb(t) = - \Lb_N(t,\xb(t)) \xb(t), \qquad \xb(0) = \xb^0, 
\end{equation}
posed in the configuration space $(\R^d)^N$, and to study the convergence to consensus in terms of the eigenvalues of the matrices $\Lb_N(t,\xb(t))$ using techniques of stability theory (see e.g. \cite{Narendra1989}). A first key observation about graph-Laplacians is that their kernel is never trivial, as it can be checked that $\Lb_N(t,\xb) \yb = 0$ whenever $\yb \in (\R^d)^N$ belongs to the so-called \textit{consensus manifold}, that is defined by 
\begin{equation}
\label{eq:ConsensusManifoldDisc}
\Ccal_N := \Big\{ \xb \in (\R^d)^N ~\text{s.t.}~ x_1 = \dots = x_N \Big\}.
\end{equation}
This implies in particular that $0$ is always a nontrivial eigenvalue of $\Lb_N(t,\xb)$ with multiplicity at least $d \geq 1$. Because the consensus manifold is also the set of equilibrium points of the dynamics \eqref{eq:DiscreteDynamics}, the convergence of a solution $\xb(\cdot) \in \Lip_{\loc}(\R_+,(\R^d)^N)$ towards consensus can be quantified in terms of the $(d+1)$-th smallest eigenvalue of the matrix $\Lb_N(t,\xb(t))$, counted with multiplicity.

%%%%%%%%%%%%%%%%%%%%%%%%%%%%%%%%%%%%%%%%%%%%%%%%%%%%%%%%%%%%%%%%%%%%%%%%%%%%%

\paragraph*{Various notions of algebraic connectivity for discrete systems.} 

The notion of \textit{algebraic connectivity} was originally introduced in the seminal work \cite{Fiedler1973} for symmetric adjacency matrices $\Ab_N := (a_{ij})_{1 \leq i,j \leq N} \in [0,1]^{N \times N}$ -- namely when $a_{ij} = a_{ji}$ for every pair of indices $i,j \in \{1,\dots,N\}$ --, and defined as the $(d+1)$-th smallest eigenvalue of the corresponding graph-Laplacian operator $\Lb_N \in \Lpazo((\R^d)^N)$, counted with multiplicity. By the Courant-Fischer min-max theorem (see e.g. \cite[Theorem 4.2.12]{Horn1985}), the latter could be alternatively defined as
\begin{equation}
\label{eq:AlgebraicConnectivity1}
\lambda_2(\Lb_N) := \min_{\xb \in \Ccal_N^{\perp}} \frac{\langle \Lb_N \xb , \xb \rangle}{|\xb|^2} \geq 0, 
\end{equation}
where the orthogonal complement of the consensus manifold can be characterised explicitly by 
\begin{equation}
\label{eq:OrthogonalConsensusDis}
\Ccal_N^{\perp} = \Big\{ \xb \in (\R^d)^N ~\text{s.t.}~ \bar{\xb} = 0 \Big\}.
\end{equation}
Here, $\bar{\xb} := \tfrac{1}{N} \sum_{i=1}^N x_i \in \R^d$ stands for the mean value of an element $\xb \in (\R^d)^N$. It has since been a classical result in algebraic graph theory that $\lambda_2(\Lb_N) > 0$ if and only if the underlying undirected graph is \textit{connected} in the sense of Definition \ref{def:Connected} above, and that the multiplicity of the zero eigenvalue of $\Lb_N$ is equal to the product between the number of connected components in the topology and the dimension $d$ of the state space (see e.g. the reference paper \cite{Mohar1991}).

When the interaction topology of the system is asymmetric, it is still possible to adapt the notion of algebraic connectivity in a meaningful way. In \cite{Wu2005}, it was shown that the definition given in \eqref{eq:AlgebraicConnectivity1} still makes sense in the class of \textit{balanced topologies}, which are characterised by the fact that 
\begin{equation*}
\frac{1}{N} \sum_{j=1}^N a_{ij} = \frac{1}{N} \sum_{j=1}^N a_{ji} 
\end{equation*}
for each $i \in \{1,\dots,N\}$. Indeed, it directly follows in this context that $\Lb_N^{\top} \xb = 0$ for each $\xb \in \Ccal_N$, and that the algebraic connectivity of the interaction topology is positive if and only if the interaction graph is strongly connected (see e.g. \cite[Lemma 17]{Wu2005}). We point the reader to Figure \ref{fig:Balanced} for some examples of balanced topologies. 

\begin{figure}[!ht]
\centering
\begin{tikzpicture}
% First figure
% Drawing the nodes
\draw (1,-1) node[shape=circle, draw] {$1$}; 
\draw (2,0) node[shape=circle, draw] {$2$};
\draw (2,-2) node[shape=circle, draw] {$3$};
\draw (0,0) node[shape=circle, draw] {$5$};
\draw (0,-2) node[shape=circle, draw] {$4$};
% Drawing the arrows
\draw[<-, line width = 0.75, blue] (0.75,-0.75)--(0.25,-0.25);
\draw[<-, line width = 0.75, blue] (0,-0.35)--(0,-1.65); 
\draw[->, line width = 0.75, blue] (1.25,-0.75)--(1.75,-0.25);  
\draw[->, line width = 0.75, blue] (2,-0.35)--(2,-1.65); 
\draw[->, line width = 0.75, blue] (1.65,-2)--(0.35,-2);
% Drawing the matrix
\draw (1.05,-4) node {\small $
\begin{pmatrix} 
1 & 1 & 0 & 0 & 0 \\ 
0 & 1 & 1 & 0 & 0 \\
0 & 0 & 1 & 1 & 0 \\
0 & 0 & 0 & 1 & 1 \\
1 & 0 & 0 & 0 & 1 \\
\end{pmatrix}$};
% Second figure
% Drawing the nodes
\begin{scope}[xshift = 6cm]
\draw (1,-1) node[shape=circle, draw] {$1$}; 
\draw (2,0) node[shape=circle, draw] {$2$};
\draw (2,-2) node[shape=circle, draw] {$3$};
\draw (0,0) node[shape=circle, draw] {$5$};
\draw (0,-2) node[shape=circle, draw] {$4$};
% Drawing the arrows
\draw[->, line width = 0.75, blue] (0.75,-0.75)--(0.25,-0.25);
\draw[->, line width = 0.75, blue] (0,-0.35)--(0,-1.65); 
\draw[->, line width = 0.75, blue] (0.25,-1.75)--(0.75,-1.25);
\draw[->, line width = 0.75, blue] (1.25,-0.75)--(1.75,-0.25);  
\draw[->, line width = 0.75, blue] (2,-0.35)--(2,-1.65); 
\draw[->, line width = 0.75, blue] (1.75,-1.75)--(1.25,-1.25);
% Drawing the matrix
\draw (1.05,-4) node {\small $
\begin{pmatrix} 
1 & 1 & 0 & 0 & 1 \\ 
0 & 1 & 1 & 0 & 0 \\
1 & 0 & 1 & 0 & 0 \\
1 & 0 & 0 & 1 & 0 \\
0 & 0 & 0 & 1 & 1 \\
\end{pmatrix}$};
\end{scope}
% Third figure
% Drawing the nodes
\begin{scope}[xshift = 12cm]
\draw (1,-1) node[shape=circle, draw] {$1$}; 
\draw (2,0) node[shape=circle, draw] {$2$};
\draw (2,-2) node[shape=circle, draw] {$3$};
\draw (0,0) node[shape=circle, draw] {$5$};
\draw (0,-2) node[shape=circle, draw] {$4$};
% Drawing the arrows
\draw[<->, line width = 0.75, blue] (0,-0.35)--(0,-1.65); 
\draw[->, line width = 0.75, blue] (2,-0.35)--(2,-1.65); 
\draw[->, line width = 0.75, blue] (1.25,-0.75)--(1.75,-0.25);  
\draw[->, line width = 0.75, blue] (1.75,-1.75)--(1.25,-1.25);
% Drawing the matrix
\draw (1.05,-4) node {\small $
\begin{pmatrix} 
1 & 1 & 0 & 0 & 0 \\ 
0 & 1 & 1 & 0 & 0 \\
1 & 0 & 1 & 0 & 0 \\
0 & 0 & 0 & 1 & 1 \\
0 & 0 & 0 & 1 & 1 \\
\end{pmatrix}$};
\end{scope}
\end{tikzpicture}
\caption{{\small \textit{Three examples of balanced topologies for graphs with $N = 5$ vertices. The graphs on the left and in the center are strongly connected, while that on the right is a disjoint union of two strongly connected components. The corresponding adjacency matrices are represented right below each graph.}}}
\label{fig:Balanced}
\end{figure}
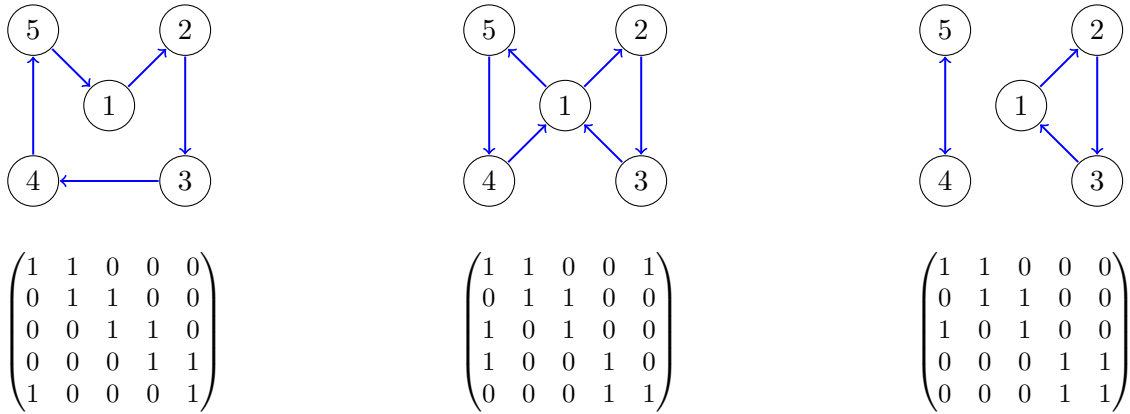

In the subsequent work \cite{Wu2005bis}, it was finally shown that the notion of algebraic connectivity could be further generalised to interaction topologies that can be written as \textit{disjoint unions of strongly connected components}, in the sense of Definition \ref{def:Connected} above. In this context, the definition of algebraic connectivity relies on the following structural result of Perron-Frobenius theory.

\begin{thm}[Algebraic characterisation of strong connectivity]
\label{thm:Perron}
Let $\Ab_N := (a_{ij})_{1 \leq i,j \leq N} \in [0,1]^{N \times N}$ be an adjacency matrix and $\Lb_N \in \Lpazo((\R^d)^N)$ be the corresponding graph-Laplacian operator. Then, the underlying interaction graph is a disjoint union of strongly connected components \textnormal{if and only if} there exists an element $v := (v_i)_{1 \leq i \leq N} \in (\R_+^*)^N$ such that
\begin{equation}
\Lb_N^{\top} \vb = 0 \qquad \text{and} \qquad \frac{1}{N} \sum_{i=1}^N v_i = 1, 
\end{equation} 
where $\vb := (v_1,\dots,v_1,\dots,v_N,\dots,v_N) \in (\R^d)^N$. 
\end{thm}
 
\begin{proof}
The proof of this result for $d=1$ can be found e.g. in \cite[Lemma 1]{Wu2005bis}. Then for an arbitrary state-space dimensions $d \geq 1$, it is sufficient to define $\vb \in (\R^d)^N$ as indicated above. 
\end{proof}

\begin{rmk}[$1$-dimensional versus $d$-dimensional linear systems]
When the right-hand side of a multi-agent system is linear, there is no coupling between the dynamics of the components $(x_i^k)_{1 \leq k \leq d}$ of the agents for $i \in \{1,\dots,N\}$. This explains why the canonical eigenvector $\vb \in (\R^d)^N$ for the $d$-dimensional system is simply obtained in this context by concatenating $d$ copies of that generated by the action of $\Lb_N$ on a $1$-dimensional space. 
\end{rmk}

Assuming that the interaction topology described by the adjacency matrix $\Ab_N$ is a disjoint union of strongly connected components, we consider the following weighted version of the graph-Laplacian
\begin{equation*}
\Lb_N^{\vb} := \textnormal{diag}(\vb) \Lb_N, 
\end{equation*}
where $\textnormal{diag}(\vb) \in \Lpazo((\R^d)^N)$ is the diagonal matrix whose entries are the components of $\vb \in (\R^d)^N$. Then, the algebraic connectivity of the underlying topology is defined by
\begin{equation}
\label{eq:AlgebraicConnectivity2}
\lambda_2(\Lb_N^{\vb}) = \min_{\xb \in \Ccal_N^{\perp}} \frac{\langle \Lb_N^{\vb} \xb , \xb \rangle}{|\xb|^2} \geq 0, 
\end{equation}
and it can be shown that this definition carries with it all the interesting properties of $\lambda_2(\cdot)$ (see \cite{Wu2005ter,Wu2005bis}). In particular, it is shown in  \cite[Section 4]{Wu2005bis} that in this context, $\lambda_2(\Lb_N^{\vb}) > 0$ if and only if the interaction graph is strongly connected. It is also worth noting that in the particular case of balanced topologies, the canonical vector given by Theorem \ref{thm:Perron} is simply $\vb := (1,\dots,1) \in (\R^d)^N$, so that the definitions of \eqref{eq:AlgebraicConnectivity1} and \eqref{eq:AlgebraicConnectivity2} coincide.    

\begin{rmk}[Comparison between the algebraic connectivity and the scrambling coefficient]
Given a symmetric adjacency matrix $\Ab_N \in [0,1]^{N \times N}$, it is a known fact that $\lambda_2(\Lb_N) \geq \eta(\Ab_N)$, namely that the scrambling coefficient provides a lower-bound on the algebraic connectivity of the corresponding interaction topology (see e.g. \cite[Section 2.2, Example 1]{Motsch2014}). This, however, is no longer true when the adjacency matrix is asymmetric. Indeed, the balanced interaction topologies displayed at the left and center positions in Figure \ref{fig:Balanced} are such that $\eta(\Ab_N) =0$, as in both situations the vertices $\{2,5\}$ are neither connected directly to one another nor to a common third party vertex, yet it can be checked fairly easily that $\lambda_2(\Lb_N) > 0$ for both of these graphs. Conversely, the interaction topology exposed at the left of Figure \ref{fig:Scrambling} is such that $\eta(\Ab_N) > 0$, but it does not admit a well-defined algebraic connectivity. 
\end{rmk}

%%%%%%%%%%%%%%%%%%%%%%%%%%%%%%%%%%%%%%%%%%%%%%%%%%%%%%%%%%%%%%%%%%%%%%%%%%%%

\paragraph*{Infinite-dimensional graph-Laplacians.} 

Similarly to what was done above for finite-dimensional cooperative systems, the graphon dynamics \eqref{eq:GraphonDynamics} can be reformulated as an ordinary differential equation in the configuration space $L^2(I,\R^d)$. To this end, we introduce the \textit{graph-Laplacian operator} $\Lbb : \R_+ \times L^2(I,\R^d) \rightarrow \Lpazo(L^2(I,\R^d))$, defined by
\begin{equation}
\label{eq:GraphLap}
\big( \Lbb(t,x)y \big)(i) := \INTDom{a(t,i,j) \phi(|x(i) - x(j)|) (y(i) - y(j))}{I}{j} \qquad \text{for $\LcalI$-almost every $i \in I$},  	
\end{equation}
for $\Lcal^1$-almost every $t \geq 0$ and every $x,y \in L^2(I,\R^d)$. In keeping with the notions introduced in Section \ref{subsection:Spectral}, observe that $\Lbb(t,x)$ can be recast as the difference between the multiplication operator $\Mpazo_d(t,x) \in \Lpazo(L^2(I,\R^d))$ by the so-called \textit{in-degree function} $d(t,x) \in L^{\infty}(I,[0,1])$, defined by 
\begin{equation*}
(\Mpazo_d(t,x) y)(i) :=  \bigg( \INTDom{a(t,i,j) \phi(|x(i) - x(j)|)}{I}{j} \bigg) y(i) \qquad \text{for $\LcalI$-almost every $i \in I$},
\end{equation*}
to which one subtracts a kernel-type adjacency operator $\Apazo(t,x) \in \Lpazo(L^2(I,\R^d))$ that writes 
\begin{equation*}
\big( \Apazo(t,x) y \big)(i) := \INTDom{a(t,i,j) \phi(|x(i) - x(j)|) x(j)}{I}{j} \qquad \text{for $\LcalI$-almost every $i \in I$}, 
\end{equation*}
where both expressions are understood for $\Lcal^1$-almost every $t \geq 0$ and each $x,y \in L^2(I,\R^d)$. This allows us to rewrite the graphon counterpart of the dynamics \eqref{eq:DiscreteLapDyn} as the semilinear Cauchy problem 
\begin{equation}
\label{eq:GraphonLapDyn}
\dot x(t) = - \Lbb(t,x(t))x(t), \qquad x(0) = x^0,  
\end{equation}
formulated in the Hilbert space $L^2(I,\R^d)$. In the sequel, we will sometimes consider the partial graph-Laplacian operators $\Lbb^a(\cdot)$ generated solely by the kernel $a \in L^{\infty}(\R_+ \times I \times I,[0,1])$, defined by 
\begin{equation}
\label{eq:PartialGraphLap}
(\Lbb^a(t) y)(i) := \INTDom{a(t,i,j)(y(i) - y(j))}{I}{j} \qquad \text{for $\LcalI$-almost every $i \in I$}
\end{equation}
for $\Lcal^1$-almost every $t \geq 0$ and each $y \in L^2(I,\R^d)$.

%%%%%%%%%%%%%%%%%%%%%%%%%%%%%%%%%%%%%%%%%%%%%%%%%%%%%%%%%%%%%%%%%%%%%%%%%%%%%%%%

\paragraph*{Generalised algebraic connectivity for graphon models.}

Just as finite-dimensional graph-Laplacians always admit a nontrivial kernel, their infinite-dimensional counterparts introduced in \eqref{eq:GraphLap} also vanish on the consensus manifold, which in the context of graphon dynamics is defined by 
\begin{equation}
\label{eq:ConsensusPerp}
\Ccal := \Big\{ x \in L^2(I,\R^d) ~\text{s.t. $x$ is constant over $I$} \Big\}.
\end{equation}
In other words, the kernel of a graph-Laplacian operator is a nontrivial subspace of dimension at least $d \geq 1$. Similarly to the finite-dimensional situation, the configuration space $L^2(I,\R^d)$ admits a canonical decomposition as the orthogonal sum $\Ccal \oplus \Ccal^{\perp}$. In this context, one also has the characterisation
\begin{equation}
\label{eq:ConsensusInfPerp}
\Ccal^{\perp} = \Big\{ x \in L^2(I,\R^d) ~\text{s.t.}~ \bar{x} = 0 \Big\}, 
\end{equation}
where $\bar{x} := \INTDom{x(i)}{I}{i} \in \R^d$ denotes the \textit{barycenter} of an element $x \in L^2(I,\R^d)$. We start our discussion on algebraic connectivities for graphon models by stating in the following definition a natural generalisation of the notion of balanced topology. 

\begin{Def}[Balanced interaction topologies]
\label{def:Balanced}
The interaction topology described by an adjacency operator $\Apazo \in \Lpazo(L^2(I,\R^d))$ is said to be \textnormal{balanced} provided that
\begin{equation*}
\INTDom{a(i,j)}{I}{i} = \INTDom{a(j,i)}{I}{i}, 
\end{equation*}
for $\LcalI$-almost every $i \in I$, or equivalently if the corresponding graph-Laplacian $\Lbb \in \Lpazo(L^2(I,\R^d))$ satisfies $\Lbb^* x = 0$ for each $x \in \Ccal$.
\end{Def}

Given a general adjacency operator $\Apazo \in \Lpazo(L^2(I,\R^d))$ with kernel $a \in L^{\infty}(I \times I,[0,1])$, we define the corresponding in-degree function as 
\begin{equation}
\label{eq:InDegreeDef}
d(i) := \INTDom{a(i,j)}{I}{j},
\end{equation}
for $\LcalI$-almost every $i \in I$, and denote by $\Mpazo_d \in \Lpazo(L^2(I,\R^d))$ the multiplication operator by $d \in L^{\infty}(I,[0,1])$. Then, the graph-Laplacian of the topology simply writes as 
\begin{equation}
\label{eq:GraphLapLin}
\Lbb := \Mpazo_d - \Apazo \in \Lpazo(L^2(I,\R^d)).
\end{equation}
In this context, we define the algebraic connectivity of a balanced topology by 
\begin{equation}
\label{eq:AlgebraicConnectivityInf1}
\lambda_2(\Lbb) := \inf_{x \in \Ccal^{\perp}} \frac{\langle \Lbb \, x , x \rangle_{L^2(I,\R^d)}}{\NormL{x}{2}{I,\R^d}^2} \geq 0,  
\end{equation}
and we shall see that it plays for $(-\Lbb)$ a role that is similar to that of the usual \textit{spectral bound} for semigroup generators (see e.g. \cite[Chapter IV, Definition 2.1]{Engel2001}). Just as for scrambling coefficients, the notion of algebraic connectivity that we propose is stable under discretisation. Indeed, if $\Lbb_N \in \Lpazo(L^2(I,\R^d))$ is the graph-Laplacian operator associated with a piecewise constant kernel $a^N \in L^{\infty}(I \times I,[0,1])$ of the form \eqref{eq:PiecewiseWeight}, then it will hold that $\lambda_2(\Lbb_N) = \lambda_2(\Lb_N)$, where $\Lb_N \in \Lpazo((\R^d)^N)$ denotes the usual graph-Laplacian matrix of the underlying digraph.

Similarly to the situation described for discrete multi-agent systems, it is possible to define a suitable notion of strong connectivity for graphons, which involves the \textit{Lebesgue points} (see e.g. \cite[Definition 1.24]{EvansGariepy}) of the interaction kernel. We recall below the definition of this concept, which is borrowed from the recent article \cite{Boudin2022}.

\begin{Def}[Strongly connected interaction topologies]
\label{def:StrongCon}
The interaction topology described by an adjacency operator $\Apazo \in \Lpazo(L^2(I,\R^d))$ is said to be \textnormal{strongly connected} if the following conditions hold. 
\begin{enumerate}
\item[(a)] For $\LcalI$-almost every $i \in I$ and for each Lebesgue point $j \in I \setminus \{i\}$ of $a(i,\cdot) \in L^{\infty}(I,[0,1])$, there exists an integer $m \geq 1$ and a finite sequence $(l_k)_{1 \leq k \leq m} \subset I$ satisfying $i = l_1$, $j = l_m$ and such that $l_{k+1} \in \supp(a(l_k,\cdot))$ is a Lebesgue point of $a(l_k,\cdot)$ for each $k \in \{1,\dots,m-1\}$.
\item[(b)] One has that $\delta := \inf_{i \in I} \INTDom{a(i,j)}{I}{j} > 0$. 
\end{enumerate}
Analogously, we say that an interaction topology is a \textnormal{disjoint union of strongly connected components} if there exists an at most countable family of sets $\{I_n\}_{n =1}^{+\infty} \subset \Ppazo(I)$ such that $I = \cup_{n \geq 1}^{+\infty} I_n$, which satisfies
\begin{equation*}
\LcalI(I_n) > 0, \quad \LcalI(I_n \cap I_m) = 0 \quad \text{and} \quad \supp(a(i_n,\cdot)) \subset I_n ~~ \text{for $\LcalI$-almost every $i_n \in I_n$},  
\end{equation*}
for each $m,n \geq 1$ with $m \neq n$, and on which the restricted operators $\Apazo_n \in \Lpazo(L^2(I_n,\R^d))$ define strongly connected topologies whose in-degree function comply with the quantitative lower-bound 
\begin{equation}
\label{eq:LowerBoundDisjoint}
\inf_{i_n \in I_n} \tfrac{1}{\LcalI(I_n)} \INTDom{a(i_n,j)}{I_n}{j} \geq \delta, 
\end{equation}
for a given uniform constant $\delta > 0$. 
\end{Def}

\begin{rmk}[On the lower-bound \eqref{eq:LowerBoundDisjoint} on the in-degree function]
Observe first that the uniform lower-bound \eqref{eq:LowerBoundDisjoint} imposed on each of the strongly connected components $I_n$ is automatically satisfied when the family of sets $\{I_n\}_{n =1}^{+\infty} \subset \Ppazo(I)$ is finite. When the latter is infinite countable, it transcribes the fact that while the measure of the sets $I_n$ must eventually vanish, the in-degree function remains lower-bounded in average. A typical example of topology expressed as an infinite union of strongly connected components satisfying this condition is given by the symmetric kernel defined by 
\begin{equation*}
a(i,j) := \sum_{n =1}^{+\infty} \mathds{1}_{I_n}(i) \mathds{1}_{I_n}(j),
\end{equation*}
for every $i,j \in I$, where the family of sets $\{I_n\}_{n =1}^{+\infty}$ is given by $I_n := \big( \tfrac{1}{n+1} , \tfrac{1}{n} \big]$ for each $n \geq 1$. 
\end{rmk}

When an interaction topology is decomposable into a disjoint union of strongly connected components, it is yet again possible to define a proper notion of algebraic connectivity, by resorting to the following adaptation of one of the main results from \cite{Boudin2022}. The latter can be seen as a generalisation of the ``only if'' part of Theorem \ref{thm:Perron} for adjacency operators.

\begin{thm}[Algebraic characterisation of strong connectivity for graphons]
\label{thm:StrongConGraphon}
Let $\Apazo \in \Lpazo(L^2(I,\R^d))$ be an adjacency operator and $\Lbb \in \Lpazo(L^2(I,\R^d))$ be the corresponding graph-Laplacian operator defined as in \eqref{eq:GraphLapLin}. Then, if the underlying  interaction topology is a disjoint union of strongly connected components in the sense of Definition \ref{def:StrongCon}, there exists an element $v \in L^{\infty}(I,\R_+^*)$ such that 
\begin{equation*}
\Lbb^* v = 0, \qquad \INTDom{v(i)}{I}{i} = 1 \qquad \text{and} \qquad \NormL{v}{\infty}{I,\R_+^*} \, \leq \, \frac{1}{\delta},
\end{equation*}
where we write $\Lbb^* v = 0$ to implicitly mean that $\Lbb^*(v,\dots,v) = 0$ in $L^2(I,\R^d)$. If in addition one imposes that $\INTDom{v(i)}{I_n}{i} = \LcalI(I_n)$ for every $n \geq 1$, then $v \in L^{\infty}(I,\R_+^*)$ is uniquely determined.  
\end{thm} 

\begin{proof}
In the case where $d = 1$, the results of \cite[Theorem 1]{Boudin2022} yield the existence of a countable family of uniquely determined functions $(v_n)_{n \geq 1} \subset  L^2(I,\R_+^*)$ which satisfy
\begin{equation*}
\supp(v_n) = I_n, \qquad \Lbb^*_n v_n = 0 \qquad \text{and} \qquad \INTDom{v_n(i)}{I_n}{i} = \LcalI(I_n),
\end{equation*}
for each $n \geq 1$. Here, $\Lbb^*_n \in \Lpazo(L^2(I_n,\R))$ denotes the restriction of $\Lbb^*$ to $L^2(I_n,\R)$, namely 
\begin{equation*}
(\Lbb^*_n \, x)(i_n) = \INTDom{\Big( a(i_n,j) x(i_n) - a(j,i_n)x(j) \Big)}{I_n}{j} \qquad \text{for $\LcalI$-almost every $i_n \in I_n$}, 
\end{equation*}
for each $x \in L^2(I_n,\R)$. Then, one can show that the map defined by 
\begin{equation}
\label{eq:vdef}
v(i) := \sum_{n=1}^{+\infty} \mathds{1}_{I_n}(i) v_n(i), 
\end{equation}
for $\LcalI$-almost every $i \in I$ satisfies the conclusions of Theorem \ref{thm:StrongConGraphon}. Observe then that by \eqref{eq:LowerBoundDisjoint}, the fact that $\Lbb^*_n v_n = 0$ directly yields 
\begin{equation*}
v_n(i_n) = \frac{\INTDom{a(j,i)v_n(j)}{I_n}{j}}{\INTDom{a(i_n,j)}{I_n}{j}} \leq \frac{\INTDom{v_n(j)}{I_n}{j}}{\delta \LcalI(I_n)} = \frac{1}{\delta}, 
\end{equation*}
for $\LcalI$-almost every $i_n \in I_n$ and each $n \geq 1$. Because the sets $(I_n)_{n \geq 1}$ form a disjoint partition of $I$, it follows easily from the definition of $v$ in \eqref{eq:vdef} that
\begin{equation*}
\NormL{v}{\infty}{I,\R_+^*} = \sup_{n \geq 1} \NormL{v_n}{\infty}{I_n,\R_+^*} \, \leq \,  \frac{1}{\delta}.
\end{equation*}
Finally when $d \geq 1$ is arbitrary, the claims of Theorem \ref{thm:StrongConGraphon} simply follow by defining the mapping $\V : i \in I \mapsto (v(i),\dots,v(i)) \in \R^d$ with $v \in L^{\infty}(I,\R_+^*)$ being given as in \eqref{eq:vdef}.
\end{proof}

Given an adjacency operator $\Apazo \in \Lpazo(L^2(I,\R^d))$ defining a disjoint union of strongly connected components in the sense of Definition \ref{def:StrongCon}, we introduce the rescaled Laplacian operator
\begin{equation*}
\Lbb_v := \Mpazo_v \, \Lbb, 
\end{equation*}
where the function $v \in L^{\infty}(I,\R_+^*)$ appearing in the multiplication operator is the one given by Theorem \ref{thm:StrongConGraphon}. Then, we define the algebraic connectivity of the underlying interaction topology by 
\begin{equation}
\label{eq:AlgebraicConnectivityInf2}
\lambda_2(\Lbb_v) := \inf_{x \in \Ccal^{\perp}} \frac{\langle \Lbb_v \, x , x \rangle_{L^2(I,\R^d)}}{\Norm{x}_{L^2(I,\R^d)}^2} \geq 0,
\end{equation}
and we will see in Section \ref{subsection:L2Consensus} that this quantity does allow to quantify consensus formation for strongly connected topologies. Therein, we will also prove that the rescaled Laplacian operator is indeed a positive semi-definite operator over $L^2(I,\R^d)$. 

\begin{open}
Is the converse implication of Theorem \ref{thm:Perron} still valid in the context of graphon models described by Theorem \ref{thm:StrongConGraphon}? Namely, is it true that an interaction topology is a countable disjoint union of strongly connected components whenever there exists an element $v \in L^{\infty}(I,\R_+^*)$ such that $\Lbb^* v = 0$? 
\end{open}

%%%%%%%%%%%%%%%%%%%%%%%%%%%%%%%%%%%%%%%%%%%%%%%%%%%%%%%%%%%%%%%%%%%%%%%%%%%%%%%

\subsection{Exponential consensus formation for algebraically persistent topologies}
\label{subsection:L2Consensus}

In this section, we investigate the asymptotic formation of $L^2$-consensus for graphon models whose interaction topologies are persistent in a suitable sense, quantified by means of the algebraic connectivity. We start by exposing in Section \ref{subsubsection:Symmetric} a general convergence result for symmetric topologies whose time-averages have lower-bounded algebraic connectivity. We then proceed in Section \ref{subsubsection:Balanced} by studying consensus formation in balanced topologies, under the condition that the algebraic connectivity itself is persistent. We finally adapt the underlying strategy to study piecewise-constant (in time) disjoint unions of strongly connected components in Section \ref{subsubsection:StronglyConnected}. 

%%%%%%%%%%%%%%%%%%%%%%%%%%%%%%%%%%%%%%%%%%%%%%%%%%%%%%%%%%%%%%%%%%%%%%%%%%%%%%%%

\subsubsection{Symmetric nonlinear interaction topologies}
\label{subsubsection:Symmetric}

In what follows, we shall prove the following result, which describes the exponential formation of $L^2$-consensus for symmetric graphon dynamics, under the condition that the time-averages of the corresponding graph-Laplacian operators over time-windows of fixed length have a uniformly lower-bounded algebraic connectivity. 

\begin{thm}[Exponential $L^2$-consensus formation for symmetric topologies]
\label{thm:ConsensusSym}
Fix an initial datum $x^0 \in L^{\infty}(I,\R^d)$ with $\NormL{x^0}{\infty}{I,\R^d} \leq R
$ for some $R > 0$, assume that hypotheses \ref{hyp:GD} hold, and that $a(t) \in L^{\infty}(I \times I,[0,1])$ defines a \textnormal{symmetric} interaction topology for $\Lcal^1$-almost every $t \geq 0$. Moreover, suppose that there exists a pair of coefficients $(\tau,\mu) \in \R_+^* \times (0,1]$ such that the persistence condition 
\begin{equation}
\label{eq:PersistenceTopo1}
\lambda_2 \bigg( \frac{1}{\tau} \INTSeg{\Lbb^a(s)}{s}{t}{t+\tau}\bigg) \geq \mu,
\end{equation}
holds for all times $t \geq 0$, where $\Lbb^a(s) \in \Lpazo(L^2(I,\R^d))$ is the linear part of the graph-Laplacian operator as defined in \eqref{eq:PartialGraphLap} for $\Lcal^1$-almost every $s \geq 0$. 

Then, there exist constants $\alpha,\gamma > 0$ depending only on $(\phi(\cdot),R,\tau,\mu)$ such that 
\begin{equation}
\label{eq:ExpConsensus1}
\NormL{x(t) - \bar{x}^0}{2}{I,\R^d} \leq \alpha \NormL{x(0) - \bar{x}^0}{2}{I,\R^d} \exp \big( - \gamma \mu \hspace{0.01cm} t \big), 
\end{equation}
for all times $t \geq 0$. In particular, the solutions of \eqref{eq:GraphonLapDyn} exponentially converge to consensus in the $L^2(I,\R^d)$-norm topology.
\end{thm}

Taking inspiration from some of the seminal works \cite{Caponigro2013,Caponigro2015} on the mathematical formalisation of consensus and flocking analysis for discrete symmetric multi-agent systems, we introduce the \textit{variance bilinear form} over $L^2(I,\R^d)$, which is defined by 
\begin{equation}
\label{eq:BilinearInf}
\B(x,y) := \INTDom{\langle x(i) , y(i) \rangle}{I}{i} - \langle \bar{x} , \bar{y} \rangle, 
\end{equation}
for any $x,y \in L^2(I,\R^d)$, and we observe that the evaluation $\B(x,x)$ coincides exactly with the $L^2$-distance between $x \in L^2(I,\R^d)$ and the consensus manifold $\Ccal$. Based on this observation, we consider for every solution $x(\cdot) \in \Lip_{\loc}(\R_+,L^2(I,\R^d))$ of \eqref{eq:GraphonLapDyn} the \textit{standard deviation} map
\begin{equation}
\label{eq:StandardDev1}
X(t) :=  \sqrt{\B(x(t),x(t))} = \,  \NormL{x(t) - \bar{x}(t)}{2}{I,\R^d}, 
\end{equation}
that is defined for all times $t \geq 0$. In the following proposition, we state some of the useful properties of the variance bilinear form, notably in conjunction with graph-Laplacians. 

\begin{prop}[Elementary properties of $\B$]
\label{prop:Bilinear}
The map $x \in L^2(I,\R^d) \mapsto \B(x,x) \in \R_+$ defines a seminorm over $L^2(I,\R^d)$, which satisfies the Cauchy-Schwarz type inequality 
\begin{equation}
\label{eq:CSBilinear}
\B(x,y) \leq \sqrt{\B(x,x)} \sqrt{\B(y,y)}, 
\end{equation}
for each $x,y \in L^2(I,\R^d)$. In addition, let $\Lbb \in \Lpazo(L^2(I,\R^d))$ be the graph-Laplacian associated with an adjacency operator $\Apazo \in \Lpazo(L^2(I,\R^d))$ defining a balanced topology in the sense of Definition \ref{def:Balanced}. Then, one has that $\Lbb \, x \in \Ccal^{\perp}$ as well as $\Lbb^* x \in \Ccal^{\perp}$, and the algebraic connectivity can be rewritten as
\begin{equation}
\label{eq:AlternativeLambda2}
\lambda_2(\Lbb) = \inf_{x \notin \Ccal} \frac{\B(\Lbb \, x , x)}{\B(x,x)} \geq 0.
\end{equation}
\end{prop}

\begin{proof}
Proving that $\B(\cdot,\cdot)$ induces a seminorm over $L^2(I,\R^d)$ is a matter of elementary computations. Concerning the Cauchy-Schwarz inequality, let it first be noted that 
\begin{equation*}
\B(x,y) = \B(x-\bar{x},y) = \B(x,y- \bar{y}) = \B(x-\bar{x},y-\bar{y}) = \langle x - \bar{x} , y-\bar{y}\rangle_{L^2(I,\R^d)},  
\end{equation*}
for any $x,y \in L^2(I,\R^d)$. The inequality \eqref{eq:CSBilinear} then follows by an application of the standard Cauchy-Schwarz inequality for the scalar product of $L^2(I,\R^d)$. Let now $\Lbb \in \Lpazo(L^2(I,\R^d))$ be the graph-Laplacian operator associated with a balanced graphon, and notice that 
\begin{equation}
\label{eq:AverageGraphLap}
\begin{aligned}
\INTDom{(\Lbb \, x)(i)}{I}{i} & = \INTDom{\INTDom{a(i,j) (x(i) - x(j))}{I}{j}}{I}{i} \\
& = \INTDom{\bigg( \INTDom{(a(i,j) - a(j,i))}{I}{j} \bigg) x(i)}{I}{i} = 0, 
\end{aligned}
\end{equation}
for every $x \in L^2(I,\R^d)$, where we used Fubini's theorem along with Definition \ref{def:Balanced}. This together with \eqref{eq:ConsensusInfPerp} implies that $\Lbb \, x \in \Ccal^{\perp}$, and  one can show that $\Lbb^* x \in \Ccal^{\perp}$ in the very same way. The alternative characterisation \eqref{eq:AlternativeLambda2} of the algebraic connectivity simply follows from the observation that
\begin{equation*}
\B(\Lbb \, x,x) = \langle \Lbb \, x,x \rangle_{L^2(I,\R^d)}, 
\end{equation*}
for any $x \in L^2(I,\R^d)$, as a consequence of \eqref{eq:AverageGraphLap}. Finally, one can obtain the positive semi-definiteness of $\Lbb$ with respect to $\B(\cdot,\cdot)$ by checking that  
\begin{equation*}
\begin{aligned}
\B(\Lbb \, x,x) & = \INTDom{\INTDom{a(i,j) \langle x(i) , x(i) - x(j) \rangle}{I}{j}}{I}{i} \\
& = \frac{1}{2} \INTDom{\INTDom{ a(i,j) \Big( |x(i)-x(j)|^2 + |x(i)|^2 - |x(j)|^2 \Big)}{I}{j}}{I}{i} \\
& = \frac{1}{2} \INTDom{\INTDom{ a(i,j) |x(i)-x(j)|^2}{I}{j}}{I}{i} + \frac{1}{2} \INTDom{ \bigg( \INTDom{a(i,j)}{I}{j} - \INTDom{a(j,i)}{I}{j} \bigg) |x(i)|^2}{I}{i}  \\
& = \frac{1}{2} \INTDom{\INTDom{ a(i,j) |x(i)-x(j)|^2}{I}{j}}{I}{i} \geq 0,  
\end{aligned}
\end{equation*}
where we again used Fubini's theorem and the fact that the interaction topology is balanced. 
\end{proof}

The proof of Theorem \ref{thm:ConsensusSym} will be conveyed by building a \textit{strictly dissipative} Lyapunov function for \eqref{eq:GraphonLapDyn}, following a methodology developed in \cite{CSComFail} for finite-dimensional symmetric multi-agent systems. Taking inspiration from the literature of strict Lyapunov design for persistent systems, we define given a trajectory $x(\cdot) \in \Lip_{\loc}(\R_+,L^2(I,\R^d))$ of \eqref{eq:GraphonLapDyn} the time-dependent linear operator
\begin{equation}
\label{eq:PsiDef}
\Psi_{\tau}(t) := (1 + c_{\phi}) \tau \Id - \frac{1}{\tau} \INTSeg{\INTSeg{\Lbb(\sigma,x(\sigma))}{\sigma}{t}{s}}{s}{t}{t+\tau} ~\in~ \Lpazo(L^2(I,\R^d)), 
\end{equation}
for all times $t \geq 0$, where we recall that $c_{\phi} = \sup_{r \in \R_+} \phi(r) < +\infty$ under hypotheses \ref{hyp:GD}.

\begin{lem}[Properties of the operators $\Psi_{\tau}$]
\label{lem:Psi}
The operators $\Psi_{\tau}(t) \in \Lpazo(L^2(I,\R^d))$ are symmetric for all times $t \geq 0$, and satisfy the estimates  
\begin{equation}
\label{eq:PsiIneq}
\tau \B(y,y) \leq \B(\Psi_{\tau}(t)y,y) \leq (1+c_{\phi}) \tau \B(y,y), 
\end{equation}
for each $y \in L^2(I,\R^d)$. Moreover, the map $t \in \R_+ \mapsto \Psi_{\tau}(t) \in \Lpazo(L^2(I,\R^d))$ is Lipschitz continuous and differentiable $\Lcal^1$-almost everywhere, with 
\begin{equation}
\label{eq:PsiDerv}
\tderv{}{t}{} \Psi_{\tau}(t) = \Lbb(t,x(t)) - \frac{1}{\tau} \INTSeg{\Lbb(s,x(s))}{s}{t}{t+\tau}, 
\end{equation}
for $\Lcal^1$-almost every $t \geq 0$.
\end{lem}

\begin{proof}
The fact that the operators $\Psi_{\tau}(t) \in \Lpazo(L^2(I,\R^d))$ are symmetric for every $t \geq 0$ is a simple consequence of the linearity of the Bochner integral. Given $y \in L^2(I,\R^d)$, one can show that
\begin{equation}
\label{eq:PsiEstProof1}
0 \leq \B \big( \Lbb(\sigma,x(\sigma)) y,y \big) \leq c_{\phi} \B(y,y), 
\end{equation}
where the first inequality follows from the positive semi-definiteness of $\Lbb(\sigma,x(\sigma))$ with respect to $\B(\cdot,\cdot)$, provided by Proposition \ref{prop:Bilinear} for $\Lcal^1$-almost every $\sigma \in [t,t+\tau]$, while the second one can be obtained by remarking that 
\begin{equation*}
\begin{aligned}
\B \big( \Lbb(\sigma,x(\sigma))y,y \big) & = \frac{1}{2} \INTDom{\INTDom{a(\sigma,i,j) \phi(|x(\sigma,i) - x(\sigma,j)|) |y(i) - y(j)|^2}{I}{j}}{I}{i} \\
& \leq \frac{1}{2} \INTDom{\INTDom{ c_{\phi} |y(i) - y(j)|^2}{I}{j}}{I}{i} \\
& = c_{\phi} \, \B(y,y),
\end{aligned}
\end{equation*}
by symmetry of the weight functions $a(\sigma) \in L^{\infty}(I \times I,[0,1])$. Upon integrating \eqref{eq:PsiEstProof1} with respect to $\sigma \in [t,s]$ and then with respect to $s \in [t,t+\tau]$ while merging the resulting expression with the definition \eqref{eq:PsiDef} of $\Psi_{\tau}(t) \in \Lpazo(L^2(I,\R^d))$, we obtain the lower- and upper-bounds 
\begin{equation*}
\tau \B(y,y) \leq \B(\Psi_{\tau}(t)y,y) \leq (1 + c_{\phi}) \tau \B(y,y),
\end{equation*}
for every $y \in L^2(I,\R^d)$ and all times $t \geq 0$. 

We now shift our attention to the regularity properties of the map $t \in \R_+ \mapsto \Psi_{\tau}(t) \in \Lpazo(L^2(I,\R^d))$. Let $t_2 \geq t_1 \geq 0$ be such that $|t_2 - t_1| \leq \tau$, and fix an element $y \in L^2(I,\R^d)$. Then, one has that
\begin{equation*}
\begin{aligned}
& \NormL{(\Psi_{\tau}(t_2) - \Psi_{\tau}(t_1))y}{2}{I,\R^d} \\
& \leq \frac{1}{\tau} \INTSeg{\INTSeg{\NormL{\Lbb(\sigma,x(\sigma)) y}{2}{I,\R^d}}{\sigma}{t_1}{s}}{s}{t_1}{t_2} + \frac{1}{\tau} \INTSeg{\INTSeg{\NormL{\Lbb(\sigma,x(\sigma)) y}{2}{I,\R^d}}{\sigma}{t_1}{t_2}}{s}{t_2}{t_1+\tau} \\
& \hspace{6.75cm} + \frac{1}{\tau} \INTSeg{\INTSeg{\NormL{\Lbb(\sigma,x(\sigma)) y}{2}{I,\R^d}}{\sigma}{t_2}{s}}{s}{t_1+\tau}{t_2+\tau} \\
& \leq (t_2 - t_1) \Bigg( \frac{1}{\tau} \INTSeg{\NormL{\Lbb(s,x(s)) y}{2}{I,\R^d}}{s}{t_1}{t_2} + \frac{1}{\tau} \INTSeg{\NormL{\Lbb(s,x(s)) y}{2}{I,\R^d}}{s}{t_2}{t_1+\tau}\\
& \hspace{7.15cm}  + \frac{1}{\tau} \INTSeg{\NormL{\Lbb(s,x(s)) y}{2}{I,\R^d}}{s}{t_2}{t_2+\tau} \Bigg) \\
& \leq 2 c_{\phi} |t_2 - t_1| \NormL{y}{2}{I,\R^d},
\end{aligned}
\end{equation*}
where we used Fubini's theorem and the fact that $\Norm{\Lbb(\sigma,x(\sigma))}_{\Lpazo(L^2(I,\R^d))} \leq c_{\phi}$. Therefore, we obtain
\begin{equation*}
\Norm{\Psi_{\tau}(t_2) - \Psi_{\tau}(t_1)}_{\Lpazo(L^2(I,\R^d))} \, \leq 2 c_{\phi} |t_2 - t_1|, 
\end{equation*}
for each time instants $t_2 \geq t_1 \geq 0$ satisfying $|t_2 - t_1| \leq \tau$. As the Lipschitz constant does not depend on $t_1,t_2 \geq 0$, we obtain that the  map $\Psi_{\tau} : \R_+ \rightarrow \Lpazo(L^2(I,\R^d))$ is Lipschitz continuous by covering the real line with intervals of length $\tau$. In particular, it is differentiable $\Lcal^1$-almost everywhere as a consequence of the generalisation of Rademacher's theorem for maps with values in Hilbert spaces (see e.g. \cite[Theorem 1.3]{Heinrich1982}). Finally, the expression \eqref{eq:PsiDerv} of its time-derivative can be obtained by standard computations.
\end{proof}

Now that we have introduced these technical tools, we can move to the proof of Theorem \ref{thm:ConsensusSym}. 

\begin{proof}[Proof of Theorem \ref{thm:ConsensusSym}]
Let $x \in \Lip_{\loc}(\R_+,L^2(I,\R^d))$ be a solution of \eqref{eq:GraphonLapDyn}, and assume without loss of generality that $x(t) \notin \Ccal$ for all times $t \geq 0$. Given an arbitrary parameter $\lambda > 0$ whose value will be fixed later on, we consider the candidate Lyapunov functional defined by  
\begin{equation*}
\Xpazo_{\tau}(t) := \lambda X(t) + \sqrt{\B\big( \Psi_{\tau}(t)x(t),x(t) \big)}
\end{equation*}
for every $t \geq 0$, and it can be checked directly as a consequence of \eqref{eq:PsiIneq} that 
\begin{equation}
\label{eq:XpazoBound}
(\lambda + \sqrt{\tau}) X(t) \leq \Xpazo_{\tau}(t) \leq \Big( \lambda + \sqrt{(1 + c_{\phi}) \tau} \, \Big) X(t)
\end{equation}
for all times $t  \geq 0$. Moreover, owing to the regularity of $\Psi_{\tau}(\cdot)$, the mapping $t \in \R_+ \mapsto \Xpazo_{\tau}(t) \in \R_+$ is differentiable $\Lcal^1$-almost everywhere, with 
\begin{equation*}
\begin{aligned}
\tderv{}{t}{} \Xpazo_{\tau}(t) & = -\frac{\lambda}{X(t)} \B \big( \Lbb(t,x(t))x(t),x(t) \big) + \frac{\B \big( \tderv{}{t}{}\Psi_{\tau}(t)x(t),x(t) \big) - 2 \B \big( \Psi_{\tau}(t)x(t),\Lbb(t,x(t))x(t) \big)}{2 \sqrt{\B\big( \Psi_{\tau}(t)x(t),x(t) \big)}}  \\
& =  -\frac{\lambda}{X(t)} \B \big( \Lbb(t,x(t))x(t),x(t) \big) - \frac{1}{2 \sqrt{\B\big( \Psi_{\tau}(t)x(t),x(t) \big)}} \B \bigg( \Big( \tfrac{1}{\tau} \mathsmaller{\INTSeg{\Lbb(s,x(s))}{s}{t}{t+\tau}} \Big) x(t),x(t) \bigg) \\
& \hspace{0.45cm} + \frac{1}{2 \sqrt{\B \big( \Psi_{\tau}(t)x(t),x(t) \big)}} \bigg( \B \big( \Lbb(t,x(t))x(t),x(t) \big) - 2 \B \big( \Psi_{\tau}(t)x(t),\Lbb(t,x(t))x(t) \big) \bigg),
\end{aligned}
\end{equation*}
for $\Lcal^1$-almost every $t \geq 0$, where we used the fact that $\Psi_{\tau}(t)$ is symmetric. In addition, as a consequence of the persistence condition \eqref{eq:PersistenceTopo1}, one has that
\begin{equation*}
\begin{aligned}
& \B \bigg( \Big( \tfrac{1}{\tau} \mathsmaller{\INTSeg{\Lbb(s,x(s))}{s}{t}{t+\tau}} \Big) x(t),x(t) \bigg) \\
& \hspace{0.4cm} = \frac{1}{2} \INTDom{\INTDom{ \Big( \tfrac{1}{\tau} \mathsmaller{\INTSeg{a(s,i,j) \phi(|x(s,i)-x(s,j)|)}{s}{t}{t+\tau}} \Big) |x(t,i) - x(t,j)|^2}{I}{j}}{I}{i} \\
& \hspace{0.4cm} \geq \frac{1}{2} \INTDom{\INTDom{ \gamma_R \Big( \tfrac{1}{\tau} \mathsmaller{\INTSeg{a(s,i,j)}{s}{t}{t+\tau}} \Big) |x(t,i) - x(t,j)|^2}{I}{j}}{I}{i} \\
& \hspace{0.4cm} = \gamma_R \, \B \bigg( \Big( \tfrac{1}{\tau} \mathsmaller{\INTSeg{\Lbb^a(s)}{s}{t}{t+\tau}} \Big) x(t),x(t) \bigg) \\
& \hspace{0.4cm} \geq \gamma_R \, \mu X(t)^2, 
\end{aligned}
\end{equation*}
where $\gamma_R := \min_{r \in [0,2R]} \phi(r)$. This latter expression along with the positive semi-definiteness of $\Lbb(t,x(t))$ with respect to $\B(\cdot,\cdot)$ and the bounds of \eqref{eq:PsiIneq} further yields 
\begin{equation}
\label{eq:XpazoDervEst1}
\begin{aligned}
\tderv{}{t}{} \Xpazo_{\tau}(t) & \leq - \frac{\gamma_R \, \mu}{2 \sqrt{(1+c_{\phi}) \tau}} X(t) + \frac{1}{X(t)} \bigg( \frac{1}{\sqrt{\tau}} - \sqrt{(1+c_{\phi}) \tau} - \lambda \bigg) \B(\Lbb(t,x(t))x(t),x(t)) \\
& \hspace{0.45cm} + \frac{1}{\sqrt{\B(\Psi_{\tau}(t)x(t),x(t))}} \, \B \bigg( \Big( \tfrac{1}{\tau} \mathsmaller{\INTSeg{\INTSeg{\Lbb(\sigma,x(\sigma))}{\sigma}{t}{s}}{s}{t}{t+\tau}} \Big) x(t) , \Lbb(t,x(t))x(t) \bigg), 
\end{aligned}
\end{equation}
for $\Lcal^1$-almost every $t \geq 0$. We now need to estimate the last term in the right-hand side of the previous inequality. As consequence of the Cauchy-Schwarz inequality supported by $\B(\cdot,\cdot)$ and Young's inequality, it holds for every $\epsilon > 0$ that
\begin{equation}
\label{eq:XpazoDervEst2}
\begin{aligned}
& \B \bigg( \Big( \tfrac{1}{\tau} \mathsmaller{\INTSeg{\INTSeg{\Lbb(\sigma,x(\sigma))}{\sigma}{t}{s}}{s}{t}{t+\tau}} \Big) x(t) , \Lbb(t,x(t))x(t) \bigg) \\
& \leq \sqrt{\B \bigg( \tfrac{1}{\tau} \mathsmaller{\INTSeg{\INTSeg{\Lbb(\sigma,x(\sigma))}{\sigma}{t}{s}}{s}{t}{t+\tau}} \Big) x(t) , \tfrac{1}{\tau} \mathsmaller{\INTSeg{\INTSeg{\Lbb(\sigma,x(\sigma))}{\sigma}{t}{s}}{s}{t}{t+\tau}} \Big) x(t) \bigg)} \sqrt{\B \big(\Lbb(t,x(t))x(t),\Lbb(t,x(t))x(t) \big)} \\
& \leq \big\| \tfrac{1}{\tau} \mathsmaller{\INTSeg{\INTSeg{\Lbb(\sigma,x(\sigma))}{\sigma}{t}{s}}{s}{t}{t+\tau}} \big\|_{\B}^{1/2}  X(t) \Big( \sqrt{\B \big(\Lbb(t,x(t))x(t),\Lbb(t,x(t))x(t) \big)} \, \Big) \\
& \leq \frac{\epsilon}{2} \big\| \tfrac{1}{\tau} \mathsmaller{\INTSeg{\INTSeg{\Lbb(\sigma,x(\sigma))}{\sigma}{t}{s}}{s}{t}{t+\tau}} \big\|_{\B} \, X(t)^2 + \frac{1}{2\epsilon} \B \big( \Lbb(t,x(t))x(t), \Lbb(t,x(t))x(t) \big), 
\end{aligned}
\end{equation}
where $\|\cdot\|_{\B}$ denotes the operator semi-norm induced by $\B(\cdot,\cdot)$ on $\Lpazo(L^2(I,\R^d))$. By Jensen's inequality, one can further estimate the latter as
\begin{equation}
\label{eq:XpazoDervEst3}
\big\| \tfrac{1}{\tau} \mathsmaller{\INTSeg{\INTSeg{\Lbb(\sigma,x(\sigma))}{\sigma}{t}{s}}{s}{t}{t+\tau}} \big\|_{\B} \leq \tfrac{1}{\tau} \mathsmaller{\INTSeg{\INTSeg{ \|\Lbb(\sigma,x(\sigma)) \|_{\B}}{\sigma}{t}{s}}{s}{t}{t+\tau}} \leq c_{\phi} \tau, 
\end{equation}
since $\| \Lbb(\sigma,x(\sigma)) \|_{\B} \leq c_{\phi}$ for $\Lcal^1$-almost every $\sigma \in [t,t+\tau]$ owing to \eqref{eq:PsiEstProof1}. By the square-root lemma for positive semi-definite symmetric operators (see e.g. \cite[Theorem VI.9]{ReedI1981}), it also holds that
\begin{equation}
\begin{aligned}
\label{eq:XpazoDervEst4}
\B \big( \Lbb(t,x(t))x(t) , \Lbb(t,x(t))x(t) \big) & \leq \sqrt{c_{\phi}} ~ \B \Big( (\Lbb(t,x(t))^* \Lbb(t,x(t)))^{1/2} x(t),x(t) \Big) \\
& = \sqrt{c_{\phi}} ~ \B(\Lbb(t,x(t))x(t),x(t)),
\end{aligned}
\end{equation}
by the symmetry of $\Lbb(t,x(t))$ for $\Lcal^1$-almost every $t \geq 0$. Hence, by merging the estimates of \eqref{eq:XpazoDervEst2}, \eqref{eq:XpazoDervEst3} and \eqref{eq:XpazoDervEst4} into the differential inequality \eqref{eq:XpazoDervEst1}, one recovers that
\begin{equation}
\label{eq:XpazoDervEst5}
\begin{aligned}
\tderv{}{t}{} \Xpazo_{\tau}(t) & \leq \bigg( \frac{\epsilon}{2 \sqrt{\tau}} - \frac{\gamma_R \, \mu}{2 \sqrt{(1+c_{\phi}) \tau}} \bigg) X(t) \\
& \hspace{0.45cm} + \frac{1}{X(t)} \bigg( \frac{1}{\sqrt{\tau}} + \frac{1}{2 \sqrt{\tau} \epsilon} - \sqrt{(1+c_{\phi}) \tau} - \lambda \bigg) \B(\Lbb(t,x(t))x(t),x(t)),
\end{aligned}
\end{equation}
for $\Lcal^1-$almost every $t \geq 0$. 

Our goal now is to choose the free parameters $\lambda,\epsilon > 0$ in such a way that \eqref{eq:XpazoDervEst5} yields a strictly dissipative estimate on the standard deviation $X(\cdot)$. To this end, we fix 
\begin{equation}
\label{eq:ParameterDef}
\epsilon := \frac{\gamma_R \, \mu}{2 \sqrt{(1 + c_{\phi})}} \qquad \text{and} \qquad \lambda := \frac{1}{\sqrt{\tau}} + \frac{1}{2 \sqrt{\tau} \epsilon} - \sqrt{(1+c_{\phi})} \tau, 
\end{equation}
which together with the estimates of \eqref{eq:XpazoBound} allows us to rewrite \eqref{eq:XpazoDervEst5} as
\begin{equation}
\label{eq:XpazoderEst6}
\tderv{}{t}{} \Xpazo_{\tau}(t) \leq - \frac{\mu}{4 \sqrt{(1+c_{\phi}) \tau} \Big( \lambda + \sqrt{(1 + c_{\phi})\tau} \, \Big)} \Xpazo_{\tau}(t).
\end{equation}
By an application of Gr\"onwall's lemma together with yet another estimation using the bounds of \eqref{eq:XpazoBound}, we can conclude that 
\begin{equation*}
X(t) \leq \Big( \lambda + \sqrt{(1 + c_{\phi})\tau} \, \Big) X(0) \exp \Bigg( - \frac{\mu}{4 \sqrt{(1+c_{\phi}) \tau} \Big( \lambda + \sqrt{(1 + c_{\phi})\tau} \, \Big)} \, t \Bigg), 
\end{equation*}
for all times $t \geq 0$. Recalling the definition \eqref{eq:StandardDev1} of the standard deviation and defining the constants 
\begin{equation*}
\alpha := \Big( \lambda + \sqrt{(1 + c_{\phi})\tau} \, \Big) \qquad \text{and} \qquad \gamma := \frac{1}{4 \sqrt{(1+c_{\phi}) \tau} \Big( \lambda + \sqrt{(1 + c_{\phi})\tau} \, \Big)}, 
\end{equation*}
with $\lambda > 0$ being given as in \eqref{eq:ParameterDef}, we obtain the exponential estimate claimed in \eqref{eq:ExpConsensus1}, which concludes the proof of Theorem \ref{thm:ConsensusSym}.
\end{proof}

%%%%%%%%%%%%%%%%%%%%%%%%%%%%%%%%%%%%%%%%%%%%%%%%%%%%%%%%%%%%%%%%%%%%%%%%%%%%

\subsubsection{Balanced linear interactions topologies}
\label{subsubsection:Balanced}

In this second section devoted to the investigation of $L^2$-consensus formation, we prove the following theorem, which shows that linear balanced graphon dynamics exponentially converge to consensus when the underlying algebraic connectivity is persistent. Throughout this section, we assume that there are no nonlinear interactions between agents -- namely $\phi(\cdot) \equiv 1$ --, so that $\Lbb(t,x) = \Lbb(t)$ for $\Lcal^1$-almost every $t \geq 0$ and any $x \in L^2(I,\R^d)$. 

\begin{thm}[Exponential $L^2$-consensus formation for balanced topologies]
\label{thm:ConsensusBalanced}
Fix an initial datum $x^0 \in L^2(I,\R^d)$, assume that hypothesis \ref{hyp:GD}-$(i)$ holds, and that the interaction kernel $t \in \R_+ \mapsto a(t) \in L^{\infty}(I \times I,[0,1])$ defines \textnormal{balanced} interaction topologies in the sense of Definition \ref{def:Balanced} for $\Lcal^1$-almost every $t \geq 0$. Moreover, suppose that there exists a pair of coefficients $(\tau,\mu) \in \R_+^* \times (0,1]$ such that the persistence condition 
\begin{equation}
\label{eq:PersistenceTopo2}
\frac{1}{\tau} \INTSeg{ \lambda_2(\Lbb(s))}{s}{t}{t+\tau} \geq \mu,
\end{equation}
holds for all times $t \geq 0$. 

Then, there exists a constant $\alpha > 0$ depending only on $(\tau,\mu)$ such that 
\begin{equation}
\label{eq:ExpConsensus2}
\NormL{x(t) - \bar{x}^0}{2}{I,\R^d} \leq \alpha \NormL{x(0) - \bar{x}^0}{2}{I,\R^d} \exp \big(\hspace{-0.05cm} - \mu \hspace{0.02cm} t \big), 
\end{equation}
for all times $t \geq 0$. In particular, the solutions of \eqref{eq:GraphonLapDyn} exponentially converge to consensus in the $L^2(I,\R^d)$-norm topology.
\end{thm}

\begin{proof}
Similarly to what was done in Section \ref{subsubsection:Symmetric} for symmetric graphon models, we will derive dissipative decay estimates on the standard deviation map $X(\cdot)$ defined along solutions of \eqref{eq:GraphonLapDyn} as in \eqref{eq:StandardDev1}. For balanced dynamics, it can be easily checked that
\begin{equation*}
\tderv{}{t}{} \bar{x}(t) = -\INTDom{(\Lbb(t)x(t))(i)}{I}{i} = 0, 
\end{equation*}
as a consequence of Proposition \ref{prop:Bilinear}, which implies that $\bar{x}(t) = \bar{x}^0$ for all times $t \geq 0$. Thus, the time-derivative of the standard deviation can be written as 
\begin{equation}
\label{eq:XdervEst1}
\tderv{}{t}{} X(t) = - \frac{1}{X(t)} \B(\Lbb(t)x(t),x(t)) \leq - \lambda_2(\Lbb(t)) X(t), 
\end{equation}
for $\Lcal^1$-almost every $t \geq 0$, where we assumed without loss of generality that $x(t) \notin \Ccal$. By applying Gr\"onwall's lemma and adapting the computations of \eqref{eq:IneqPersistence2} in the proof of Theorem \ref{thm:ConsensusDiam}, we obtain
\begin{equation*}
X(t) \leq X(0) \exp(- \mu(t-\tau)), 
\end{equation*}
for all times $t \geq 0$. Observing finally that $X(t) = \NormL{x(t) - \bar{x}^0}{2}{I,\R^d}$ in our context, we recover the exponential convergence to consensus \eqref{eq:ExpConsensus2} with $\alpha := \exp(\mu \tau)$.  
\end{proof}

It is important to notice that the persistence conditions imposed in Theorem \ref{thm:ConsensusSym} and Theorem \ref{thm:ConsensusBalanced} are very different in nature. While the latter, written in \eqref{eq:PersistenceTopo2}, expresses the fact that $\lambda_2(\Lbb(\cdot))$ is positive sufficiently often on each time window of length $\tau> 0$, the former merely requires that the algebraic connectivity of the \textit{average} topology is lower-bounded. In particular, it is perfectly possible to build time-dependent signals satisfying \eqref{eq:PersistenceTopo1} for which $\lambda_2(\Lbb(\cdot)) \equiv 0$ (see e.g. \cite[Section 4]{CSComFail} for a finite-dimensional example, as well as Section \ref{section:Discussion} below). This practical and heuristic comparison between \eqref{eq:PersistenceTopo1} and \eqref{eq:PersistenceTopo2} can be made rigorous by means of the following general result. 

\begin{prop}[Concavity of the algebraic connectivity]
\label{prop:Concavity}
The algebraic connectivity defined as in \eqref{eq:AlgebraicConnectivityInf1} and seen as a mapping $\lambda_2 : \Lpazo(L^2(I,\R^d)) \rightarrow \R$ is concave. In particular, given a time-dependent operator $T \in L^2_{\loc}(\R_+,\Lpazo(L^2(I,\R^d)))$ and a coefficient $\tau \in \R_+^*$, it holds that
\begin{equation}
\label{eq:IntegralIneqConcave}
 \frac{1}{\tau}  \INTSeg{\lambda_2(T(s))}{s}{t}{t+\tau} \leq \lambda_2 \bigg( \frac{1}{\tau} \INTSeg{T(s)}{s}{t}{t+\tau}  \bigg) , 
\end{equation}
for all times $t \geq 0$.
\end{prop}

\begin{proof}
Let $\zeta \in [0,1]$ be a fixed parameter, $T_1,T_2 \in \Lpazo(L^2(I,\R^d))$ be two bounded operators and $\epsilon > 0$ be given. By the definition \eqref{eq:AlgebraicConnectivityInf1} of the algebraic connectivity, there exists $x_{\epsilon} \in \Ccal^{\perp}$ such that 
\begin{equation}
\label{eq:ConcavityEst1}
\frac{\big\langle \big( (1-\zeta) T_1 + \zeta T_2 \big) x_{\epsilon} , x_{\epsilon} \big\rangle_{L^2(I,\R^d))}}{\NormL{x_{\epsilon}}{2}{I,\R^d}^2} \leq \lambda_2 \Big( (1-\zeta) T_1 + \zeta T_2 \Big) + \epsilon. 
\end{equation}
Because $x_{\epsilon} \in \Ccal^{\perp}$, it also holds 
\begin{equation}
\label{eq:ConcavityEst2}
\lambda_2(T_1) \leq \frac{\big\langle T_1 x_{\epsilon} , x_{\epsilon} \big\rangle_{L^2(I,\R^d))}}{\NormL{x_{\epsilon}}{2}{I,\R^d}^2} \qquad \text{and} \qquad \lambda_2(T_2) \leq \frac{\big\langle T_2 x_{\epsilon} , x_{\epsilon} \big\rangle_{L^2(I,\R^d))}}{\NormL{x_{\epsilon}}{2}{I,\R^d}^2}, 
\end{equation}
and by combining the estimates of \eqref{eq:ConcavityEst1} and \eqref{eq:ConcavityEst2}, we recover the inequality 
\begin{equation*}
(1-\zeta) \lambda_2(T_1) + \zeta \lambda_2(T_2) \leq \lambda_2 \Big( (1-\zeta) T_1 + \zeta T_2 \Big) + \epsilon, 
\end{equation*}
which holds for every $\zeta \in [0,1]$ and each $T_1,T_2 \in \Lpazo(L^2(I,\R^d))$. Since $\epsilon > 0$ is arbitrary, we conclude that $\lambda_2 : \Lpazo(L^2(I,\R^d)) \rightarrow \R$ is concave. Given a curve $T \in L^2(\R_+,\Lpazo(I,\R^d))$ of operators, the identity \eqref{eq:IntegralIneqConcave} simply follows from an application of the measure-theoretic version of Jensen's inequality (see e.g. \cite[Lemma 1.15]{AmbrosioFuscoPallara}). 
\end{proof}

\begin{open}[Exponential consensus for algebraically persistent balanced topologies]
Do the conclusions of Theorem \ref{thm:ConsensusBalanced} still hold if one replaces the pointwise persistence condition \eqref{eq:PersistenceTopo2} on the algebraic connectivity by the more general one \eqref{eq:PersistenceTopo1} appearing in Theorem \ref{thm:ConsensusSym} ?
\end{open}

%%%%%%%%%%%%%%%%%%%%%%%%%%%%%%%%%%%%%%%%%%%%%%%%%%%%%%%%%%%%%%%%%%%%%%%%%%%%

\subsubsection{Disjoint unions of strongly connected linear topologies}
\label{subsubsection:StronglyConnected}

In this third and last section dealing with the formation of $L^2$-consensus, we prove that a result similar to Theorem \ref{thm:ConsensusBalanced} holds for piecewise constant switching topologies describing disjoint unions of strongly connected components. As in Section \ref{subsubsection:Balanced} above, we make the assumption that $\phi(\cdot) \equiv 1$, so that $\Lbb(t,x) = \Lbb(t)$ for $\Lcal^1$-almost every $t \geq 0$ and any $x \in L^2(I,\R^d)$. 

In what follows, we assume that there exists an increasing sequence $(t_k)_{k \geq 0} \subset \R_+$, satisfying
\begin{equation}
\label{eq:DwellTime}
t_0 = 0 \qquad \text{and} \qquad t_{k+1} - t_k \geq \tau_d, 
\end{equation}
for a given \textit{dwell-time} $\tau_d > 0$, and such that $t \in [t_k,t_{k+1}) \mapsto a(t) \in L^{\infty}(I \times I,[0,1])$ is constant for each $k \geq 0$. We point to the reference monograph \cite[Section 3.2]{Liberzon2003} for a precise discussion of the role of dwell-times in the stability of switched systems.

\begin{thm}[Exponential $L^2$-consensus formation for strongly connected topologies]
\label{thm:ConsensusStrong}
Fix an initial datum $x^0 \in L^2(I,\R^d)$, assume that hypothesis \ref{hyp:GD}-$(i)$ holds and that the signal $t \in \R_+ \mapsto a(t) \in L^{\infty}(I \times I,[0,1])$ is piecewise constant and satisfies the dwell-time condition \eqref{eq:DwellTime}. Moreover, suppose that $a(t)$ defines a \textnormal{disjoint union of strongly connected components} in the sense of Definition \ref{def:Balanced} for $\Lcal^1$-almost every $t \geq 0$. Denoting by $v(t) \in L^{\infty}(I,\R_+^*)$ the canonical vector associated with $a(t) \in L^{\infty}(I \times I,[0,1])$ via Theorem \ref{thm:StrongConGraphon}, we further assume that there exists a triplet of coefficients $(\tau,\mu,\nu) \in \R_+^* \times (0,1]$ satisfying the inequality  
\begin{equation}
\label{eq:PersistenceTopo31}
\frac{2}{\mu \nu^2} \log \Big( \frac{1}{\nu} \Big) < \tau_d, 
\end{equation}
and which are such that
\begin{equation}
\label{eq:PersistenceTopo32}
\frac{1}{\tau} \INTSeg{\lambda_2(\Lbb_v(s))}{s}{t}{t+\tau} \geq \mu, 
\end{equation}
and
\begin{equation}
\label{eq:PersistenceTopo33}
\hspace{5.25cm} \nu \leq v(t,i) \leq \frac{1}{\nu} \qquad \text{for $\LcalI$-almost every $i \in I$}, 
\end{equation}
for all times $t \geq 0$. Here $\Lbb_v(s) := \Mpazo_{v(s)} \Lbb(s) \in \Lpazo(L^2(I,\R^d))$ and the algebraic connectivity $\lambda_2(\Lbb_v(s))$ is defined as in \eqref{eq:AlgebraicConnectivityInf2} for $\Lcal^1$-almost every $s \in [t,t+\tau]$. 

Then, there exists an element $x^{\infty} \in \Cpazo(x^0)$ and a constant $\alpha > 0$ depending only on $(\tau,\mu,\nu)$ such that 
\begin{equation}
\label{eq:ExpConsensus3}
\Norm{x(t) - x^{\infty}}_{L^2(I,\R^d)} \, \leq \alpha \NormL{x(0) - \bar{x}^0}{2}{I,\R^d} \exp \Big( - \big( \mu \nu^2 - \tfrac{2}{\tau_d} \log \big( \tfrac{1}{\nu} \big) \big) t \, \Big), 
\end{equation}
for all times $t \geq 0$. In particular, the solutions of \eqref{eq:GraphonLapDyn} exponentially converge to consensus in the $L^2(I,\R^d)$-norm topology. 
\end{thm}

A discussion shedding light on the lower-bound imposed on $v(t)$ in \eqref{eq:PersistenceTopo33} is proposed in Remark \ref{rmk:LowerBound}, based on the results of Section \ref{subsection:Connectivity}.

\begin{proof}[Proof of Theorem \ref{thm:ConsensusStrong}]
Let it first be noted that $\Lbb_v(t) x \in \Ccal^{\perp}$ for $\Lcal^1$-almost every $t \geq 0$ and each $x \in L^2(I,\R^d)$, since 
\begin{equation}
\label{eq:v_average}
\begin{aligned}
\INTDom{(\Lbb_v(t)x)(i)}{I}{i} & = \INTDom{\INTDom{v(t,i)a(t,i,j)(x(i) - x(i))}{I}{j}}{I}{i} \\
& = \INTDom{\bigg(  \INTDom{\Big( a(t,i,j)v(t,i) - a(t,j,i) v(t,j) \Big)}{I}{j} \bigg) x(i)}{I}{i} = 0, 
\end{aligned}
\end{equation}
by Fubini's theorem along with the fact that $\Lbb(t)^*v(t) = 0$. Therefore, the weighted barycenters  
\begin{equation*}
\bar{x}_v : t \in \R_+ \mapsto \INTDom{v(t,i)x(t,i)}{I}{i} \in \R^d, 
\end{equation*}
associated with $v(\cdot)$ are constant on each time interval of the form $[t_k,t_{k+1})$ with $k \geq 0$. Following the proof strategy developed for Theorem \ref{thm:ConsensusBalanced}, we introduce weighted variance bilinear forms 
\begin{equation*}
\B_{v(t)}(x,y) := \langle \Mpazo_{v(t)} \, x,y \rangle_{L^2(I,\R^d)} - \langle \bar{x}_{v(t)}, \bar{y}_{v(t)} \rangle, 
\end{equation*}
defined for each $x,y \in L^2(I,\R^d)$ and $\Lcal^1$-almost every $t \geq 0$. In this context, one can check that
\begin{equation*}
\B_{v(t)} \Big( x(t) - \bar{x}_v(t),x(t) - \bar{x}_v(t) \Big) = \B_{v(t)}(x(t),x(t)), 
\end{equation*}
where both equalities rely on the fact that $\INTDom{v(t,i)}{I}{i} = 1$ for all times $t \geq 0$ by construction. Similarly, one can again use \eqref{eq:v_average} to verify that
\begin{equation*}
\langle \Lbb_v(t) x,x \rangle_{L^2(I,\R^d)} = \B_{v(t)}(\Lbb(t)x,x),  
\end{equation*}
for each $x \in L^2(I,\R^d)$ and $\Lcal^1$-almost every $t \geq 0$, so that the algebraic connectivity defined in \eqref{eq:AlgebraicConnectivityInf2} can be rewritten as 
\begin{equation*}
\lambda_2(\Lbb_v(t)) = \inf_{x \notin \Ccal} \frac{\B_{v(t)}(\Lbb(t)x,x)}{\B(x,x)} \geq 0.
\end{equation*}
The inequality on the right comes from the fact that $\Lbb(t)$ is positive semi-definite with respect to $\B_{v(t)}(\cdot,\cdot)$, which can be shown by a small modification of the argument in the proof of Proposition \ref{prop:Bilinear}. Finally, we introduce the piecewise constant weighted standard deviations, which are defined by 
\begin{equation*}
X_v(t) := \sqrt{\B_{v(t)}(x(t),x(t))}, 
\end{equation*}
for all times $t \geq 0$. Using again the fact that $\INTDom{v(t,i)}{I}{i} = 1$ for all times $t \geq 0$, one can show via routine computations that the weighted standard deviation can be alternatively expressed as 
\begin{equation*}
X_v(t) = \bigg( \INTDom{\INTDom{v(t,i) v(t,j) |x(t,i) - x(t,j)|^2}{I}{j}}{I}{i} \bigg)^{1/2}.
\end{equation*} 
Together with the bounds \eqref{eq:PersistenceTopo33} imposed on $v(t) \in L^{\infty}(I,\R_+^*)$, this allows us to recover the estimates
\begin{equation}
\label{eq:StrongExponential3}
\nu X(t) \leq X_v(t) \leq \frac{1}{\nu} X(t), 
\end{equation}
for all times $t \geq 0$, where $X(\cdot)$ stands for the usual standard deviation defined as in \eqref{eq:StandardDev1}.

Observe now that for each $k \geq 0$, the map $X_v(\cdot)$ is differentiable over $[t_k,t_{k+1})$, with 
\begin{equation}
\label{eq:StrongExponential1}
\begin{aligned}
\tderv{}{t}{} X_v(t) & = - \frac{1}{X_v(t)} \B_{v(t)}(\Lbb(t)x(t),x(t)) \leq - \frac{1}{X_v(t)} \lambda_2(\Lbb_v(t)) X(t)^2 \leq -  \nu^2 \lambda_2(\Lbb_v(t)) X_v(t). 
\end{aligned}
\end{equation}
Then, by a simple application of Gr\"onwall's lemma, it follows that
\begin{equation}
\label{eq:StrongExponential2}
X_v(t) \leq X_v(t_k) \exp\Big( \hspace{-0.05cm} - \nu^2 \lambda_2(\Lbb_v(t_k))(t-t_k) \Big), 
\end{equation}
for all times $t \in [t_k,t_{k+1})$ and each $k \geq 0$. It is worth noting that $X_v(\cdot)$ is in fact a discontinuous Lyapunov function, so that \eqref{eq:StrongExponential2} cannot be used directly to prove the convergence of the system to consensus. By combining \eqref{eq:StrongExponential2} and \eqref{eq:StrongExponential3}, we obtain the dissipation inequality 
\begin{equation*}
X(t) \leq \frac{1}{\nu} X(t_k) \exp\Big( \hspace{-0.05cm} -\nu^2 \lambda_2(\Lbb_v(t_k))(t-t_k) \Big), 
\end{equation*}
which holds for all times $t \in [t_k,t_{k+1})$ and each $k \geq 0$. Thus, by performing a simple induction argument and repeating the computations in \eqref{eq:IneqPersistence2} involving the persistence condition \eqref{eq:PersistenceTopo32}, we can in turn recover that
\begin{equation}
\label{eq:StrongExponential4}
X(t) \leq X(0) \frac{1}{\nu^{2(k+1)}} \exp \bigg( - \nu \INTSeg{\lambda_2(\Lbb_v(s))}{s}{0}{t} \bigg) \leq \tfrac{\exp(\mu \nu^2 \tau)}{\nu^2} X(0) \frac{\exp \big( - \mu \nu^2 \hspace{0.02cm} t \big)}{\nu^{2 t /\tau_d}}, 
\end{equation}
for all $t \in [t_k,t_{k+1})$, where we used the fact that $t_k \geq k \tau_d$ for each $k \geq 0$ as a consequence of the dwell-time assumption \eqref{eq:DwellTime}. Finally, owing to the constraint \eqref{eq:PersistenceTopo31} imposed on the dwell-time and connectivity coefficients, it holds in particular that
\begin{equation}
\label{eq:StrongExponential5}
X(t) \leq \tfrac{\exp(\mu \nu^2 \tau)}{\nu^2} X(0) \exp \Big( - \big( \mu \nu^2 - \tfrac{2}{\tau_d} \log(\tfrac{1}{\nu} \big) \big) t \, \Big) ~\underset{t \rightarrow +\infty}{\longrightarrow}~ 0, 
\end{equation}
so that the standard deviation $X(\cdot)$ exponentially converges to zero as $t \rightarrow +\infty$.  

Our goal now is to show that $x(\cdot)$ converges to consensus. To this end, consider two arbitrary instants $t_2  \geq t_1 \geq 0$, and observe that the classical barycenters of the curve $x(\cdot)$ are such that
\begin{equation}
\label{eq:StrongExponential6}
\begin{aligned}
|\bar{x}(t_2) - \bar{x}(t_1)| & \leq \INTDom{|x(t_2,i) - x(t_1,i)|}{I}{i} \\
& \leq \INTSeg{ \bigg( \INTDom{\INTDom{a(s,i,j) |x(s,j) - x(s,i)|}{I}{j}}{I}{i} \bigg)}{s}{t_1}{t_2}  \leq \INTSeg{X(s)}{s}{t_1}{+\infty} ~\underset{t_1 \rightarrow +\infty}{\longrightarrow}~ 0, 
\end{aligned} 
\end{equation}
where we used Fubini's theorem and Cauchy-Schwarz's inequality along with the decay estimate of \eqref{eq:StrongExponential5}. This implies that the curve $\bar{x}(\cdot) \in C^0(\R_+,\R^d)$ converges exponentially to some limit point $x^{\infty} \in \R^d$ as $t \rightarrow +\infty$. By combining \eqref{eq:StrongExponential5} and \eqref{eq:StrongExponential6}, we can conclude that 
\begin{equation*}
\NormL{x(t) - x^{\infty}}{2}{I,\R^d} \leq X(t) \, + \NormL{\bar{x}(t) - x^{\infty}}{2}{I,\R^d} \leq \alpha X(0) \exp \Big( - \big( \mu \nu^2 - \tfrac{2}{\tau_d} \log(\tfrac{1}{\nu} \big) \big) t \, \Big), 
\end{equation*}
for a given constant $\alpha > 0$ that only depends on the coefficients $(\tau,\mu,\nu)$, so that solutions of \eqref{eq:GraphonLapDyn} exponentially converge to consensus in the $L^2(I,\R^d)$-norm topology. 
\end{proof}

\begin{open}[Asymptotic consensus formation without dwell-times]
Would a weaker form of Theorem \ref{thm:ConsensusStrong} still hold for general time-varying signals that do not satisfy a dwell-time condition? 
\end{open}

%%%%%%%%%%%%%%%%%%%%%%%%%%%%%%%%%%%%%%%%%%%%%%%%%%%%%%%%%%%%%%%%%%%%%%%%%%%%%%%%
%								NEW SECTION AHEAD							   %	
%%%%%%%%%%%%%%%%%%%%%%%%%%%%%%%%%%%%%%%%%%%%%%%%%%%%%%%%%%%%%%%%%%%%%%%%%%%%%%%%

\section{Further discussions around connectivity and consensus formation}
\label{section:Discussion}
\setcounter{equation}{0} \renewcommand{\theequation}{\thesection.\arabic{equation}}

In what follows, we elaborate on the properties of the algebraic connectivity defined in Section \ref{subsection:Algebraic}, as well as on its interplay with consensus formation. In Section \ref{subsection:Connectivity}, we prove that interaction topologies which are either symmetric or written as disjoint unions of strongly connected components are strongly connected if and only if their algebraic connectivity is positive. Then, in Section \ref{subsection:L2Linfty}, we prove that the formation of $L^2$- and $L^{\infty}$-consensus are equivalent when the in-degree function is persistent in a suitable sense, and close the manuscript by a series of numerical illustrations presented in Section \ref{subsection:Numerics} 

%%%%%%%%%%%%%%%%%%%%%%%%%%%%%%%%%%%%%%%%%%%%%%%%%%%%%%%%%%%%%%%%%%%%%%%%%%%%%%%%

\subsection{Link between algebraic and graphon connectivity}
\label{subsection:Connectivity}

In Section \ref{subsection:L2Consensus}, we have shown that the notion of algebraic connectivity proposed in Section \ref{subsection:Algebraic} allows for a systematic investigation of consensus formations for various kinds of time-dependent graphon models. As recalled above, the classical notion of algebraic connectivity encodes important structural properties of the underlying graph, including information about its connectivity. Our aim in this section is to discuss the counterpart of these results for graphon models. 

We start by a technical lemma concerning the existence of a largest strongly connected component containing a point in a graphon. In the sequel, for the sake of conciseness, we shall say that two points $i,j \in I$ are \textit{$m$-connected} if $i$ can be linked to $j$ and reciprocally $j$ can be linked to $i$ in the sense of Definition \ref{def:StrongCon}-$(a)$ by directed chains of Lebesgue points of length at most $m \geq 1$ . 

\begin{lem}[Existence of largest strongly connected components]
\label{lem:StrongComponent}
Let $\Apazo \in \Lpazo(L^2(I,\R^d))$ and $i \in I$ be a Lebesgue point of the interaction kernel $a \in L^{\infty}(I \times I,[0,1])$. Then, there exists a strongly connected set $\Spazo(i) \subset I$ in the sense of Definition \ref{def:StrongCon}-$(a)$ that contains $i$, and which is maximal in the sense that 
\begin{equation*}
\LcalI(J \setminus \Spazo(i)) = 0, 
\end{equation*}
for any other strongly connected set $J \subset I$ containing $\Spazo(i)$. 
\end{lem}

\begin{proof}
Given $N \geq 1$, let $\Spazo_N(i) \in \Ppazo(I)$ be defined as
\begin{equation*}
\Spazo_N(i) := \Big\{ j \in I ~\text{s.t}~~ \text{$i$ and $j$ are $m$-connected for some $m \leq N$} \Big\}, 
\end{equation*}
and consider the union over $N \geq 1$ of these sets
\begin{equation*}
\Spazo(i) := \bigcup_{N \geq 1} \Spazo_N(i).
\end{equation*}
By construction, it is clear that $\Spazo(i)$ is a strongly connected subset of $I$ in the sense of Definition \ref{def:StrongCon}-$(a)$. Let $J \subset I$ be a strongly connected set containing $\Spazo(i)$ and suppose that $\LcalI(J \setminus \Spazo(i)) > 0$, which can be equivalently rewritten as
\begin{equation}
\label{eq:ContradictionStrong}
\LcalI \bigg( J \, \setminus \, \Big( \bigcup_{N \geq 1} \Spazo_N(i) \Big) \bigg) = \LcalI \bigg( \bigcap_{N \geq 1} \Big( J \setminus \Spazo_N(i) \Big) \bigg) = \lim_{N \rightarrow + \infty} \LcalI(J \setminus \Spazo_N(i)) > 0, 
\end{equation}
since the sequence of sets $(\Spazo_N(i))_{N \geq 1} \subset \Ppazo(I)$ is increasing. Then for every $N \geq 1$ large enough, the identity \eqref{eq:ContradictionStrong} implies that there exists a subset $J_N \subset J$ satisfying $\LcalI(J_N) \geq \LcalI(J \setminus \Spazo(i)) > 0$, whose elements cannot be connected to $i$ by a path of length at most $N \geq 1$. This contradicts our primary assumption that $J$ is strongly connected, and shows that $\Spazo(i)$ is maximal.
\end{proof}

This technical result being established, we can provide a first correspondence between the positivity of the algebraic connectivity and graphon connectivity for symmetric interaction topologies. 

\begin{thm}[On the connectivity of undirected graphons]
\label{thm:UndirectedConnectivity}
Let $\Apazo \in \Lpazo(L^2(I,\R^d))$ be a symmetric adjacency operator and $\Lbb \in \Lpazo(L^2(I,\R^d))$ be the corresponding graph-Laplacian. Then, the interaction topology described by $\Apazo$ is strongly connected in the sense of Definition \ref{def:StrongCon} if and only if $\lambda_2(\Lbb) > 0$. 
\end{thm}

\begin{proof}
\textbf{Converse implication.} We start by assuming that $\Apazo \in \Lpazo(L^2(I,\R^d))$ is an adjacency operator for which $\lambda_2(\Lbb) > 0$, and we will prove that condition $(b)$ of Definition \ref{def:StrongCon} holds, namely $\inf_{i \in I} \INTDom{a(i,j)}{I}{j} > 0$. We recall that the in-degree function $d \in L^{\infty}(I,[0,1])$ associated with $\Apazo$ writes
\begin{equation*}
d(i) := \INTDom{a(i,j)}{I}{j}, 
\end{equation*}
for $\LcalI$-almost every $i \in I$, and that $\sigma_{\ess}(\Mpazo_d) = \rg(d)$ as a consequence of Proposition \ref{prop:Spectrum}. Moreover, let it be noted that  
\begin{equation*}
\sigma_{\ess}(\Lbb) = \sigma_{\ess}(\Mpazo_d - \Apazo) = \sigma_{\ess}(\Mpazo_d), 
\end{equation*}
because $\Apazo \in \Lpazo(L^2(I,\R^d))$ is a compact operator (see e.g. \cite[Section VI.5]{ReedI1981}), and the essential spectrum is invariant under compact perturbations (see e.g. \cite[Example XIII.4.3]{ReedIV1978}). Then, by the min-max theorem for self-adjoint operators (see e.g. \cite[Theorem XIII.1]{ReedIV1978}), we obtain the lower-bound
\begin{equation*}
0 < \lambda_2(\Lbb) \leq \inf \Big\{ \lambda \in \R ~\text{s.t.}~ \lambda \in \sigma_{\ess}(\Lbb) \Big\},
\end{equation*}
on the essential spectrum of the graph-Laplacian operator $\Lbb$. In particular, it necessarily holds that $\inf_{i \in I} d(i) = \inf_{i \in I} \INTDom{a(i,j)}{I}{j} > 0$ whenever $\lambda_2(\Lbb) > 0$. 

We turn our attention to condition $(a)$ of Definition \ref{def:StrongCon}. Assume by contradiction that it is violated, namely that the disconnection set
\begin{equation}
\label{eq:Disconnection}
\begin{aligned}
& \Dcal := \bigg\{ (i,j) \in I \times I ~~\text{s.t. $i$ and $j$ are not $m$-connected for each $m \geq 1$} \bigg\}, 
\end{aligned}
\end{equation}
has positive $\LcalII$-measure. We claim that in this case, there exists a subset $\Spazo \subset I$ such that 
\begin{equation*}
\LcalI(\Spazo) > 0, \quad \LcalI(I \setminus \Spazo) > 0 \quad \text{and} \quad a(i,j) = 0 ~~ \text{for $\LcalII$-almost every $(i,j) \in \Spazo \times I \setminus \Spazo$}.
\end{equation*}
To prove this assertion, we consider for each $i \in I$ the (possibly empty) sliced disconnection set
\begin{equation*}
\Dcal(i) := \Big\{ j \in I ~\text{s.t.}~ (i,j) \in \Dcal \Big\}, 
\end{equation*}
and observe that by Fubini's theorem, there exists a subset $J \subset I$ of positive $\LcalI$-measure such that $\LcalI(\Dcal(i)) > 0$ for each $i \in J$. Given an element $i \in J$, let $\Spazo(i) \subset I$ be the largest strongly connected component containing $i$, which exists by Lemma \ref{lem:StrongComponent}, and notice that 
\begin{equation*}
\LcalI(\Spazo(i)) > 0, 
\end{equation*}
since $\inf_{i \in I} \INTDom{a(i,j)}{I}{j} > 0$ and $a \in L^{\infty}(I \times I,[0,1])$ is symmetric. In addition, one also has that $\LcalI(\Spazo(i) \cap \Dcal(i)) = 0$ by definition \eqref{eq:Disconnection} of the disconnection set, which then implies  
\begin{equation*}
\LcalI(I \setminus \Spazo(i)) > 0.
\end{equation*}
Finally by the maximality of $\Spazo(i)$, it necessarily holds that 
\begin{equation*}
a(k,l) = 0 \qquad \text{for $\LcalII$-almost every $(k,l) \in \Spazo(i) \times I \setminus \Spazo(i)$}, 
\end{equation*}
as we could otherwise build a strictly larger strongly connected set containing $\Spazo(i)$. Armed with this construction, we consider the nontrivial element $x \in \Ccal^{\perp} \subset L^2(I,\R^d)$ defined by 
\begin{equation}
\label{eq:KernelConstruction}
x := \frac{1}{\LcalI(\Spazo(i))} \mathds{1}_{\Spazo(i)} - \frac{1}{\LcalI(I \setminus \Spazo(i))} \mathds{1}_{I \setminus \Spazo(i)}, 
\end{equation}
and observe that 
\begin{equation*}
\begin{aligned}
\langle \Lbb \, x , x \rangle_{L^2(I,\R^d)} & = \INTDom{\INTDom{a(k,l) \langle x(k) , x(k) - x(l) \rangle}{I}{l}}{I}{k} \\
& = \INTDom{\INTDom{a(k,l) \langle x(k) , x(k) - x(l) \rangle}{I}{l}}{\Spazo(i)}{k} + \INTDom{\INTDom{a(k,l) \langle x(k) , x(k) - x(l) \rangle}{I}{l}}{I \setminus \Spazo(i)}{k} \\
& = \INTDom{\INTDom{a(k,l) \langle x(k) , x(k) - x(l) \rangle}{\Spazo(i)}{l}}{\Spazo(i)}{k} + \INTDom{\INTDom{a(k,l) \langle x(k) , x(k) - x(l) \rangle}{I \setminus \Spazo(i)}{l}}{I \setminus \Spazo(i)}{k} \\
& = 0,  
\end{aligned}
\end{equation*}
which contradicts the fact that $\lambda_2(\Lbb) > 0$. Thus, we can conclude that the interaction topology defined by a symmetric adjacency operator $\Apazo \in \Lpazo(L^2(I,\R^d))$ is strongly connected as soon as $\lambda_2(\Lbb) > 0$. 

\paragraph*{Direct implication.} We now assume that the interaction topology defined by the symmetric adjacency operator $\Apazo \in \Lpazo(L^2(I,\R^d))$ is strongly connected in the sense of Definition \ref{def:StrongCon}. By Definition \ref{def:Spectrum} and Proposition \ref{prop:Spectrum}, the fact that $\delta := \inf_{i \in I} \INTDom{a(i,j)}{I}{j} > 0$ implies that $\sigma_{\ess}(\Lbb) \subset [\delta,1]$. Recalling that the elements of $\Ccal$ belong to the kernel of $\Lbb$, it follows from the min-max theorem that either $\lambda_2(\Lbb) > 0$, or there exists a nontrivial element $x \in \Ccal^{\perp}$ for which $\langle \Lbb \, x,x \rangle_{L^2(I,\R^d)} = 0$. By repeating the general argument of \cite[Section 3.1]{Boudin2022}, it can however be shown that the latter case is impossible.
\end{proof}

We end this section by showing how the proof strategy devised for the proof of Theorem \ref{thm:StrongConGraphon} generalises to topologies representable as a disjoint union of strongly connected components. 

\begin{thm}[On the connectivity of directed graphons]
\label{thm:DirectedConnectivity}
Let $\Apazo \in \Lpazo(L^2(I,\R^d))$ be an adjacency operator defining a disjoint union of strongly components in the sense of Definition \ref{def:StrongCon}, and $v \in L^{\infty}(I,\R_+^*)$ be the canonical eigenvector given by Theorem \ref{thm:StrongConGraphon}. Then, the interaction topology described by $\Apazo$ is strongly connected if and only if $\lambda_2(\Lbb_v) > 0$, and in this case it holds that $\inf_{i \in I} v(i) \geq \lambda_2(\Lbb_v)$. 
\end{thm}

\begin{proof}
As already explained in the proof of Theorem \ref{thm:UndirectedConnectivity} above, the direct implication of our claim can be deduced from the arguments of \cite[Section 3.1]{Boudin2022} for general strongly connected graphons. Hence, we will focus on the converse one and assume that $\lambda_2(\Lbb_v) > 0$. Observe first that the definition \eqref{eq:AlgebraicConnectivityInf2} of the algebraic connectivity can be rewritten as 
\begin{equation*}
\lambda_2(\Lbb_v) = \inf_{x \in \Ccal^{\perp}} \frac{\langle \Lbb_v \, x , x \rangle_{L^2(I,\R^d)}}{\Norm{x}_{L^2(I,\R^d)}^2} = \inf_{x \in \Ccal^{\perp}} \frac{\langle  \Lbb_v^{\sym} x , x \big\rangle_{L^2(I,\R^d)}}{\Norm{x}_{L^2(I,\R^d)}^2}, 
\end{equation*}
where we denoted by $\Lbb_v^{\sym} = \tfrac{1}{2}(\Lbb_v + \Lbb_v^*) \in \Lpazo(L^2(I,\R^d))$ the symmetric part of the operator $\Lbb_v$. The latter is a self-adjoint operator by construction, whose expression can be given explicitly as 
\begin{equation*}
\Lbb_v^{\sym} = \Mpazo_v \Mpazo_d - \tfrac{1}{2} \Big( \Mpazo_v \, \Apazo + \Apazo^* \Mpazo_v \Big), 
\end{equation*}
where we recall that $\Mpazo_d \in \Lpazo(L^2(I,\R^d))$ is the multiplication operator by the in-degree function defined as in \eqref{eq:InDegreeDef}. Moreover, it can be checked straightforwardly that the operator $\Mpazo_v \, \Apazo + \Apazo^* \Mpazo_v$ is compact, and by reproducing the argument detailed in the proof of Theorem \ref{thm:UndirectedConnectivity} above, one has 
\begin{equation*}
\sigma_{\ess}(\Lbb_v^{\sym}) = \sigma_{\ess}(\Mpazo_v \Mpazo_d) = \rg(v d).
\end{equation*}
Upon noticing that $\Lbb_v^* \, x = 0$ for each $x \in \Ccal$, we obtain by an application of the min-max theorem 
\begin{equation}
\label{eq:DirectedConnectivity1}
0 < \lambda_2(\Lbb_v) \leq \inf_{i \in I} v(i) \INTDom{a(i,j)}{I}{j}.
\end{equation}
Recalling that $\NormL{v}{\infty}{I,\R_+^*} \leq 1/\delta$ where $\delta \in (0,1]$ is such that \eqref{eq:LowerBoundDisjoint} of Definition \ref{def:StrongCon} holds, one has
\begin{equation*}
\inf_{i \in I} \INTDom{a(i,j)}{I}{j} \geq \delta \lambda_2(\Lbb_v) > 0, 
\end{equation*}
which yields Definition \ref{def:StrongCon}-$(b)$. Using the fact that $a \in L^{\infty}(I \times I,[0,1])$, we can also deduce from \eqref{eq:DirectedConnectivity1} that 
\begin{equation*}
\inf_{i \in I} v(i) \geq \lambda_2(\Lbb_v) > 0. 
\end{equation*}
There now remains to show that the conditions of Definition \ref{def:StrongCon}-$(a)$ holds. To this end, we again assume by contradiction that the disconnection set $\Dcal \subset I \times I$ defined as in \eqref{eq:Disconnection} has positive $\LcalII$-measure, and choose an element $i \in I$ for which $\LcalI(\Dcal(i)) > 0$ where $\Dcal(i) = \{ j \in I ~\text{s.t.}~ (i,j) \in \Dcal \}$. Observe that by construction, $i \in I_N$ for some $N \geq 1$. Thus by Lemma \ref{lem:StrongComponent}, the largest strongly connected component $\Spazo(i) \subset I$ containing $i$ is necessarily given by $\Spazo(i) = I_N$, and
\begin{equation*}
\LcalI(\Spazo(i)) = \LcalI(I_N) > 0.
\end{equation*}
Moreover, the fact that $\Lcal(\Spazo(i) \cap \Dcal(i)) = 0$ yields
\begin{equation*}
\LcalI(I \setminus \Spazo(i)) > 0,
\end{equation*}
and our standing assumption that the interaction topology is a disjoint union of strongly connected components further means 
\begin{equation*}
a(k,l) = 0 \qquad \text{for $\LcalII$-almost every $(k,l) \in \Spazo(i) \times I \setminus \Spazo(i)$}. 
\end{equation*}
Then, by defining the non-trivial vector $x \in \Ccal^{\perp}$ as in \eqref{eq:KernelConstruction}, we can conclude that the condition $\lambda_2(\Lbb_v) = 0$ is violated, which ends the proof of our claim. 
\end{proof}

\begin{rmk}[On the lower-bound \eqref{eq:PersistenceTopo33} imposed on the canonical eigenvectors in Theorem \ref{thm:ConsensusStrong}] 
\label{rmk:LowerBound}
By applying Theorem \ref{thm:DirectedConnectivity} on every strongly connected set $I_n$ while using \eqref{eq:DirectedConnectivity1}, it can be checked that 
\begin{equation*}
0 < \tfrac{1}{\LcalI(I_n)} \lambda_2(\Lbb_{v_n}) \leq \inf_{i_n \in I_n} v_n(i_n), 
\end{equation*}
where $\Lbb_{v_n} := \Mpazo_{v_n} \, \Lbb_n \in \Lpazo(L^2(I_n,\R^d))$ denotes the restriction of the weighted graph-Laplacian to $I_n$ for each $n \geq 1$. Thus, the lower-bound \eqref{eq:PersistenceTopo33} holds whenever there exists a constant $\nu > 0$ such that 
\begin{equation*}
\lambda_2(\Lbb_{v_n}) \geq \nu \LcalI(I_n), 
\end{equation*}
for each $n \geq 1$. This condition -- that is in fact reminiscent of \eqref{eq:LowerBoundDisjoint} -- heuristically means that the total amount of interactions occurring within each of the countably many strongly connected components must be comparable to their size.  
\end{rmk}

%%%%%%%%%%%%%%%%%%%%%%%%%%%%%%%%%%%%%%%%%%%%%%%%%%%%%%%%%%%%%%%%%%%%%%%%%%%%%%%%

\subsection{Equivalence between $L^2$- and $L^{\infty}$-consensus for strongly connected graphons}
\label{subsection:L2Linfty}

In this section, we prove an interesting and quite unexpected feature of consensus formation in graphon models, namely that convergence towards consensus in the $L^2$-norm topology implies convergence in the $L^{\infty}$-norm topology as soon as an average version of the lower-bound on the in-degree function imposed in Definition \ref{def:StrongCon}-$(b)$ holds.

\begin{thm}[Equivalence between $L^2$- and $L^{\infty}$-consensus formation]
\label{thm:Equivalence}
Let $x^0 \in L^{\infty}(I,\R^d)$ be such that $\NormL{x^0}{\infty}{I,\R^d} \leq R$ for some $R > 0$, assume that hypotheses \ref{hyp:GD} hold, and denote by $x(\cdot) \in \Lip_{\loc}(\R_+,L^2(I,\R^d))$ the corresponding solution of \eqref{eq:GraphonDynamics}. Moreover, suppose that there exists a pair of coefficients $(\tau,\mu) \in \R_+^* \times (0,1]$ such that the persistence condition
\begin{equation}
\label{eq:PersistenceInfimum}
\frac{1}{\tau} \INTSeg{\INTDom{a(s,i,j)}{I}{j}}{s}{t}{t+\tau} \geq \mu, 
\end{equation}
holds for all times $t \geq 0$ and $\LcalI$-almost every $i \in I$. Then, one has that
\begin{equation}
\label{eq:EquivalenceConvergence}
\NormL{x(t) - x^{\infty}}{\infty}{I,\R^d} ~\underset{t \rightarrow +\infty}{\longrightarrow}~ 0 \qquad \text{if and only if} \qquad \NormL{x(t) - x^{\infty}}{2}{I,\R^d} ~\underset{t \rightarrow +\infty}{\longrightarrow}~ 0,
\end{equation}
for any given element $x^{\infty} \in \Cpazo(x^0)$.
\end{thm}

\begin{proof}
The direct implication in \eqref{eq:EquivalenceConvergence} is trivial since the $L^2$-norm is bounded from above by the $L^{\infty}$-norm on the compact set $I := [0,1]$. Hence, we only need to prove the converse implication. To this end, we define for every $\epsilon > 0$ and $\Lcal^1$-almost every $t \geq 0$ the sets 
\begin{equation*}
I_{\epsilon}(t) := \Big\{ i \in I ~\text{s.t.}~ |x(t,i) - x^{\infty}| \leq \epsilon \Big\}, 
\end{equation*}
and observe that by Chebyshev's inequality (see e.g. \cite[Remark 1.18]{AmbrosioFuscoPallara}), one has that
\begin{equation*}
\LcalI(I \setminus I_{\epsilon}(t)) = \LcalI \Big( \Big\{ i \in I ~\text{s.t.}~ |x(t,i) - x^{\infty}| > \epsilon \Big\}\Big) \, \leq \, \frac{1}{\epsilon^2} \INTDom{|x(t,i) - x^{\infty}|^2}{I}{i}, 
\end{equation*}
for every $\epsilon > 0$. Hence by the $L^2$-convergence posited in the right-hand side of \eqref{eq:EquivalenceConvergence}, there exists a time horizon $T_{\epsilon} \geq 0$ such that $\LcalI(I \setminus I_{\epsilon}(t)) \leq \epsilon^2$ for $\Lcal^1$-almost every $t \geq T_{\epsilon}$. Recalling that $x(\cdot)$ is a solution of the graphon dynamics \eqref{eq:GraphonDynamics}, it also holds
\begin{equation*}
\begin{aligned}
\tfrac{1}{2} \tderv{}{t}{} |x(t,i) - x^{\infty}|^2 & = \INTDom{a(t,i,j) \phi(|x(t,i) - x(t,j)|) \big\langle x(t,i) - x^{\infty} , x(t,j) - x(t,i) \big\rangle}{I}{j} \\
& = - \bigg( \INTDom{a(t,i,j) \phi(|x(t,i) - x(t,j)|)}{I}{j} \bigg) |x(t,i) - x^{\infty}|^2 \\
& \hspace{0.45cm} + \INTDom{a(t,i,j) \phi(|x(t,i) - x(t,j)|) \big\langle x(t,i) - x^{\infty} , x(t,j) - x^{\infty} \big\rangle}{I}{j} \\
& \leq -\gamma_R \bigg( \INTDom{a(t,i,j)}{I}{j} \bigg) |x(t,i) - x^{\infty}|^2  \\
& \hspace{0.45cm} + c_{\phi} \bigg( \INTDom{|x(t,j) - x^{\infty}|}{I_{\epsilon}(t)}{j} + \INTDom{|x(t,j) - x^{\infty}|}{I \setminus I_{\epsilon}(t)}{j} \bigg) |x(t,i) - x^{\infty}| \\
& \leq -\gamma_R \bigg( \INTDom{a(t,i,j)}{I}{j} \bigg) |x(t,i) - x^{\infty}|^2 + c_{\phi} \epsilon R(1 + \epsilon R),
\end{aligned} 
\end{equation*}
for $\Lcal^1$-almost every $t \geq T_{\epsilon}$ and $\LcalI$-almost every $i \in I$, where we recall that $c_{\phi} = \sup_{r \in \R_+} \phi(r) < +\infty$ and $\gamma_R = \min_{r \in [0,R]} \phi(2r)$. Then by Gr\"onwall's lemma, one further has
\begin{equation*}
\begin{aligned}
|x(t,i) - x^{\infty}|^2 & \leq |x(T_{\epsilon},i) - x^{\infty}|^2 \exp \bigg( - 2 \gamma_R \INTSeg{\INTDom{a(s,i,j)}{I}{j}}{s}{T_{\epsilon}}{t} \bigg) \\
& \hspace{0.45cm} + c_{\phi} \epsilon R(1 + \epsilon R) \INTSeg{\exp \Bigg( - 2 \gamma_R \INTSeg{\INTDom{a(\sigma,i,j)}{I}{j}}{\sigma}{s}{t}\Bigg)}{s}{T_{\epsilon}}{t}, 
\end{aligned}
\end{equation*}
and by using the persistence condition \eqref{eq:PersistenceInfimum} while repeating the computations of \eqref{eq:IneqPersistence2}, we obtain
\begin{equation}
\label{eq:PointwiseEst1}
\begin{aligned}
|x(t,i) - x^{\infty}|^2 & \leq R^2 \exp \Big( - 2 \gamma_R \mu(t-T_{\epsilon} - \tau)  \Big) \\
& \hspace{0.45cm} +  c_{\phi} \epsilon R(1 + \epsilon R) \exp(\gamma_R \mu \tau) \Bigg( \frac{1 - \exp \big( -2 \gamma_R \mu (t - T_{\epsilon}) \big) }{\gamma_R \mu} \Bigg) \\
& \leq R^2 \exp \Big( - 2 \gamma_R \mu(t-T_{\epsilon} - \tau)  \Big) + \tfrac{\exp(\gamma_R \mu \tau)}{\gamma_R \mu} c_{\phi} \epsilon R(1 + \epsilon R),  
\end{aligned}
\end{equation}
for all times $t \geq T_{\epsilon}$ and $\LcalI$-almost every $i \in I$. Thus given a parameter $\delta > 0$, we can choose $\epsilon > 0$ and then $t \geq T_{\epsilon}$ in such a way that 
\begin{equation*}
\tfrac{\exp(\gamma_R \mu \tau)}{\gamma_R \mu} c_{\phi} \epsilon R(1 + \epsilon R) \leq \frac{\delta^2}{2} \qquad \text{and} \qquad R^2 \exp \Big( - 2 \gamma_R \mu(t-T_{\epsilon} - \tau)  \Big) \leq \frac{\delta^2}{2}, 
\end{equation*}
which means that we can find a time horizon $T_{\delta} \geq T_{\epsilon}$ for which
\begin{equation*}
|x(t,i) - x^{\infty}| \leq \delta \qquad \text{for $\LcalI$-almost every $i \in I$}, 
\end{equation*}
whenever $t \geq T_{\delta}$. This is equivalent to stating that $\NormL{x(t) - x^{\infty}}{\infty}{I,\R^d}$ goes to 0 as $t \rightarrow +\infty$, which concludes the proof of our claim. 
\end{proof}

\begin{rmk}[On the persistence assumption \eqref{eq:PersistenceInfimum}]
It can be checked that the lower-bound \eqref{eq:PersistenceInfimum} on the average value of the in-degree function holds true e.g. when the algebraic connectivity is itself persistent -- as in Theorem \ref{thm:ConsensusBalanced} or Theorem \ref{thm:ConsensusStrong} --, or more generally if the interaction topology is algebraically persistent as in Theorem \ref{thm:ConsensusSym}. This is a again a consequence of the min-max theorem, as we recall that by \eqref{eq:DirectedConnectivity1}, it holds that
\begin{equation*}
\begin{aligned}
\lambda_2(\Lbb_v) & \leq \inf \Big\{\lambda \in \R ~\text{s.t.}~ \lambda \in \sigma_{\ess}(\Lbb_v^{\sym}) \Big\} \\
& = \inf_{i \in I} \INTDom{v(i) a(i,j)}{I}{j} \\
& \leq~ \NormL{v}{\infty}{I,\R^d} \inf_{i \in I} \INTDom{a(i,j)}{I}{j}, 
\end{aligned}
\end{equation*}
for $\Lcal^1$-almost every $t \geq 0$. Here, $v \in L^{\infty}(I,\R_+^*)$ is the canonical eigenvector given either by Theorem \ref{thm:StrongConGraphon} if the topology is a disjoint union of strongly connected components, or by $v \equiv 1$ if the topology is either balanced or symmetric. 
\end{rmk}

%%%%%%%%%%%%%%%%%%%%%%%%%%%%%%%%%%%%%%%%%%%%%%%%%%%%%%%%%%%%%%%%%%%%%%%%%%

\subsection{Numerical illustrations}
\label{subsection:Numerics}

In this last section, we provide illustrations of our results on a series of relevant examples. We start by showing positive exponential consensus results, first for a topology with persistent scrambling, then for a balanced topology with positive connectivity, and finally for a symmetric nonlinear topology that is algebraically persistent. We conclude by an example of balanced topology with null connectivity for which a non-exponential consensus only with respect to the $L^2$-norm seems to arise numerically. 

%%%%%%%%%%%%%%%%%%%%%%%%%%%%%%%%%%%%%%%%%%%%%%%%%%%%%%%%%%%%%%%%%%%%%%%%%%

\paragraph*{Exponential consensus under persistent scrambling.}

In the following paragraphs, we start by exemplifying the consensus result established in Theorem \ref{thm:ConsensusDiam} of Section \ref{section:DiamConsensus}. To this end, we fix a real parameters $T > 0$ and an integer $n \geq 1$, and for any time $t \geq 0$ we introduce the notation ${\texttt{t}} := t-\lfloor t/T \rfloor \in [0,T)$. We then define the interaction function $a \in L^{\infty}(\R_+ \times I \times I,[0,1])$ by 
\begin{equation}
\label{eq:LeaderModel}
a(t,i,j) := \left\{
\begin{aligned}
& 1 - 2 n \bigg\lvert j - \frac{2n \, \texttt{t} + T}{2n T} \bigg\lvert ~~ & \text{if}~~ j \in \Big[ \tfrac{\vt}{T} , \tfrac{\vt+1}{T} \Big] ~~\text{and}~~ \vt \in \Big[ 0 , \tfrac{n-1}{n} T \Big), \\
& 0 ~~ & \text{otherwise},
\end{aligned}
\right.
\end{equation}
for all times $t \geq 0$ and every $i,j \in I$. 

\begin{figure}[!ht]
\hspace{-0.6cm}
\begin{tikzpicture}
\draw (0,0) node {\includegraphics[scale=0.15]{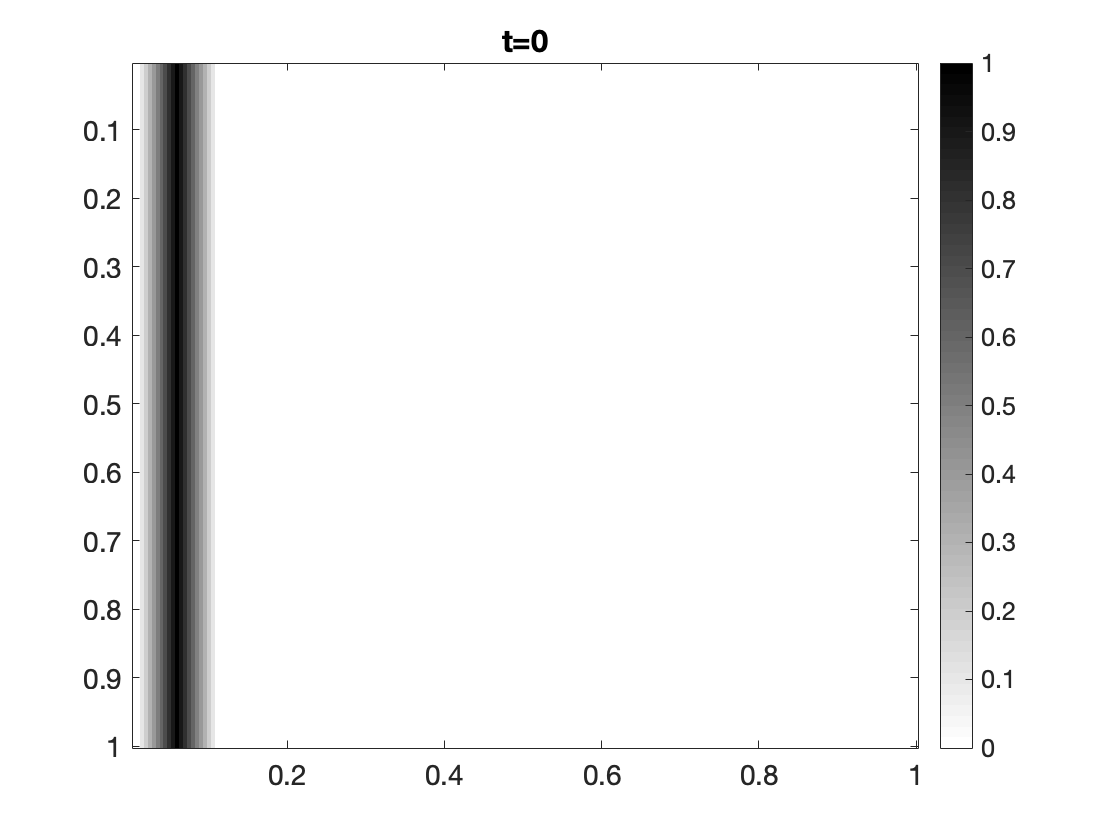}};
\draw (6,0) node {\includegraphics[scale=0.15]{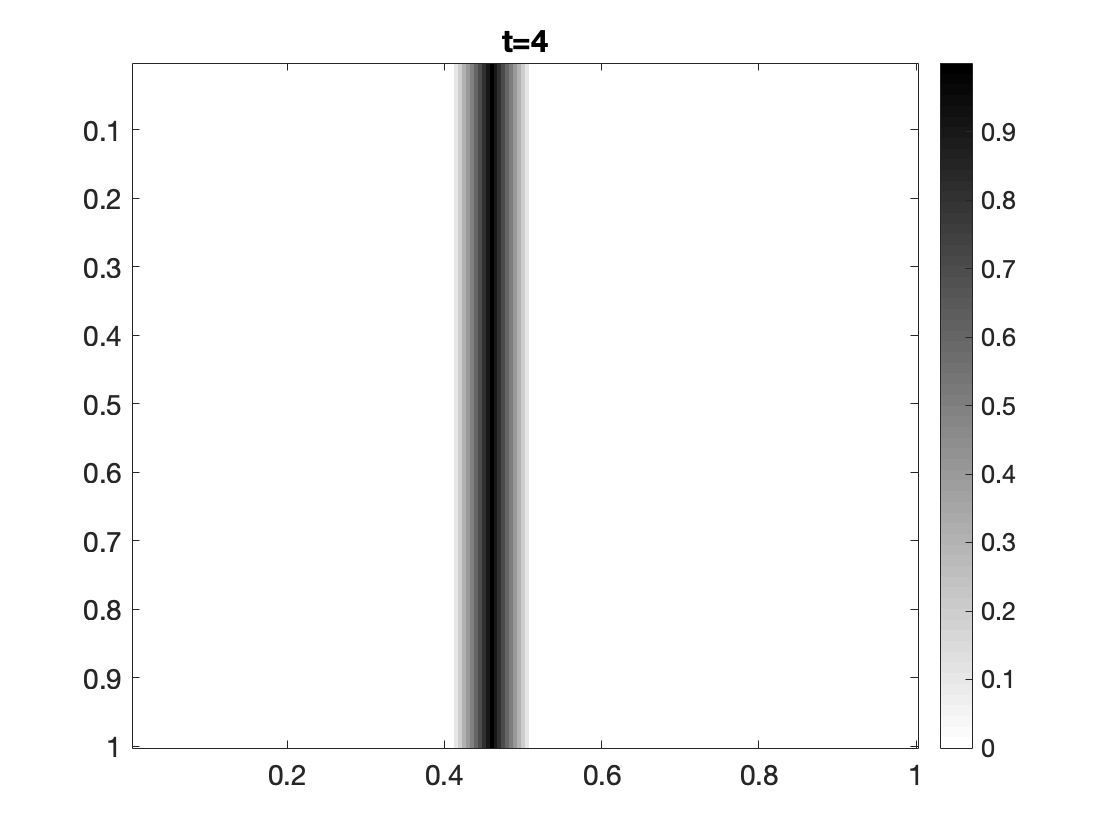}};
\draw (12,0) node {\includegraphics[scale=0.15]{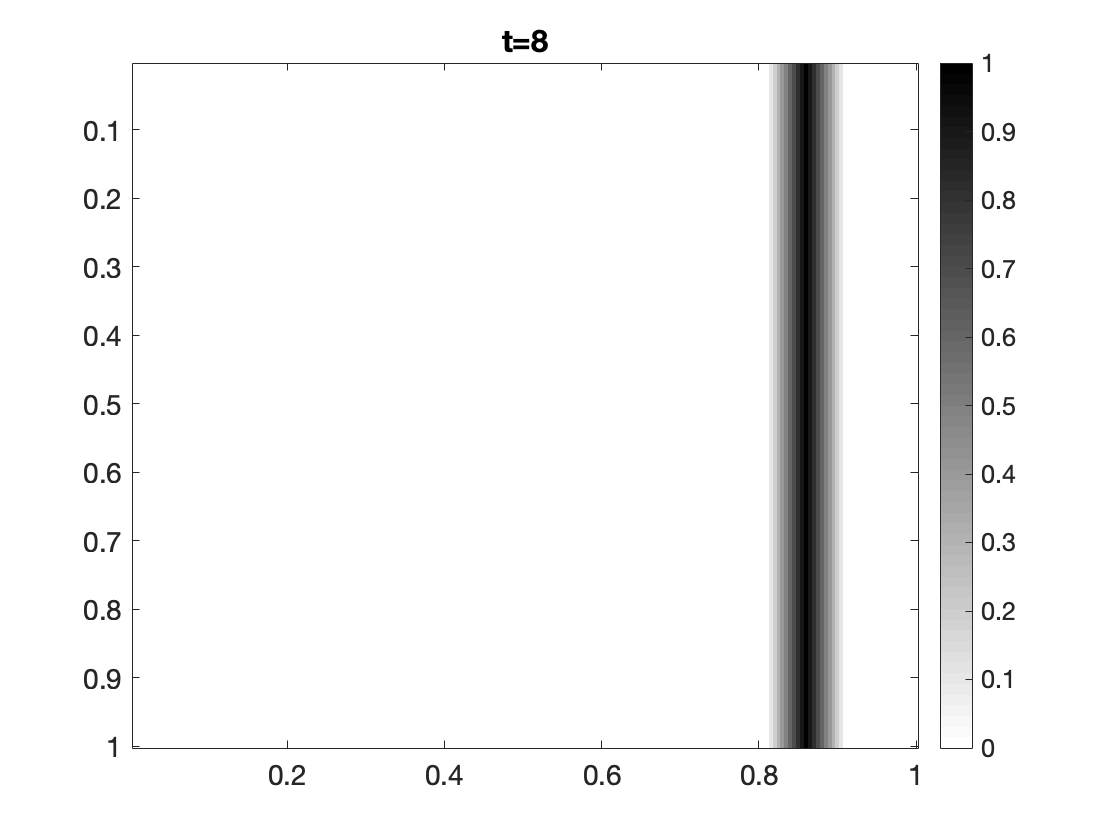}};
\draw (14,0) node {\textcolor{white}{a}}; 
\end{tikzpicture}
\caption{{\small \textit{Representation of the interaction function $(t,i,j) \in \R_+ \times I \times I \mapsto a(t,i,j) \in [0,1]$ defined in \eqref{eq:LeaderModel} with $T = n = 10$ at times $t = 0$ (left), $t = 4$ (center) and $t = 8$ (right).}}}
\label{fig:KernelLeader}
\end{figure}

A first observation is that the value of $a(t,i,j)$ defined in \eqref{eq:LeaderModel} is independent of $i \in I$, so that all the agents follow the same subset of elements at all times $t \geq 0$. As illustrated in Figure \ref{fig:KernelLeader}, this is a natural generalisation to the graphon setting of leader-follower dynamics (compare with the leftmost example in Figure \ref{fig:Scrambling}), in which the leaders are modelled as a time-varying subset of agents with positive measure. Moreover, notice that the signal $t \in \R_+ \mapsto a(t) \in L^{\infty}(I \times I,[0,1])$ is $T$-periodic, and that the graphon is completely disconnected whenever $\vt \in \big( \tfrac{n-1}{n} T , T \big)$. In this context, we can explicitly compute the scrambling coefficients, since
\begin{equation*}
\eta(\Apazo(t)) = \inf_{i,j \in I} \INTDom{\min \big\{ a(t,i,k) , a(t,j,k) \big\}}{I}{k} = \inf_{i \in I} \INTDom{a(t,i,k)}{I}{k} = \frac{T}{2n}, 
\end{equation*}
whenever $t \geq 0$ is such that $\vt \in \big[0,\tfrac{n-1}{n}T \big)$. On the other hand, it is clear that $\eta(\Apazo(t)) = 0$ if $\vt \in \big[ \tfrac{n-1}{n}T , T \big)$. Thus for every parameter $\tau > T/n$, one can check that for all times $t \geq 0$, it holds
\begin{equation*}
\frac{1}{\tau} \INTSeg{\eta(\Apazo(s))}{s}{t}{t+\tau} \, \geq \, \frac{1}{\tau} \INTSeg{\eta(\Apazo(s))}{s}{(n-1)T/n}{T + (\tau - T/n)} \, \geq \, \frac{T(n \tau - T)}{2 \tau n^2}, 
\end{equation*}
so that the persistence condition \eqref{eq:PersistenceScramb} of Theorem \ref{thm:ConsensusDiam} is satisfied. Therefore, we expect that the system will converge exponentially to consensus with respect to the $L^{\infty}$-norm, a fact that is illustrated by the numerical simulations displayed in Figure \ref{fig:ConsensusScramb} below. Therein, we fix the parameters $T = n = 10$ and start from the initial configuration $x^0 : i \in I \mapsto \sin^2(4i) \in [0,1]$.

\begin{figure}[!ht]
\hspace{-0.6cm}
\begin{tikzpicture}
\draw (0,0) node {\includegraphics[scale=0.3]{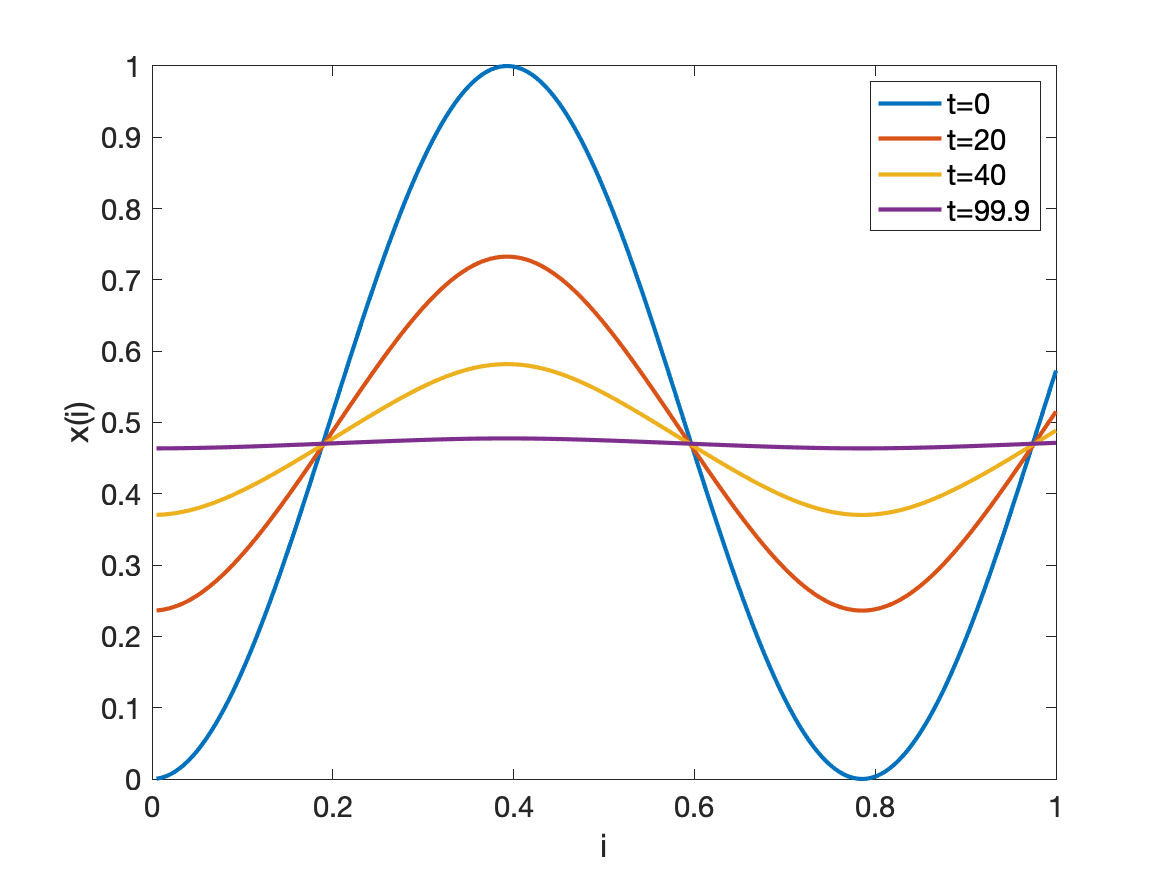}};
\draw (6,0) node {\includegraphics[scale=0.3]{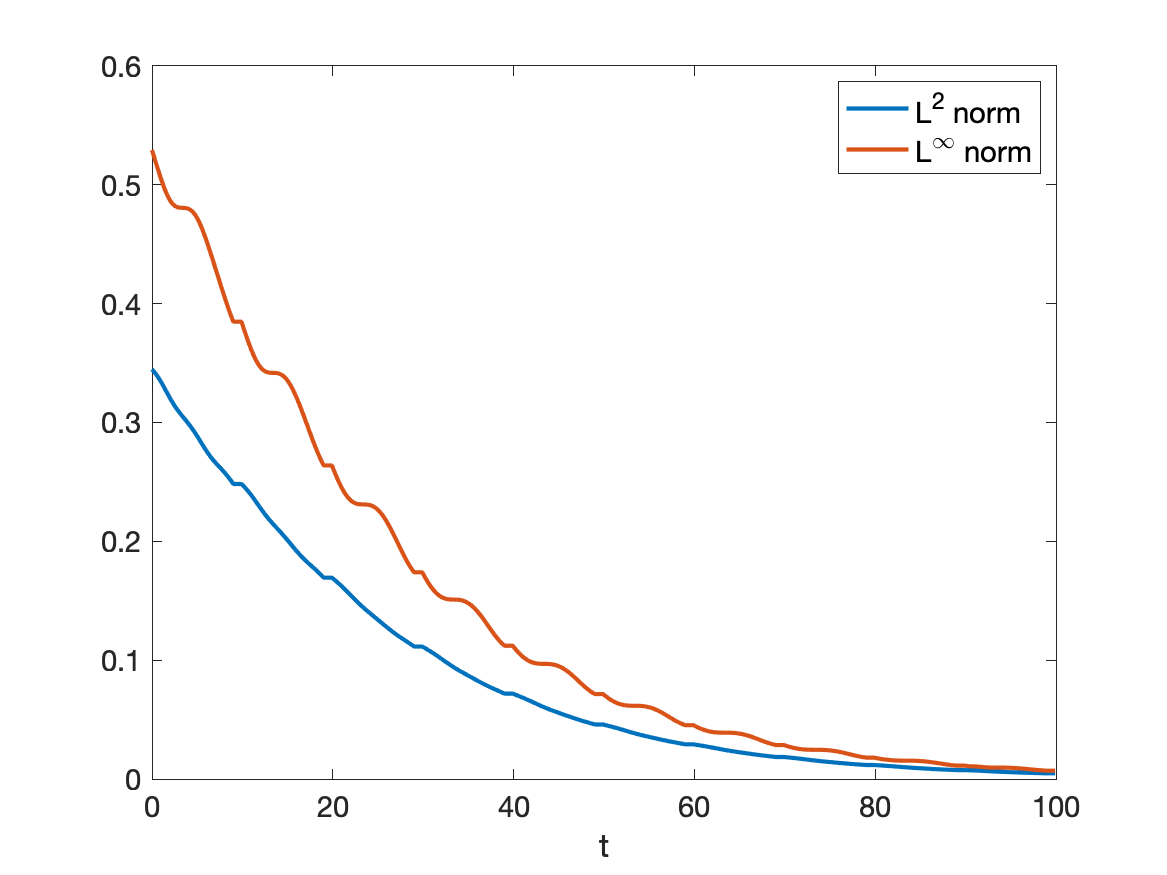}};
\draw (12,0) node {\includegraphics[scale=0.3]{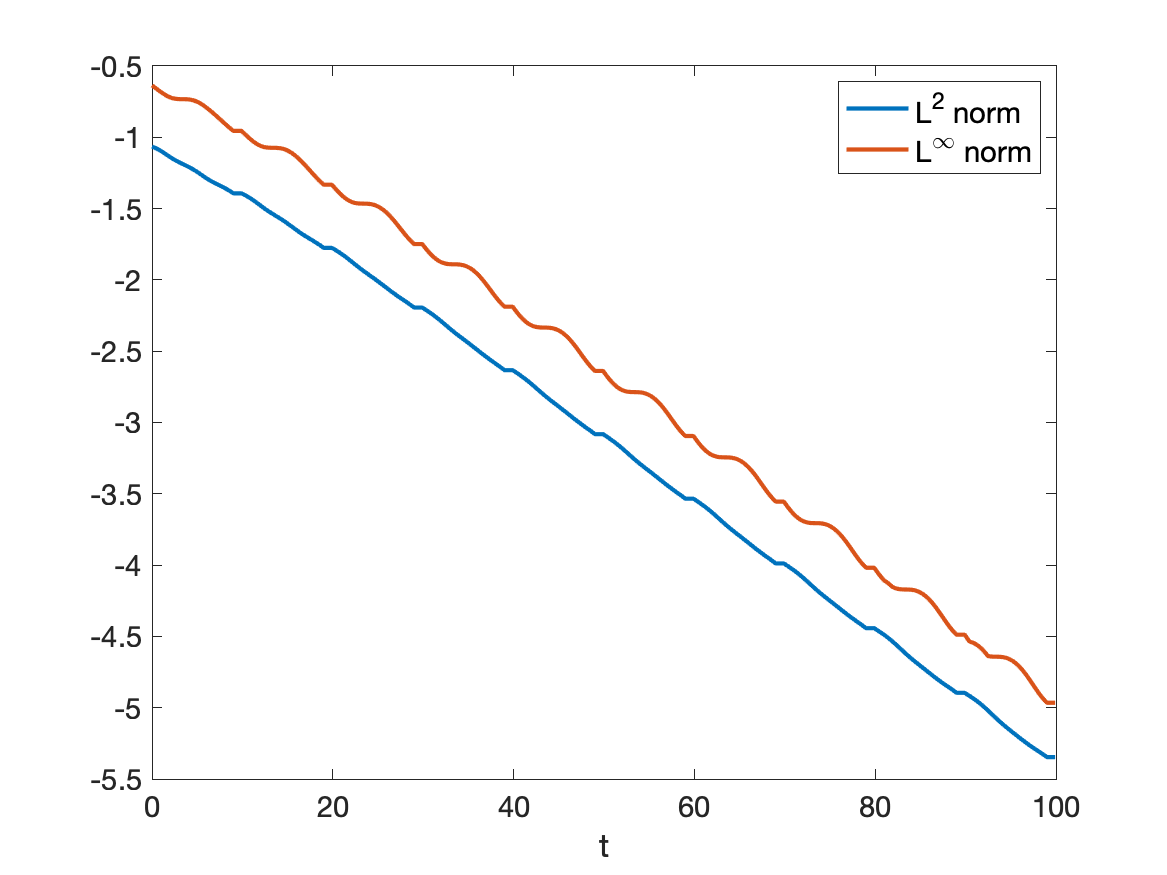}};
\end{tikzpicture}
\caption{{\small \textit{Snapshots of the solution $i \in I \mapsto x(t,i) \in [0,1]$ generated by the communication weights \eqref{eq:LeaderModel} at different instants (left) along with the time-evolution of the $L^2$- and $L^{\infty}$-distance to the consensus point in natural scale (center) and log scale (right).}}}
\label{fig:ConsensusScramb}
\end{figure}

%%%%%%%%%%%%%%%%%%%%%%%%%%%%%%%%%%%%%%%%%%%%%%%%%%%%%%%%%%%%%%%%%%%%%%%%%%%%

\paragraph*{Exponential consensus for a balanced topology with positive connectivity}

This time, we consider a stationary balanced interaction topology which is defined as follows. Let $\xi \in C^0([0,1])$ be the function given by 
\begin{equation*}
\xi(s) := (1- 4s) \mathds{1}_{\big[ 0,\tfrac{1}{4} \big]}(s) \qquad \text{for each $s \in [0,1]$},  
\end{equation*}
and, by a slight abuse of notation, denote again by $\xi \in C^0(\R,[0,1])$ its $1$-periodisation over the real line. Then, we consider the constant-in-time interaction kernel $a \in L^{\infty}(I \times I,[0,1])$ defined by 
\begin{equation}
\label{eq:BalancedModel1}
a(i,j) := \xi(i-j), 
\end{equation}
for every $i,j \in I$, which is illustrated in Figure \ref{fig:KernelBalanced}. This particular choice of interaction function encodes the notion of directed cycle at the level of graphon models (draw a comparison with the leftmost topology in Figure \ref{fig:Balanced}), where the neighbours are represented by small clusters of agents with positive measure. 

\begin{figure}[!ht]
\centering
\begin{tikzpicture}
\draw (0,0) node {\includegraphics[scale=0.4]{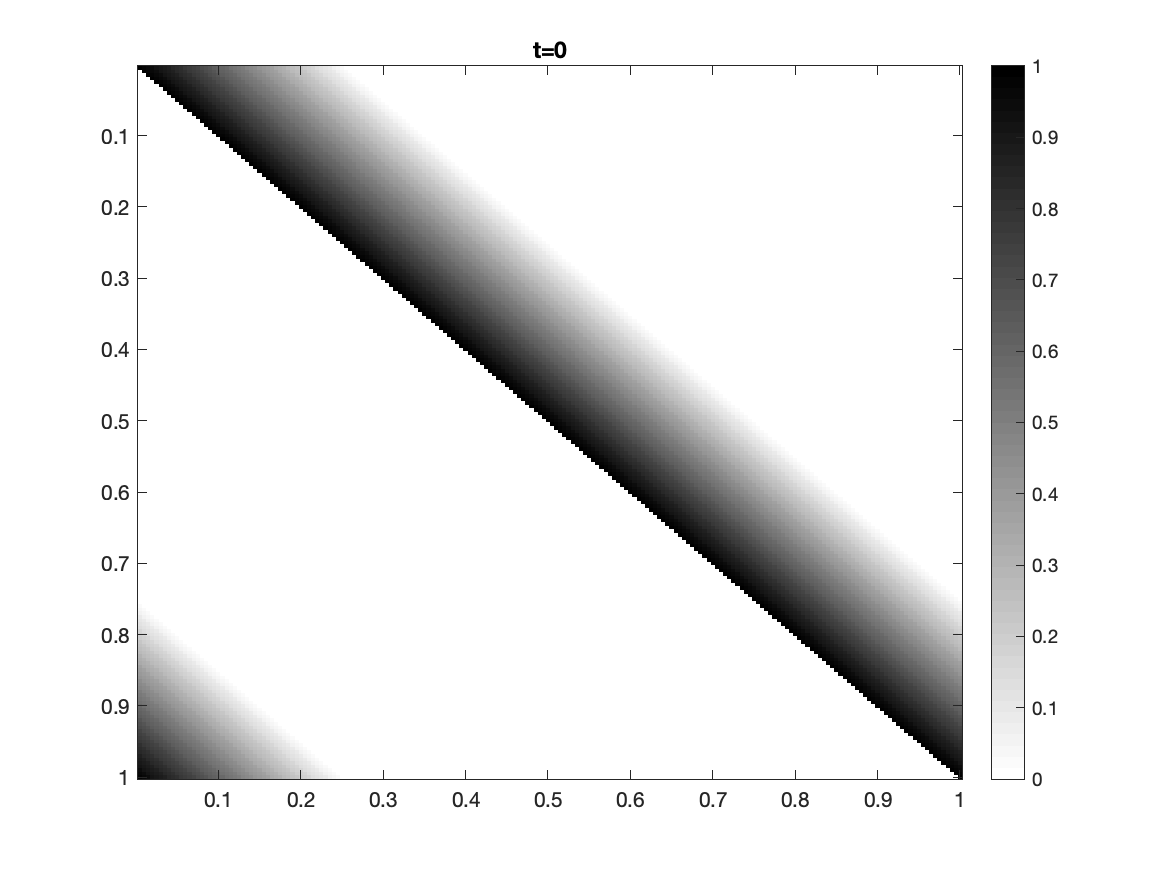}};
\end{tikzpicture}
\caption{{\small Representation of the interaction function $(i,j) \in I \times I \mapsto a(i,j) \in [0,1]$ defined in \eqref{eq:BalancedModel1}.}}
\label{fig:KernelBalanced}
\end{figure}

It can be readily verified that the interaction kernel $a \in L^{\infty}(I \times I,[0,1])$ defined as in \eqref{eq:BalancedModel1} generates a balanced topology, since
\begin{equation}
\label{eq:BalancedTopoCheck}
\begin{aligned}
\INTDom{a(i,j)}{I}{j} = \INTDom{\xi(i-j)}{I}{j} &  = \INTSeg{\xi(i-j)}{j}{0}{i} + \INTSeg{\xi(i-j+1)}{j}{i}{1} \\
& = \INTSeg{\xi(s)}{s}{0}{i} + \INTSeg{\xi(s')}{s'}{i}{1} = \INTSeg{\xi(s)}{s}{0}{1} = \INTDom{a(j,i)}{I}{j}, 
\end{aligned}
\end{equation}
where we used the fact that $\xi(\cdot)$ is $1$-periodic and applied the the changes of coordinates $s = i-j$ and $s' = 1+i-j$ in the first and second integral respectively. We now prove that the topology is strongly connected in the sense of Definition \ref{def:StrongCon}. It can be verified straightforwardly by using \eqref{eq:BalancedTopoCheck} that 
\begin{equation*}
\inf_{i \in I} \INTDom{a(i,j)}{I}{j} = \INTSeg{\xi(s)}{s}{0}{1} = \INTSeg{(1-4s)}{s}{0}{1/4} = \frac{1}{8}, 
\end{equation*}
so that the condition of Definition \ref{def:StrongCon}-$(b)$ holds. Moreover, for every pair $i,j \in I$ with $i < j$, define the finite sequence of indices $(l_{k})_{1 \leq k \leq 8}$ by 
\begin{equation*}
l_1 = i \qquad \text{and} \qquad l_{k+1} = \min \big\{ l_k + \tfrac{1}{8} , j \big\} ~~ \text{for $k \in \{1,\dots,7\}$}, 
\end{equation*}
and observe that $j = l_{K_j}$ for some $K_j \in \{1,\dots,7\}$ with $a(l_k,l_{k+1}) \geq \tfrac{1}{2}$ for each $k \in \{1,\dots,K_j-1\}$. Therefore, the conditions of Definition \ref{def:StrongCon}-$(a)$ are satisfied, and the interaction topology is strongly connected. Hence by Theorem \ref{thm:DirectedConnectivity}, it holds that $\lambda_2(\Lbb) > 0$ and the system should converge exponentially towards consensus in the $L^2$-norm. In addition, because the sufficient condition of Theorem \ref{thm:Equivalence} is satisfied, the convergence should also occur in the $L^{\infty}$-norm topology. The simulations made with the initial datum $x^0 : i \in I \mapsto \sin^2(4i) \in [0,1]$ and displayed in Figure \ref{fig:ConsensusBalanced} confirm these assertions. In addition, let it be noted that the convergence with respect to the $L^{\infty}$-norm also seems to be exponential, which could suggest that a stronger variant of Theorem \ref{thm:Equivalence} stating the equivalence between exponential convergences towards consensus may in fact be true.

\begin{figure}[!ht]
\hspace{-0.6cm}
\begin{tikzpicture}
\draw (0,0) node {\includegraphics[scale=0.15]{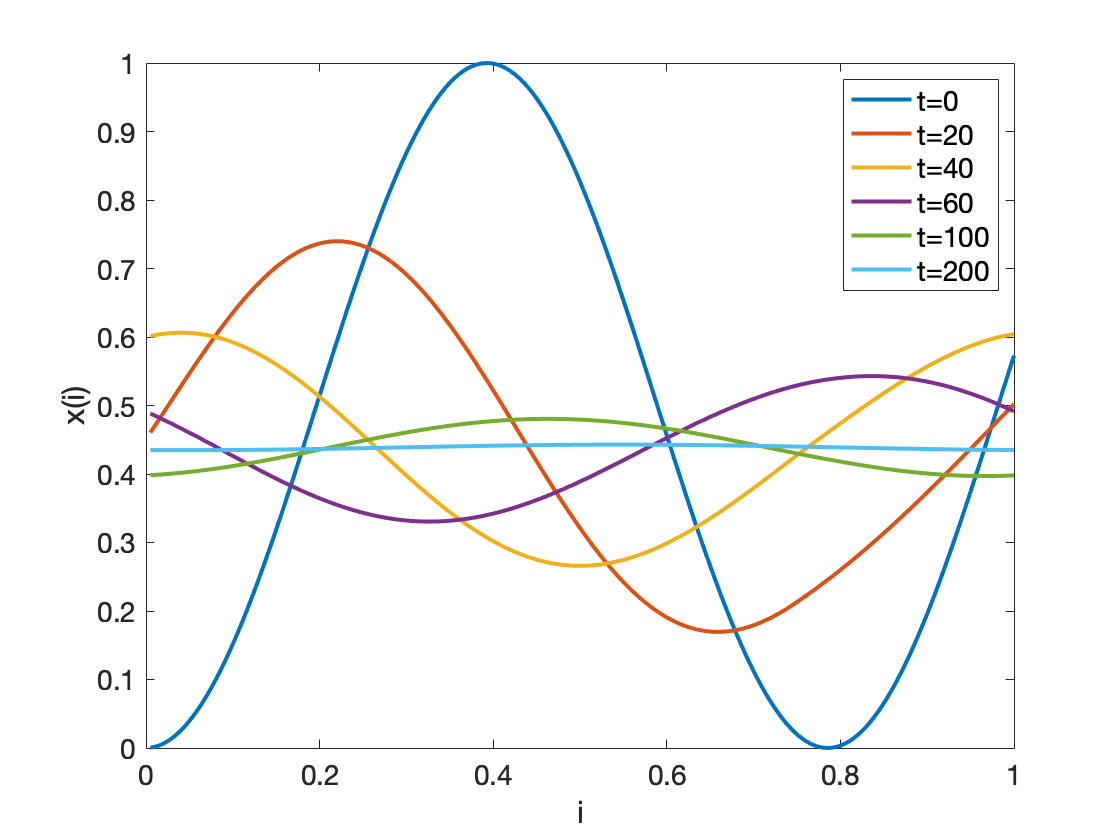}};
\draw (6,0) node {\includegraphics[scale=0.3]{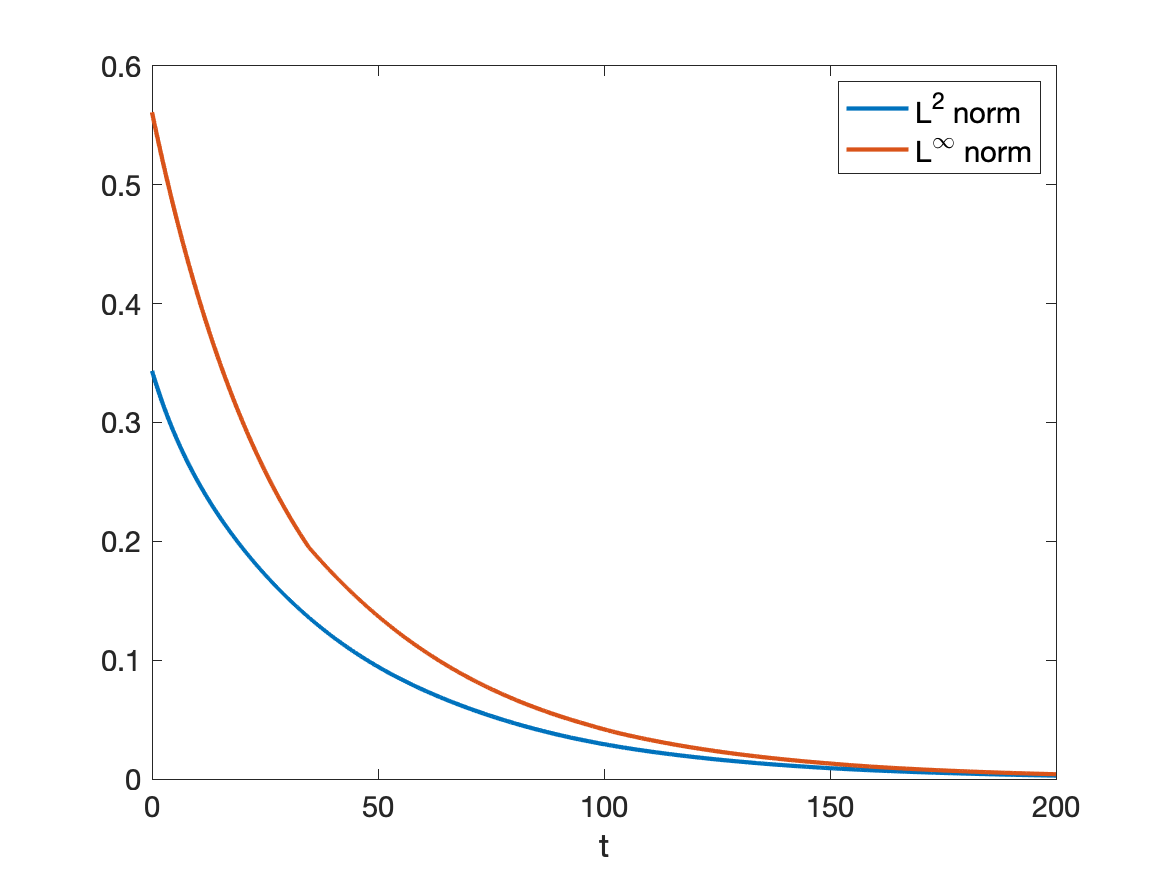}};
\draw (12,0) node {\includegraphics[scale=0.3]{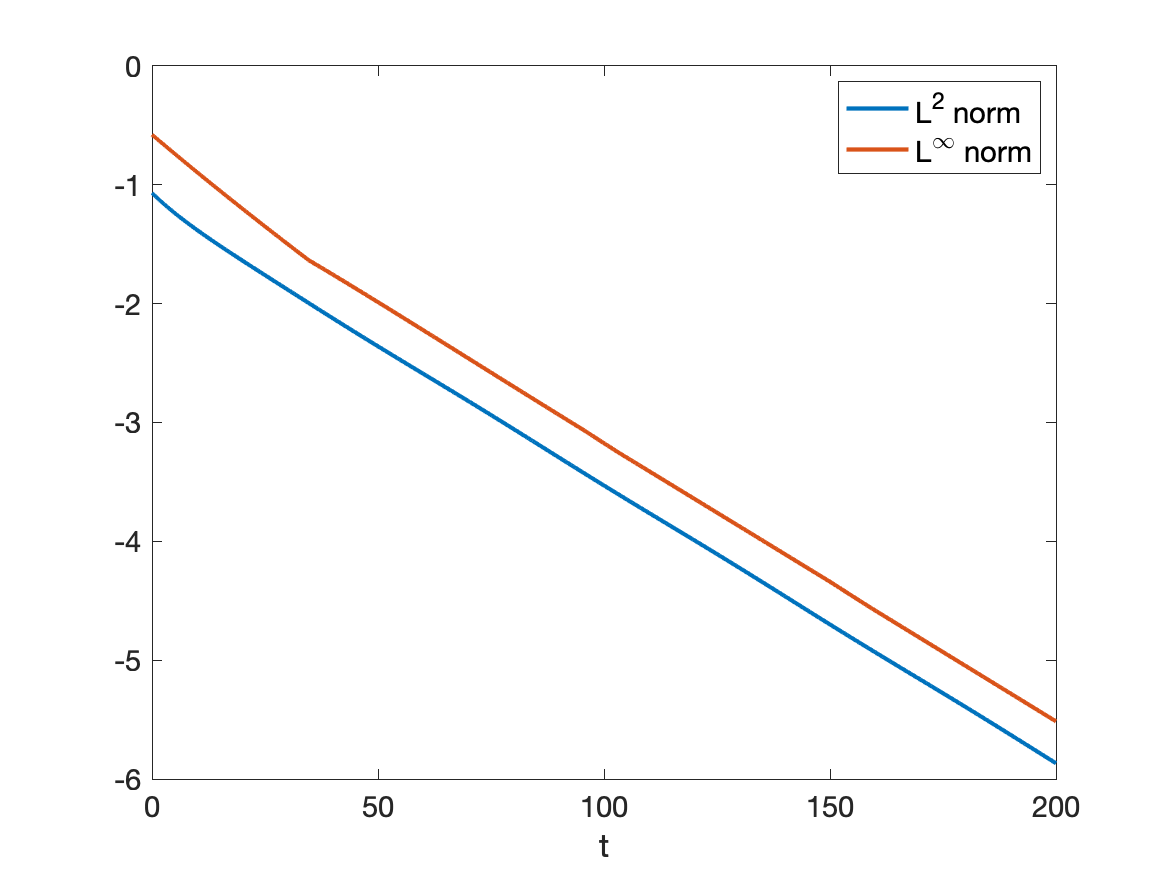}}; 
\end{tikzpicture}
\caption{{\small \textit{Snapshots of the solution $i \in I \mapsto x(t,i) \in [0,1]$ generated by the communication weights \eqref{eq:BalancedModel1} at different instants (left) along with the time-evolution of the $L^2$- and $L^{\infty}$-distance to the consensus point in natural scale (center) and log scale (right).}}}
\label{fig:ConsensusBalanced}
\end{figure}

%%%%%%%%%%%%%%%%%%%%%%%%%%%%%%%%%%%%%%%%%%%%%%%%%%%%%%%%%%%%%%%%%%%%%%%%%%%%%%%%%%%%%%%%%%

\paragraph*{Exponential consensus for an algebraically persistent symmetric nonlinear topology.}

Next, we turn our attention to a nonlinear symmetric and time-dependent interaction kernel which defines disconnected topologies at all times $t \geq 0$, but which is algebraically persistent in the sense defined in Theorem \ref{thm:ConsensusSym}. Given a real parameter $T > 0$ and an integer $n \geq 1$, we introduce the notation $\vt := t - \lfloor t/T \rfloor \in [0,T)$. We then define the interaction function $a \in L^{\infty}(\R_+ \times I \times I,[0,1])$ by

\begin{equation}
\label{eq:SymmetricModel}
a(t,i,j) := 
\left\{
\begin{aligned}
& 1 ~~ & \text{if}~~ i,j \in \big[ \tfrac{\vt}{T} , \tfrac{\vt}{T} + \tfrac{1}{n} \big] ~~\text{and}~~ \vt \in \big[ 0, \big( \tfrac{n-1}{n} \big) T \big), \, \\
& 1 ~~ & \text{if}~~ i,j \in \big[ \tfrac{\vt}{T} ,1 \big] \cup \big[ 0,\tfrac{\vt}{T} + \tfrac{1}{n} - 1 \big] ~~\text{and}~~ \vt \in \big[ \big( \tfrac{n-1}{n} \big) T , T \big),  \\
& 0 ~~ & \text{otherwise,} 
\end{aligned}
\right.
\end{equation}
for all times $t \geq 0$ and every $i,j \in I$. The latter is symmetric for all times $t \geq 0$ by construction, and is illustrated in Figure \ref{fig:KernelSymmetric}. We also consider the nonlinear Cucker-Smale type kernel given by 
\begin{equation*}
\phi(r) := \frac{1}{(1+r)^2}, 
\end{equation*}
for each $r \geq 0$, and consider the solution $x(\cdot)$ of \eqref{eq:GraphonLapDyn} with initial condition $x^0 : i \in I \mapsto \sin^2(4i) \in [0,1]$. 

\begin{figure}[!ht]
\hspace{-0.6cm}
\begin{tikzpicture}
\draw (0,0) node {\includegraphics[scale=0.15]{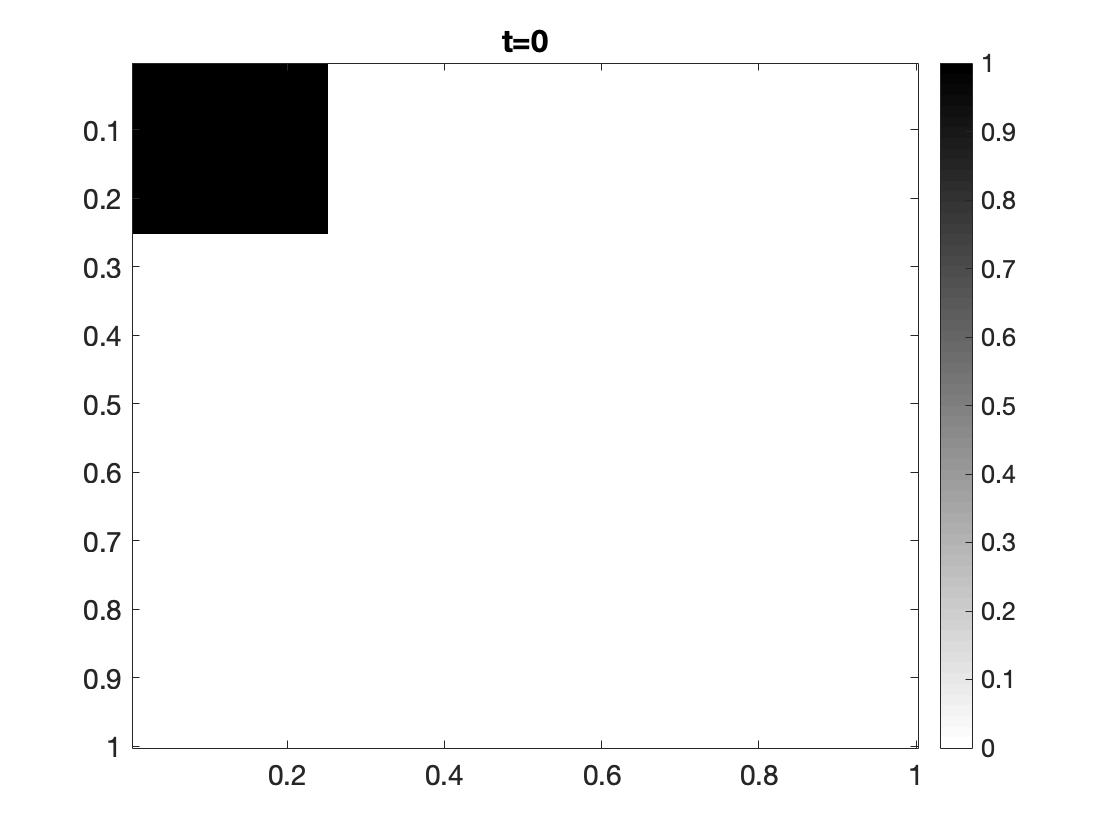}};
\draw (6,0) node {\includegraphics[scale=0.3]{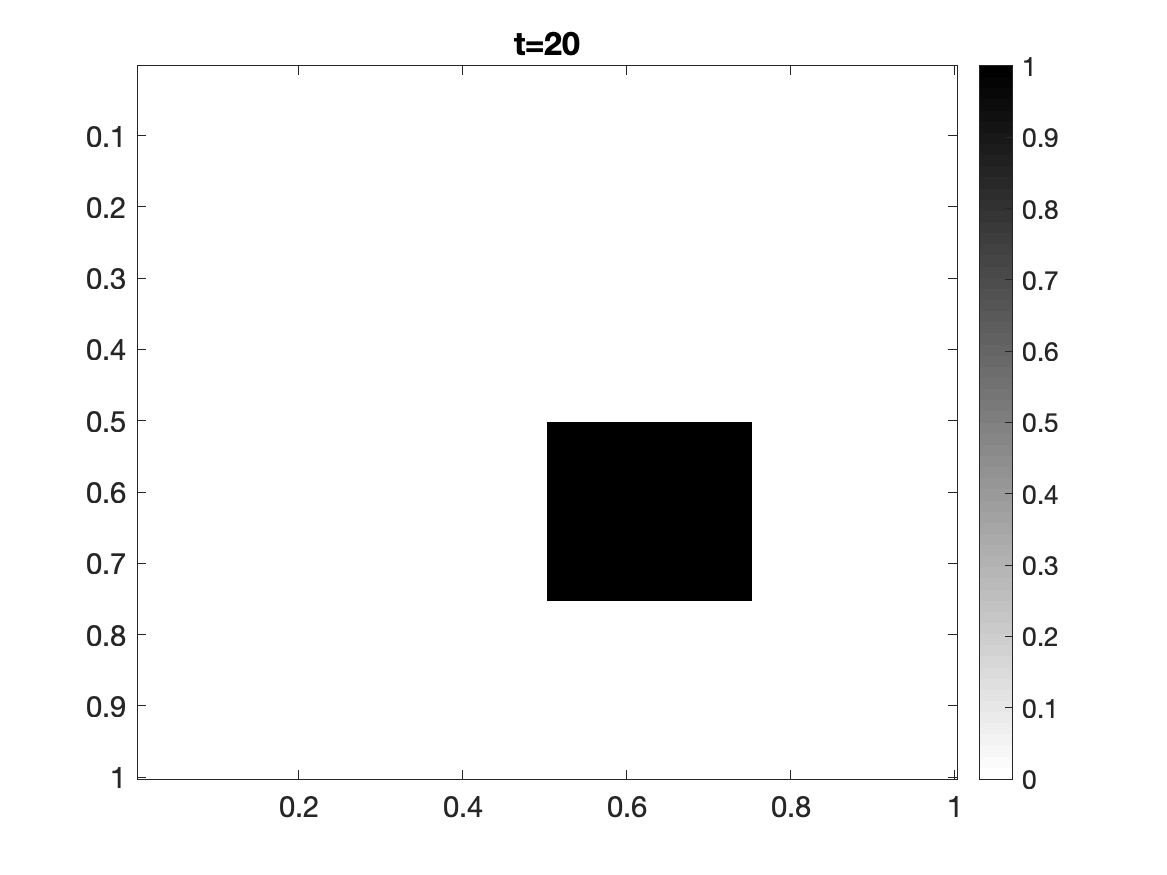}};
\draw (12,0) node {\includegraphics[scale=0.3]{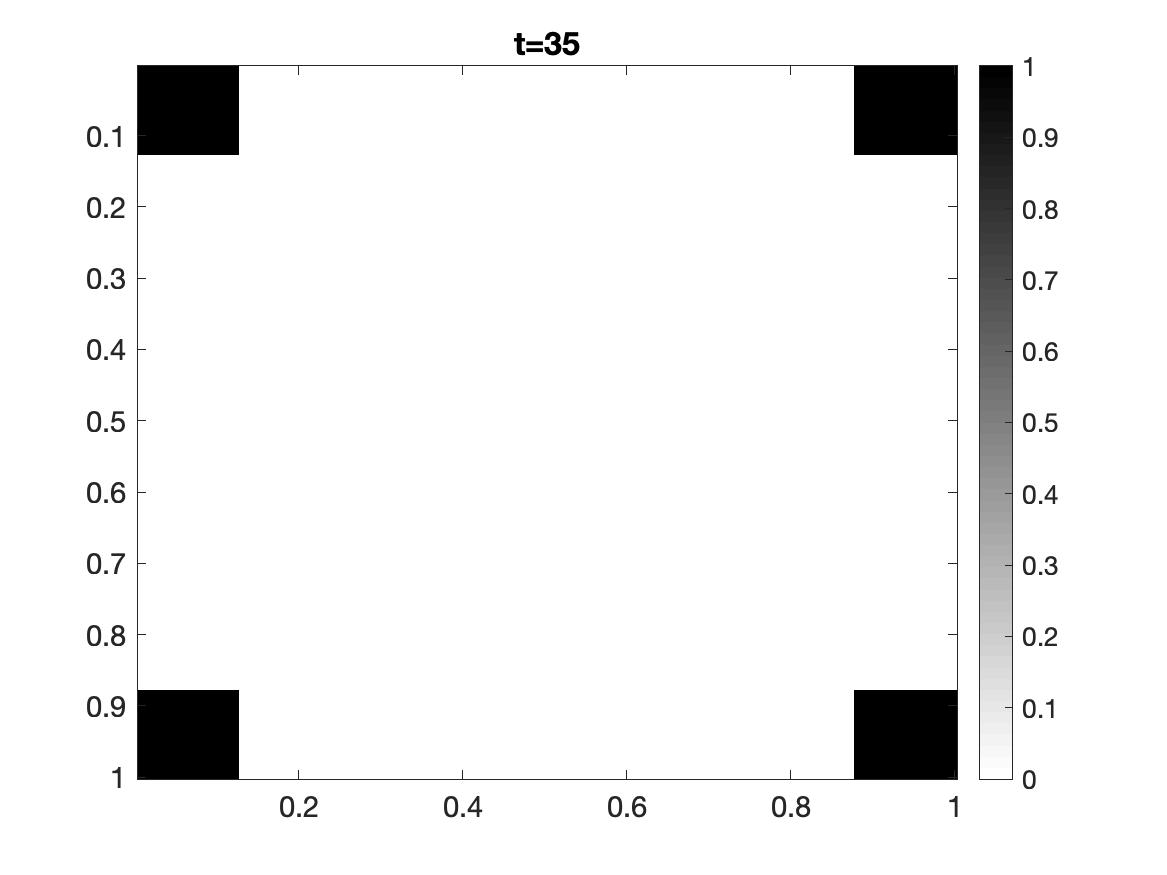}};
\end{tikzpicture}
\caption{{\small \textit{Representation of the interaction function $(t,i,j) \in \R_+ \times I \times I \mapsto a(t,i,j) \in [0,1]$ defined in \eqref{eq:SymmetricModel} with $T = 40$ and $n = 4$ at times $t = 0$ (left), $t = 20$ (center) and $t = 35$ (right).}}}
\label{fig:KernelSymmetric}
\end{figure}

\begin{figure}[!ht]
\centering
\begin{tikzpicture}
\draw (0,0) node {\includegraphics[scale=0.2]{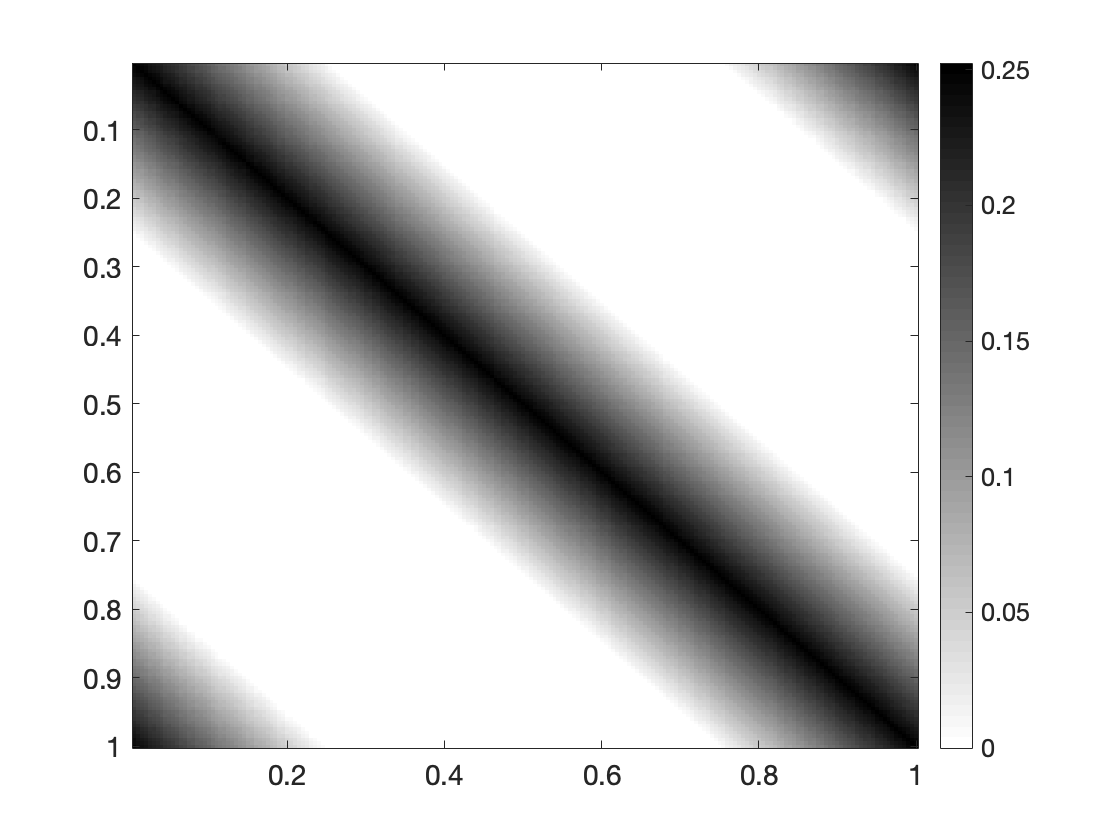}};
\end{tikzpicture}
\caption{{\small Representation of the average over a time-window of length $T = 10$ of the interaction function $(t,i,j) \in I \times I \mapsto a(t,i,j) \in [0,1]$ defined in \eqref{eq:SymmetricModel}.}}
\label{fig:KernelAverageSymmetric}
\end{figure}

It is quite clear from its definition in \eqref{eq:SymmetricModel} that in our example, the interaction topology is disconnected at each instant. However, the average of the interaction function over any time window of length $T$ defines the counterpart for graphon models of an undirected cycle, which is illustrated in Figure \ref{fig:KernelAverageSymmetric} (to be compared with the central picture in Figure \ref{fig:Scrambling}). Indeed, for all $i,j \in I$ satisfying 
\begin{equation*}
j \in \Big[ \max \big\{ 0,i - \tfrac{1}{n} \big\}, \min\big\{ 1, i + \tfrac{1}{n} \big\} \Big],
\end{equation*}
it can be verified that 
\begin{equation*}
\frac{1}{T} \INTSeg{\INTDom{a(s,i,j)}{I}{j}}{s}{t}{t+T} = \frac{1}{T} \INTSeg{\INTDom{a(s,i,j)}{I}{j}}{s}{0}{T} \geq \frac{1}{n^2}, 
\end{equation*}
for all times $t \geq 0$, where we used the fact that signal $t \in \R_+ \mapsto a(t) \in L^{\infty}(I \times I,[0,1])$ is $T$-periodic. Therefore by Theorem \ref{thm:UndirectedConnectivity}, the algebraic connectivity of the averaged interaction topology is positive, and the hypotheses of Theorem \ref{thm:ConsensusSym} are satisfied. Thus, as confirmed by the results displayed in Figure \ref{fig:ConsensusSymmetric}, we expect that the system exponentially converges to consensus in the $L^2$-norm topology. In addition, it can be readily verified that the in-degree function satisfies the persistence condition of Theorem \ref{thm:Equivalence}, so that convergence towards consensus also holds with respect to the $L^{\infty}$-norm. 

\begin{figure}[!ht]
\hspace{-0.6cm}
\begin{tikzpicture}
\draw (0,0) node {\includegraphics[scale=0.3]{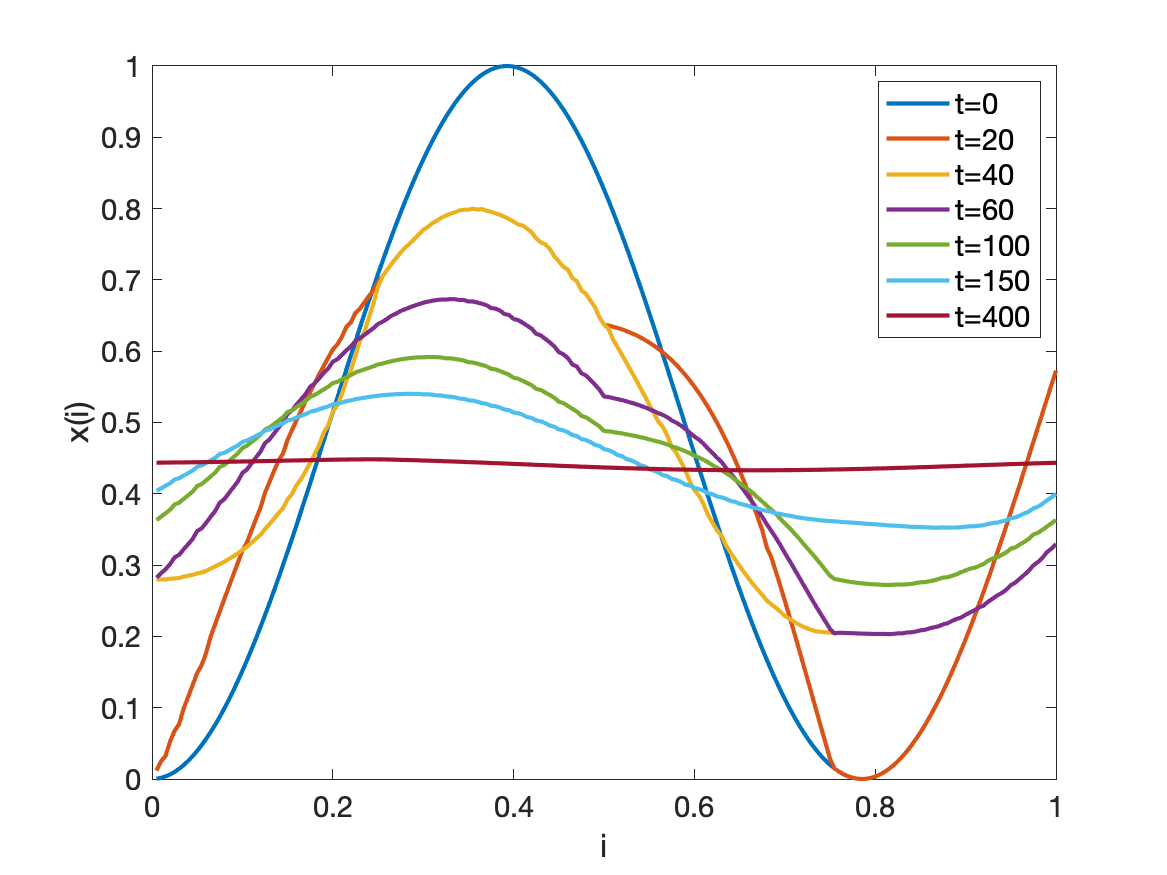}};
\draw (6,0) node {\includegraphics[scale=0.3]{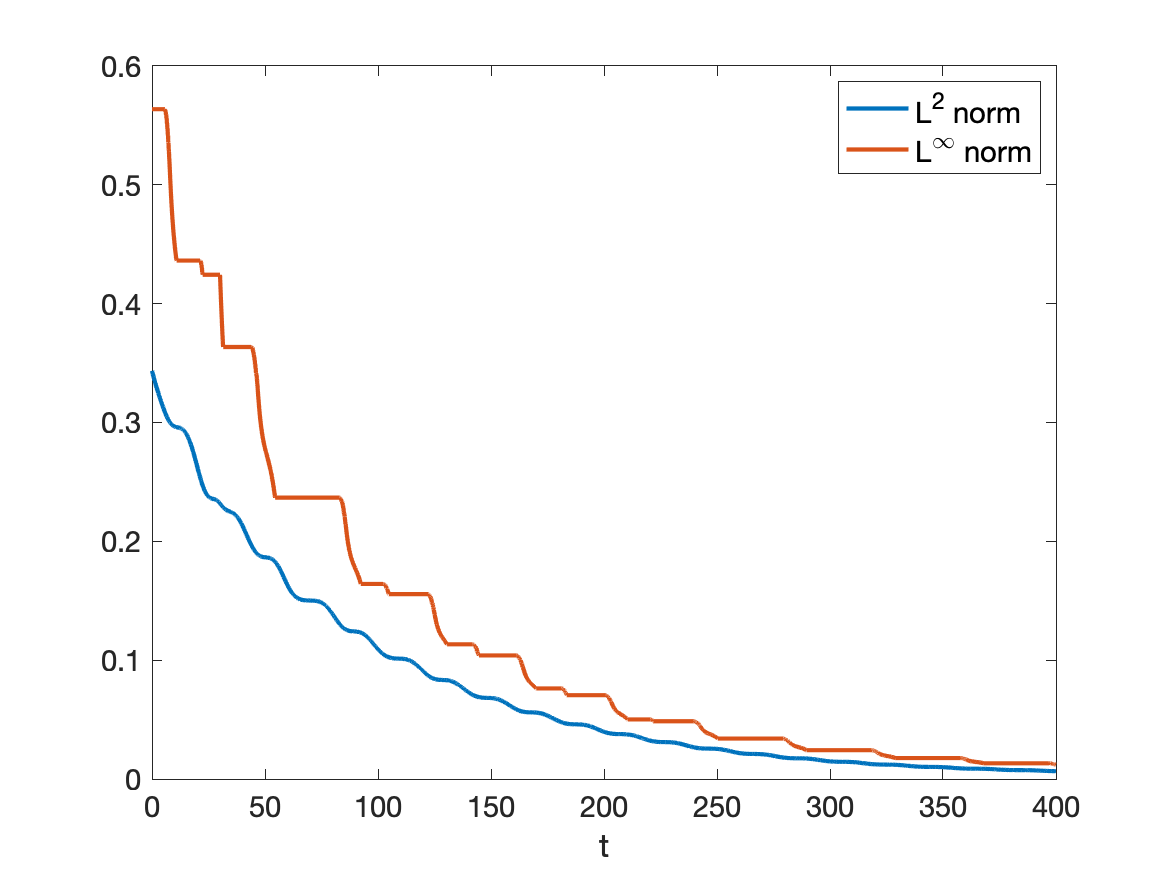}};
\draw (12,0) node {\includegraphics[scale=0.3]{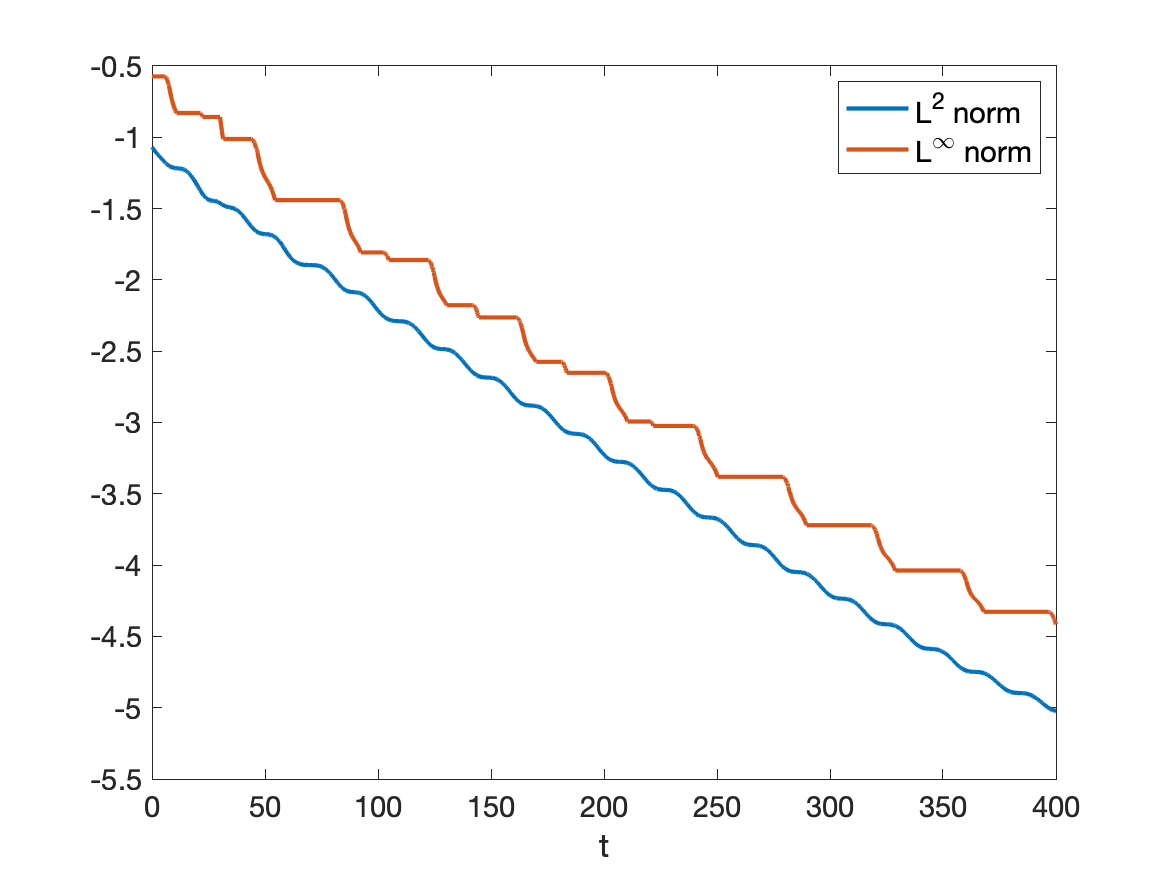}}; 
\end{tikzpicture}
\caption{{\small \textit{Snapshots of the solution $i \in I \mapsto x(t,i) \in [0,1]$ generated by the communication weights \eqref{eq:SymmetricModel} at different instants (left) along with the time-evolution of the $L^2$- and $L^{\infty}$-distance to the consensus point in natural scale (center) and log scale (right).}}}
\label{fig:ConsensusSymmetric}
\end{figure}

%%%%%%%%%%%%%%%%%%%%%%%%%%%%%%%%%%%%%%%%%%%%%%%%%%%%%%%%%%%%%%%%%%%%%%%%%%%%%%%%%%%%%%%%%%

\paragraph*{Non-exponential consensus for a symmetric topology with null connectivity.} In this last paragraph, we provide an interesting limit example of the theory developed in this article, which highlights several of our results. To this end, we consider the stationary interaction kernel defined by 
\begin{equation}
\label{eq:NonConsensusModel}
a(i,j) :=
\left\{
\begin{aligned}
& 1 ~~ & \text{if either $i \in \big[ 0,\tfrac{1}{2} \big]$ and $\tfrac{i}{2} \leq j \leq 2i$ ~or~ $i,j \in \big[ \tfrac{1}{2},1]$},\\
& 0 ~~ & \text{otherwise,}
\end{aligned}
\right.
\end{equation}
for every $i,j \in I$, and which is illustrated in Figure \ref{fig:KernelNonConsensus}. 

\begin{figure}[!ht]
\centering
\begin{tikzpicture}
\draw (0,0) node {\includegraphics[scale=0.4]{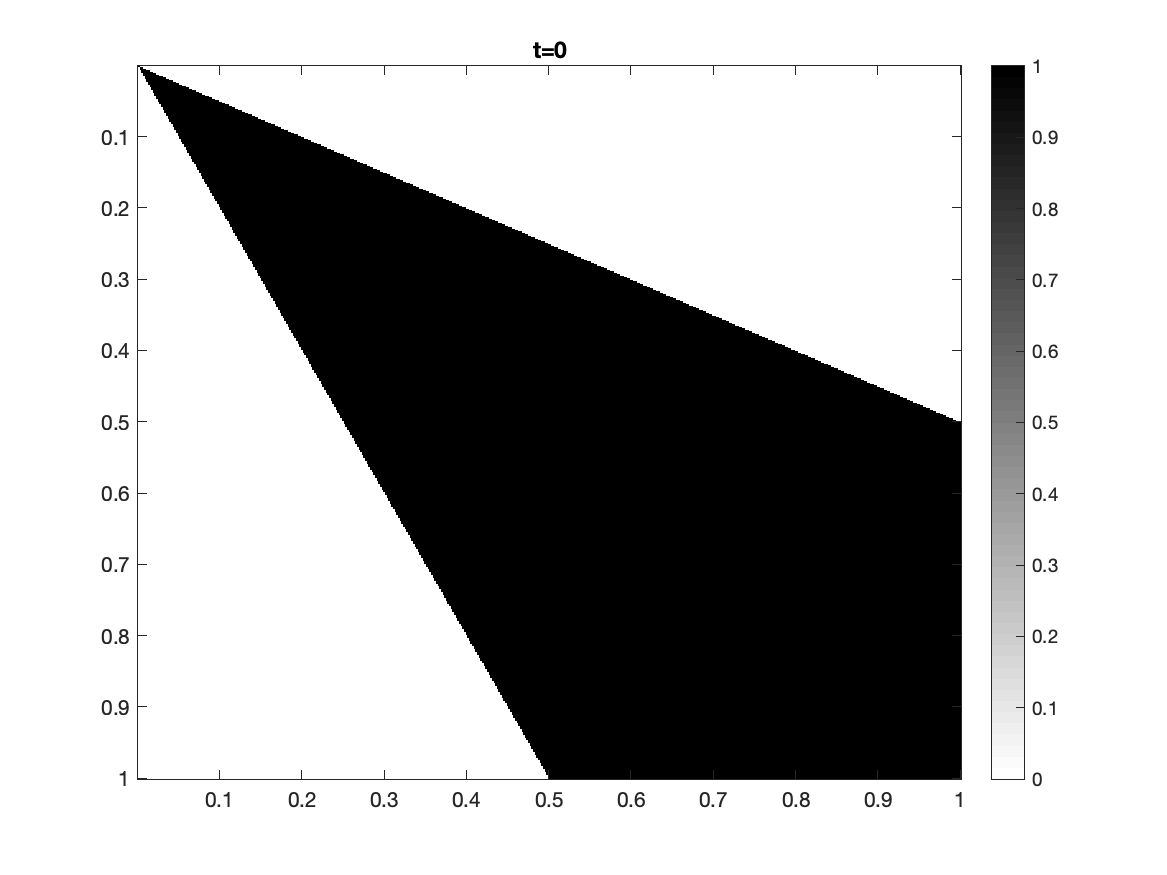}};
\end{tikzpicture}
\caption{{\small Representation of the interaction function $(i,j) \in I \times I \mapsto a(i,j) \in [0,1]$ defined in \eqref{eq:NonConsensusModel}.}}
\label{fig:KernelNonConsensus}
\end{figure}

It is clear from its definition that $a \in L^{\infty}(I \times I,[0,1])$ defines a symmetric topology, but one can readily check  that $\inf_{i \in I} \INTDom{a(i,j)}{I}{j} = 0$, which implies that the interaction graphon is not strongly connected in the sense of Definition \ref{def:StrongCon}. By Theorem \ref{thm:UndirectedConnectivity}, this implies in particular that the corresponding algebraic connectivity satisfies $\lambda_2(\Lbb) = 0$, which means that even though the convergence to consensus in the $L^2$-norm may occur, it may not be exponential. Furthermore, the sufficient condition of Theorem \ref{thm:Equivalence} ensuring that $L^2$- and $L^{\infty}$-consensus formation are equivalent is also violated, which suggests that solutions of the graphon dynamics may not converge to consensus in the $L^{\infty}$-norm topology. This intuition is supported by the fact that, in this precise example, the agent with label $i = 0$ is not connected to any other agent in the system, and should not move. 

\begin{figure}[!ht]
\hspace{-0.6cm}
\begin{tikzpicture}
\draw (0,0) node {\includegraphics[scale=0.3]{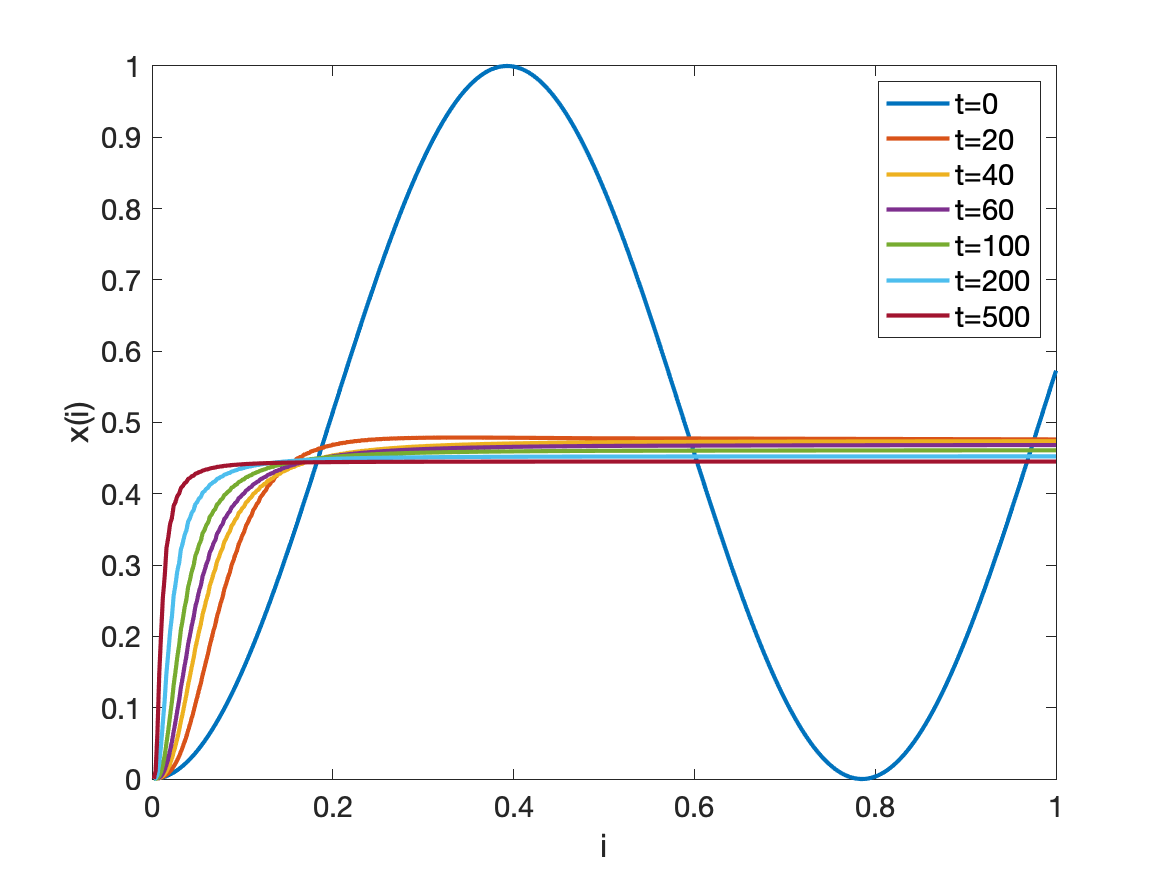}};
\draw (6,0) node {\includegraphics[scale=0.3]{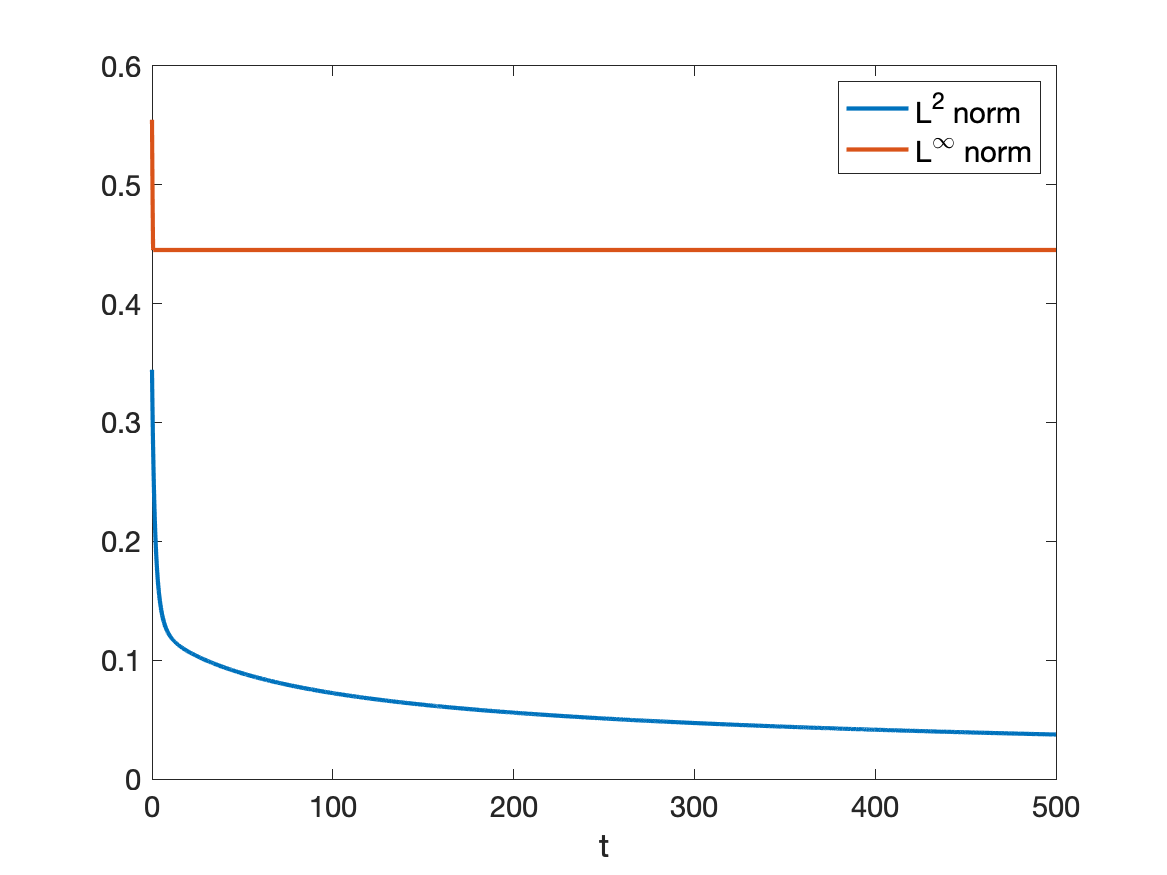}};
\draw (12,0) node {\includegraphics[scale=0.3]{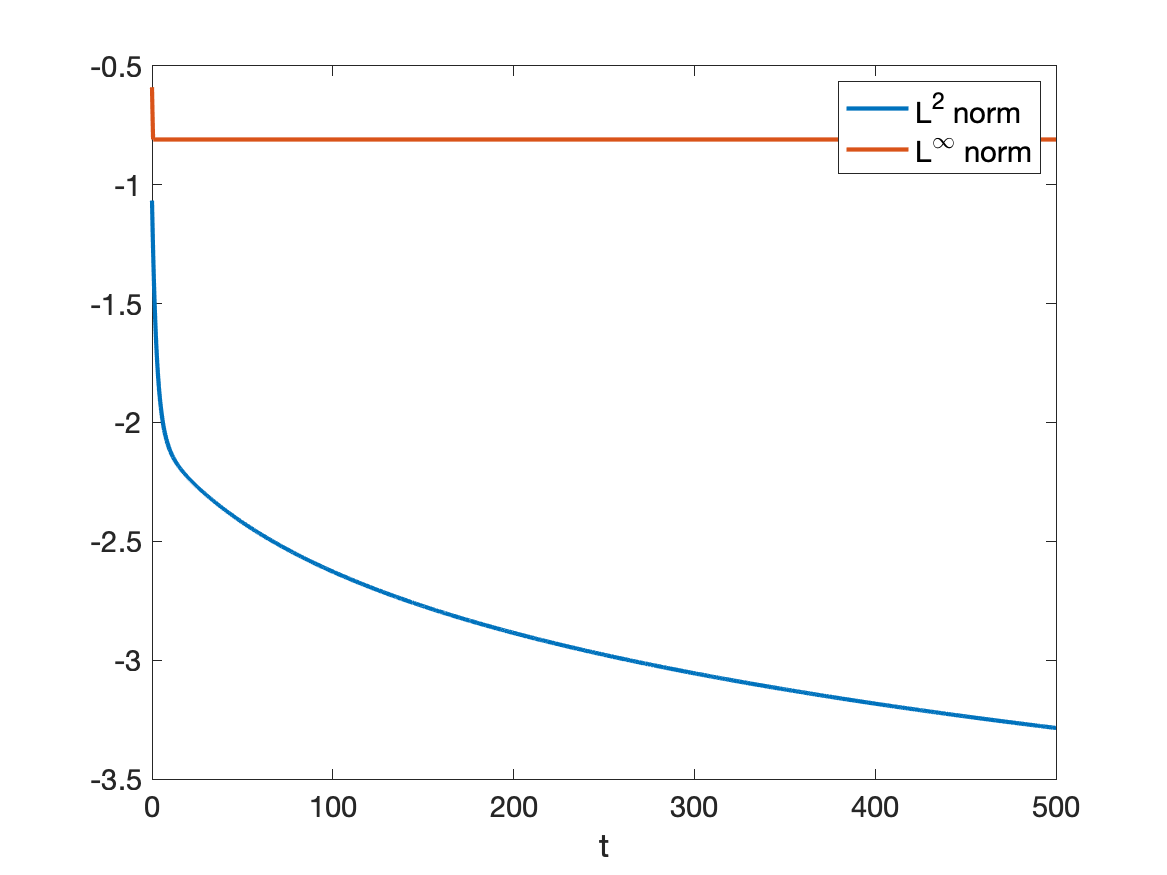}}; 
\end{tikzpicture}
\caption{{\small \textit{Snapshots of the solution $i \in I \mapsto x(t,i) \in [0,1]$ generated by the communication weights \eqref{eq:NonConsensusModel} at different instants (left) along with the time-evolution of the $L^2$- and $L^{\infty}$-distance to the consensus point in natural scale (center) and log scale (right).}}}
\label{fig:NonConsensus}
\end{figure}

The plots displayed in Figure \ref{fig:NonConsensus} are coherent with what was expected, in the sense that the $L^{\infty}$-norm remains bounded from below by a constant, while a non-exponential $L^2$-consensus formation seems to arise. To understand why the latter appears -- at least numerically --  even though $\lambda_2(\Lbb) = 0$, recall that the semi-discretisation of the graphon dynamics is given for any chosen integer $N \geq 1$ by 
\begin{equation}
\label{eq:DiscretisedBalancedDyn}
\dot{x}_i(t) = \frac{1}{N} \sum_{j=1}^N a_{ij}(x_j(t) - x_i(t)), \qquad x_i(0) =x^0_i, 
\end{equation}
for all times $t \geq 0$ and every $i \in \{1,\dots,N\}$, where 
\begin{equation}
\label{eq:DiscretisedBalanced}
a_{ij} := N^2 \INTSeg{\INTSeg{a(k,l)}{l}{(2j-1)/2N}{(2j+1)/2N}}{k}{(2i-1)/2N}{(2i+1)/2N} \qquad \text{and} \qquad x^0_i := \INTSeg{x^0(k)}{k}{(2i-1)/2N}{(2i+1)/2N}, 
\end{equation}
for every $i,j \in \{1,\dots,N\}$. The key point here is to remark that the adjacency matrices $\Ab_N := (a_{ij})_{1 \leq i,j \leq N} \in [0,1]^N$ -- that are represented for different values of $N$ in Figure \ref{fig:Matrices} -- define strongly connected interaction topologies with $\lambda_2(\Lb_N) > 0$ and $\min_{i \in \{1,\dots,N\}} \tfrac{1}{N} \sum_{j=1}^N a_{ij} > 0$. Therefore, at the discrete level, we always observe the formation of $\ell^2$- and $\ell^{\infty}$-consensus as illustrated in Figure \ref{fig:NonConsensusDiscrete}.

\begin{figure}[!ht]
\hspace{-0.6cm}
\begin{tikzpicture}
\draw (0,0) node {\includegraphics[scale=0.15]{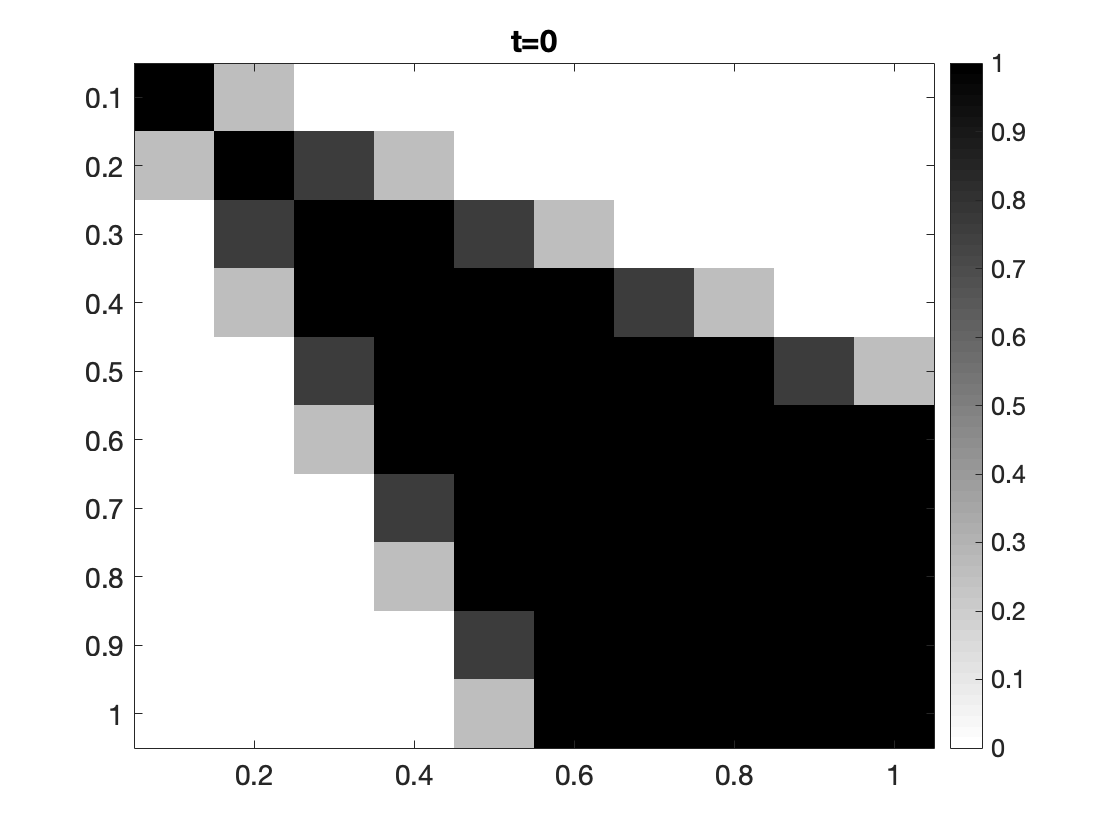}};
\draw (6,0) node {\includegraphics[scale=0.15]{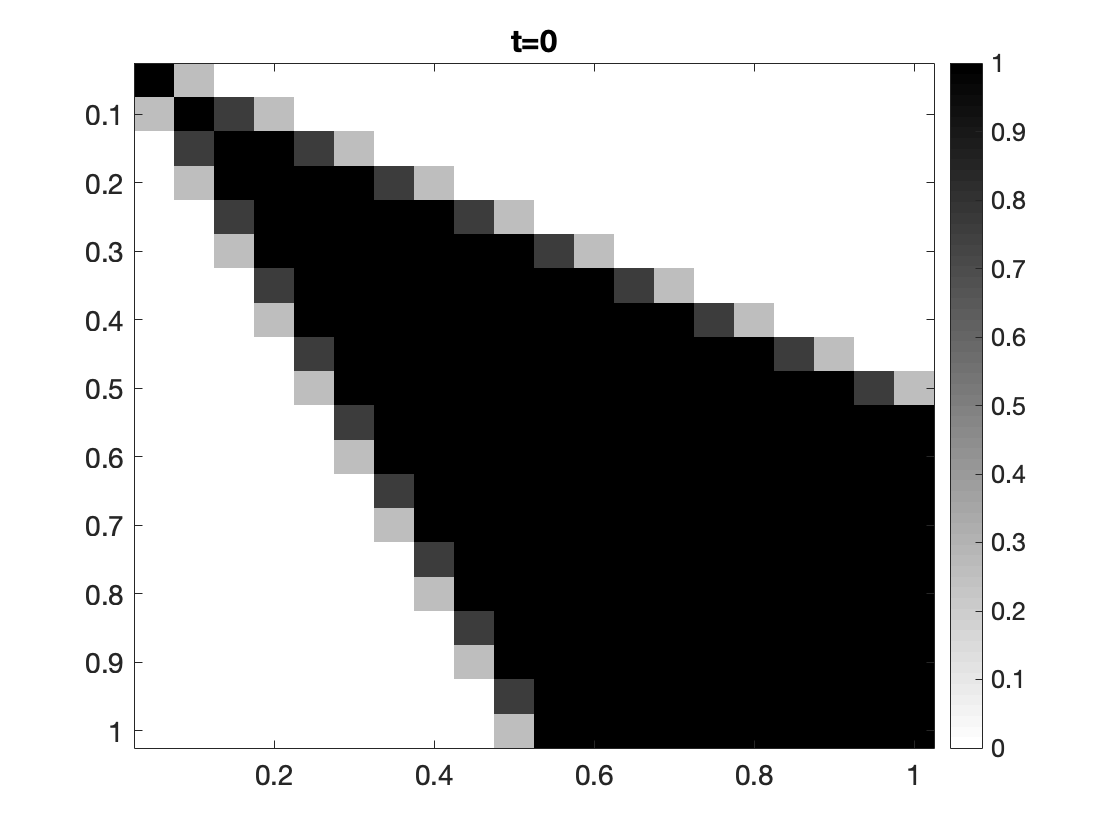}};
\draw (12,0) node {\includegraphics[scale=0.15]{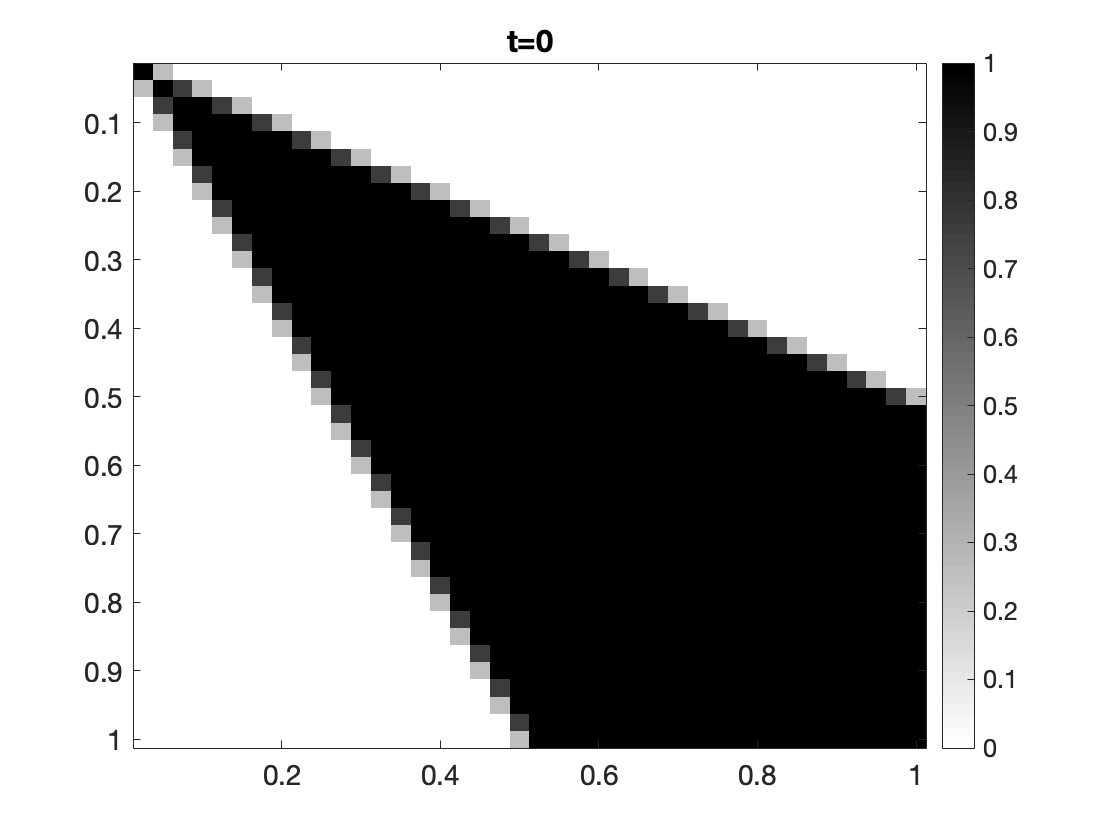}}; 
\end{tikzpicture}
\caption{{\small \textit{Representation of the interaction matrices $\Ab_N := (a_{ij})_{1 \leq i,j \leq N}$ whose coefficients are defined by \eqref{eq:DiscretisedBalanced} for $N=10$ (left), $N=20$ (center) and $N=40$ (right).}}}
\label{fig:Matrices}
\end{figure}

\begin{figure}[!ht]
\hspace{-0.6cm}
\begin{tikzpicture}
\draw (0,0) node {\includegraphics[scale=0.3]{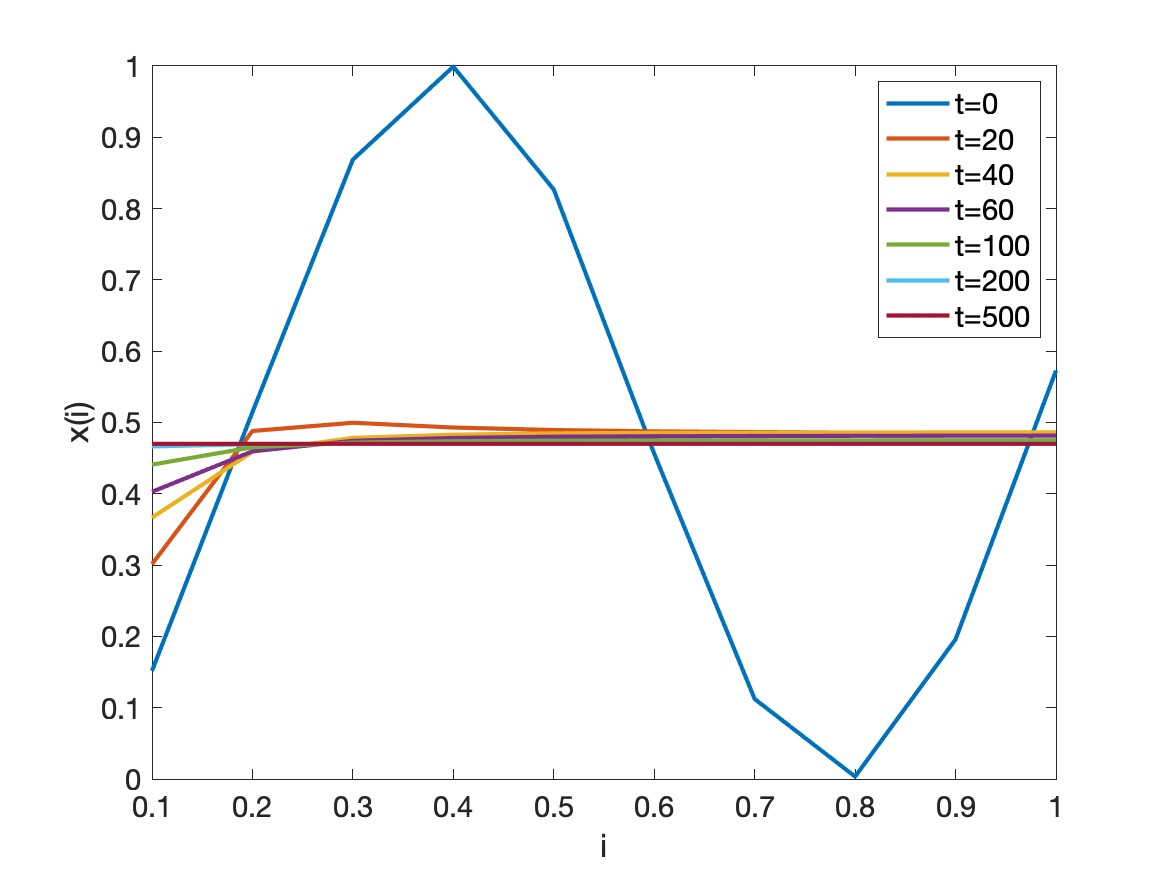}};
\draw (6,0) node {\includegraphics[scale=0.3]{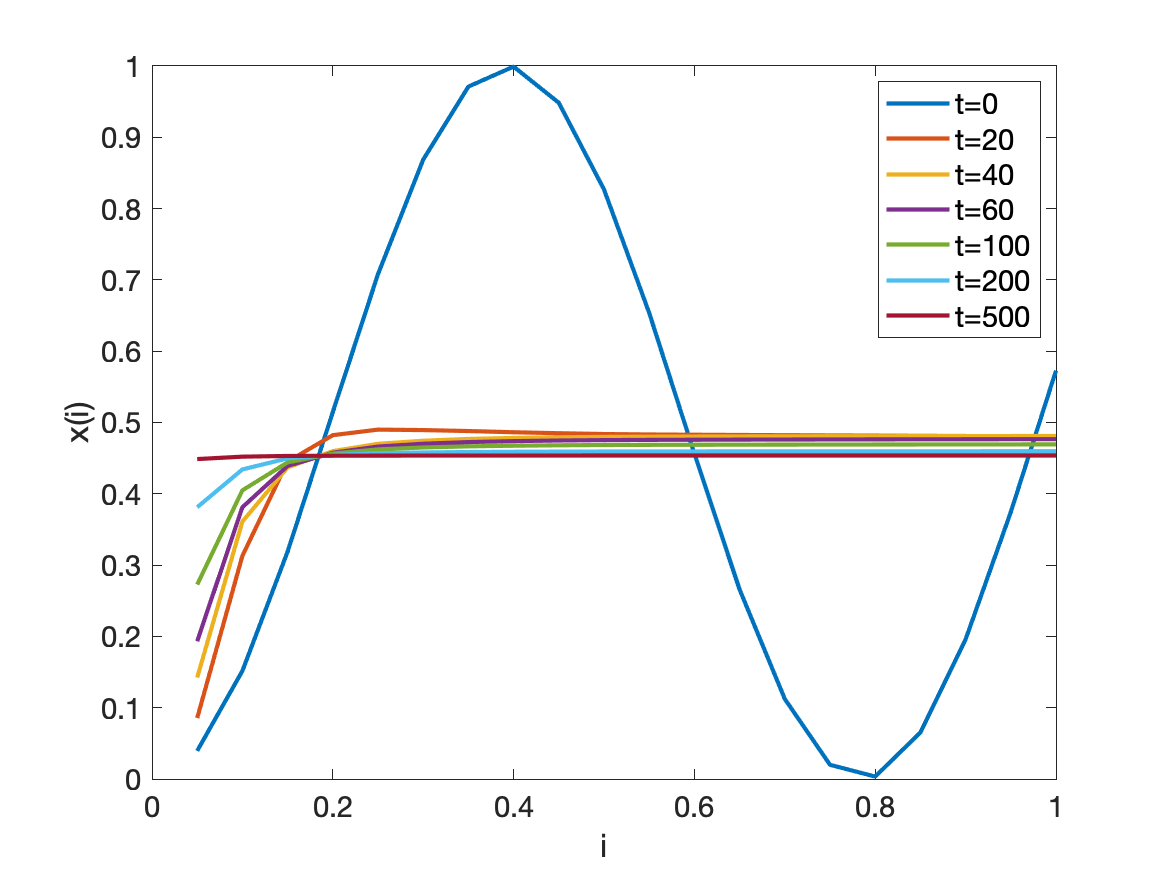}};
\draw (12,0) node {\includegraphics[scale=0.3]{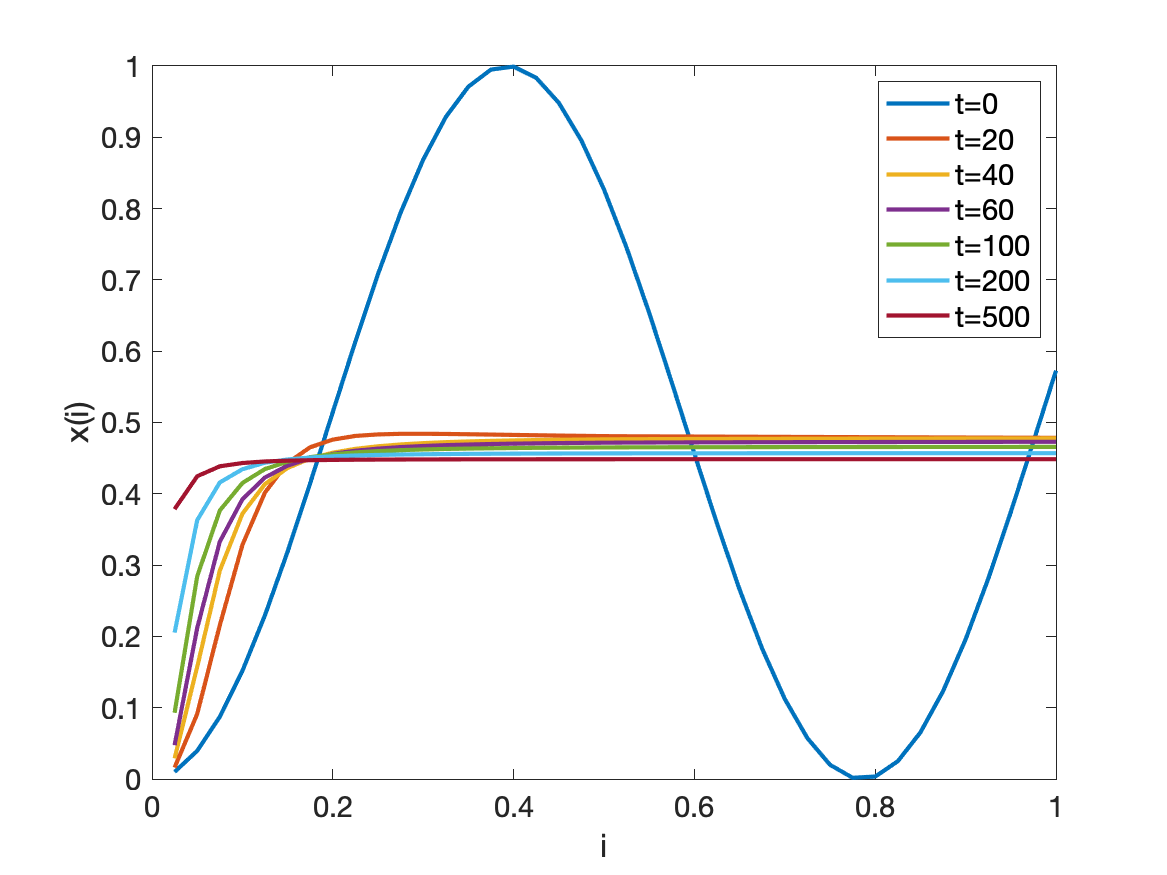}}; 
\end{tikzpicture}
\caption{{\small \textit{Snapshots of the solution $i \in I \mapsto x(t,i) \in [0,1]$ generated by the communication weights \eqref{eq:NonConsensusModel} at different instants for $N=10$ (left), $N=20$ (center) and $N=40$ (right).}}}
\label{fig:NonConsensusDiscrete}
\end{figure}

The fact that the exponential convergence to consensus of the sequence of approximating microscopic problems is not a stable property as $N \rightarrow +\infty$ can be further underpined by the following observation. As amply discussed in a more general context in \cite{VonLuxburg2008}, only the discrete part of the spectrum of a graph-Laplacian operator can be faithfully described by the limits of eigenvalues of discrete approximating graph-Laplacian matrices. In our example, the essential spectrum of the operator is $\sigma_{\ess}(\Lbb) = \rg(d) = [0,1]$, and the sequence $(\lambda_2(\Lb_N))_{N \geq 1}$ of algebraic connectivities of the discrete models goes to $0$ as $N \rightarrow +\infty$. This induces a depreciation of the convergence rates as $N$ gets larger, which is illustrated for both norms in Figure \ref{fig:ConvergenceRates}. 

\begin{figure}[!ht]
\centering
\begin{tikzpicture}
\draw (0,0) node {\includegraphics[scale=0.15]{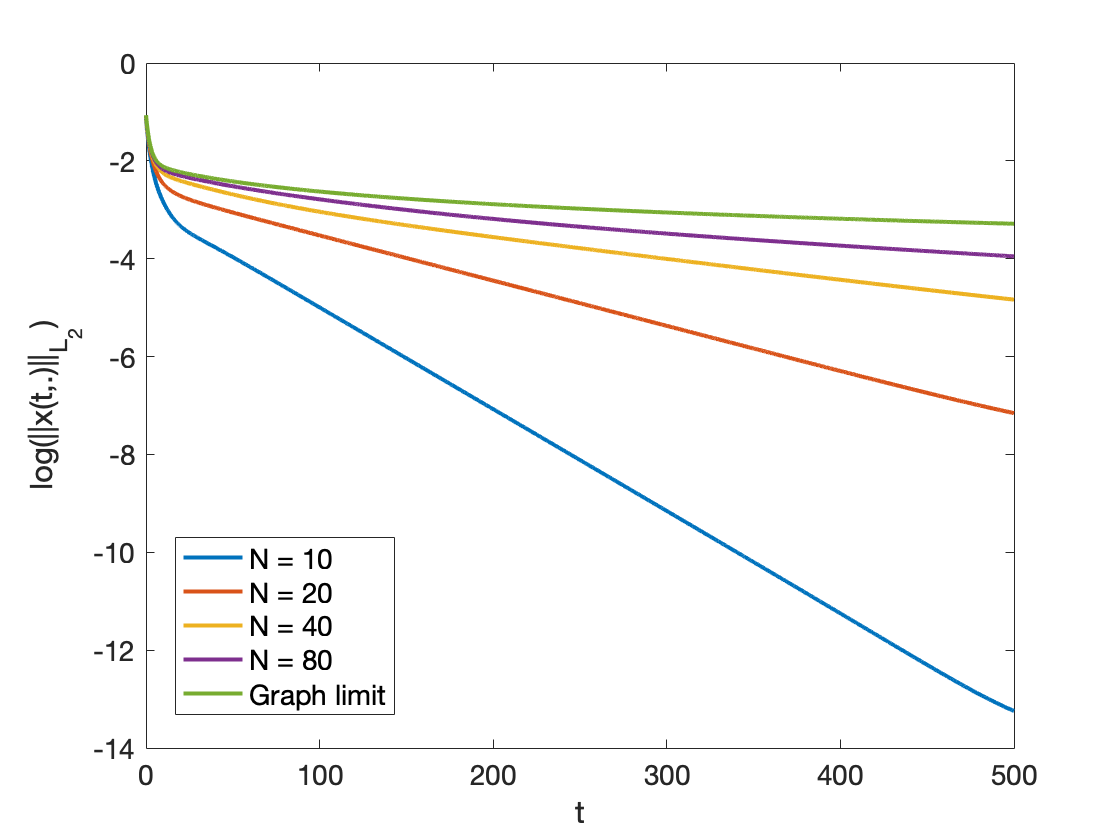}};
\draw (6,0) node {\includegraphics[scale=0.15]{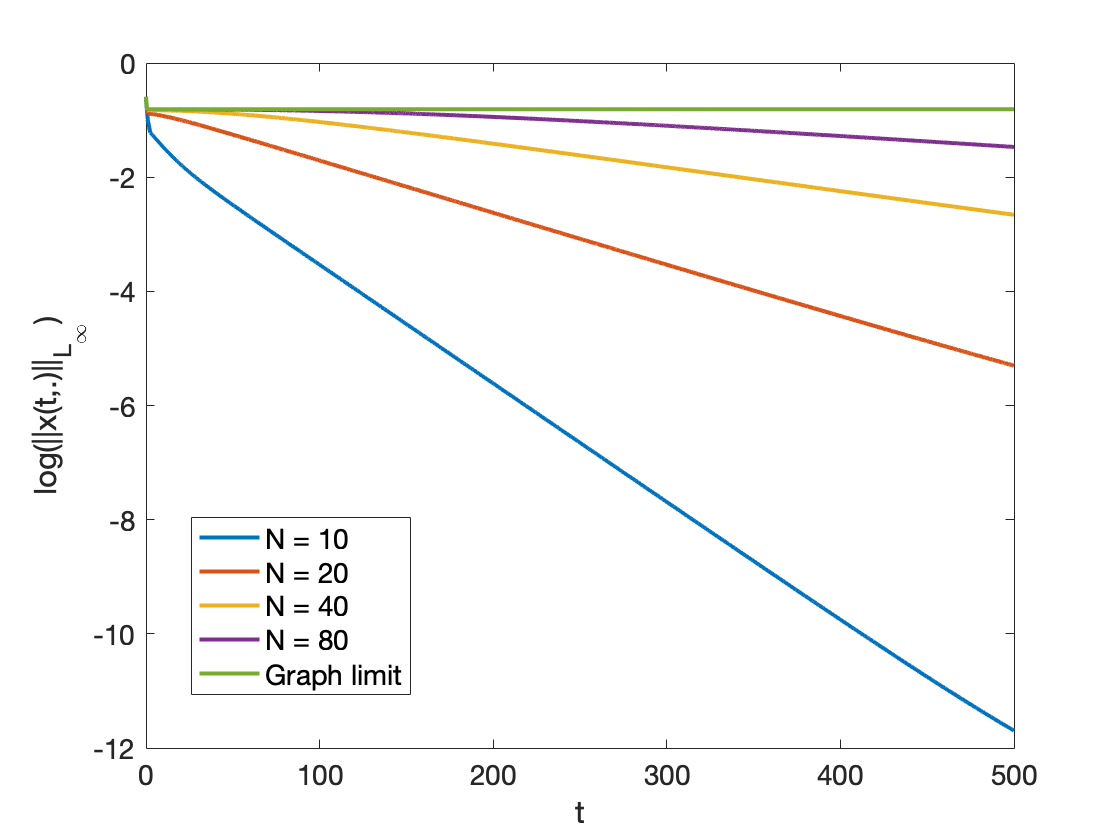}};
\end{tikzpicture}
\caption{{\small \textit{Evolution of the long-time convergence rates to consensus of solutions the discrete systems \eqref{eq:DiscretisedBalancedDyn} with respect to the $L^2$-norm (left) and the $L^{\infty}$-norm (right) as $N$ increases.}}}
\label{fig:ConvergenceRates}
\end{figure}

%%%%%%%%%%%%%%%%%%%%%%%%%%%%%%%%%%%%%%%%%%%%%%%%%%%%%%%%%%%%%%%%%%%%%%%%%%%%%%%%
%							ACKNOWLEDGMENTS AHEAD							   %
%%%%%%%%%%%%%%%%%%%%%%%%%%%%%%%%%%%%%%%%%%%%%%%%%%%%%%%%%%%%%%%%%%%%%%%%%%%%%%%%

\smallskip

\begin{flushleft}
{\small{\bf  Acknowledgement.}  The authors wish to thank Christophe Hazard and Emmanuel Trélat for their useful insights and bibliographical suggestions.}
\end{flushleft}

%%%%%%%%%%%%%%%%%%%%%%%%%%%%%%%%%%%%%%%%%%%%%%%%%%%%%%%%%%%%%%%%%%%%%%%%%%%%								BIBLIOGRPAHY AHEAD							   %
%%%%%%%%%%%%%%%%%%%%%%%%%%%%%%%%%%%%%%%%%%%%%%%%%%%%%%%%%%%%%%%%%%%%%%%%%%%%

\bibliographystyle{plain}
{\footnotesize
\bibliography{../../ControlWassersteinBib}
}

\end{document}